\documentclass{amsart}
\usepackage{amssymb}
\usepackage{amsmath}
\usepackage{mathrsfs}                  
\usepackage[dvips,ps,all,color]{xy}    
\UseCrayolaColors

\title%
[The ideal of relations for the ring of invariants of $n$ points on the line]
{The ideal of relations for the ring of invariants\\of $n$ points on
  the line}

\author[Benjamin Howard, John Millson, Andrew Snowden and Ravi Vakil]
{Benjamin Howard, John Millson, Andrew Snowden and Ravi Vakil*}

\thanks{*B.~Howard was supported by NSF fellowship DMS-0703674.
J.~Millson was supported by the NSF grant DMS-0405606, the NSF FRG grant
DMS-0554254 and the Simons Foundation.
This research was partially conducted during the period A.~Snowden was
employed by the Clay Mathematics Institute as a Liftoff Fellow.  A.~Snowden
was also partially supported by NSF fellowship DMS-0902661.
R.~Vakil was partially supported by NSF grant DMS-0801196.}

\date{September 17, 2009.}

\newcommand{\cut}[1]{}

\newcommand{\Ncom}[1]{}
\newcommand{\Rnts}[1]{}

\newtheorem{theorem}{Theorem}[section]
\newtheorem{lemma}[theorem]{Lemma}
\newtheorem{corollary}[theorem]{Corollary}
\newtheorem{proposition}[theorem]{Proposition}

\theoremstyle{definition}

\theoremstyle{remark}
\newtheorem{remark}[theorem]{\it Remark}

\newcommand{\retro}{\mathrm{retro}}

\newcommand{\bwg}[1]{{\bigwedge}^{\! #1}\,}
\newcommand{\bw}{\bwg{2}}

\DeclareMathOperator{\Proj}{Proj}
\DeclareMathOperator{\Sym}{Sym}
\DeclareMathOperator{\lev}{lev}
\DeclareMathOperator{\gr}{gr}

\DeclareMathOperator{\bS}{\mathbf{S}}
\DeclareMathOperator{\Ind}{Ind}
\DeclareMathOperator{\Hom}{Hom}

\DeclareMathOperator{\Aut}{Aut}
\DeclareMathOperator{\sgn}{sgn}
\DeclareMathOperator{\Rel}{Rel}

\def\SL{\mathrm{SL}}
\def\Sp{\mathrm{Sp}}
\def\SO{\mathrm{SO}}
\def\Orth{\mathrm{O}}

\def\Z{\mathbb{Z}}

\def\P{\mathbb{P}}
\def\C{\mathbb{C}}
\def\Q{\mathbb{Q}}

\def\half{\tfrac{1}{2}}

\def\un{\mathrm{un}}
\def\wt{\mathbf{wt}}
\def\cq{/\!/}

\let\ms\mathscr
\let\mf\mathfrak
\let\mc\mathcal
\let\ss\scriptstyle
\let\wt\widetilde
\let\ol\overline
\let\ul\underline

\newcommand\dedge[1]{\overrightarrow{#1\rule{0em}{.7em}}}
\newcommand\uedge[1]{\overline{#1\rule{0em}{.7em}}}
\newcommand\dedgescr[1]{\overrightarrow{#1\rule{0em}{.5em}}}

\makeatletter

\@addtoreset{equation}{section}
\makeatother

\begin{document}


\begin{abstract}
The study of the projective coordinate ring of the (geometric invariant theory)
moduli space of $n$ ordered points on $\P^1$ up to automorphisms began with
Kempe in 1894, who proved that the ring is generated in degree one in the main
($n$ even, unit weight) case.  We describe the relations among the invariants
for all possible weights.  In the main case, we show that up to the
$\mf{S}_n$-symmetry, there is a single equation.  For $n \neq 6$, it is a
simple quadratic binomial relation.  (For $n=6$, it is the classical Segre cubic
relation.)  For  general weights, the ideal of relations is generated by
quadratics inherited from the case of $8$ points.  This paper completes the
program set out in \cite{HMSV}.
\end{abstract}

\maketitle

\tableofcontents




\section{Introduction}
\label{s:intro}

We consider the ring of invariants of $n$ points on the projective line,
defined as the projective coordinate ring of the GIT quotient of $(\P^1)^n$ by
the group $\SL(2)$.  This is a classical archetype of GIT and a common first
example in the theory.  To form the quotient, one must choose weights
$w=(w_1, \dots, w_n)$; it is then given by
\begin{displaymath}
(\P^1)^n \dashrightarrow (\P^1)^n \cq_w \SL(2) = \Proj{R_w}
\end{displaymath}
where $R_w=\bigoplus_{k=0}^{\infty} R_w^{(k)}$, with $R_w^{(k)}=
\Gamma((\P^1)^n, \mathcal{O}(kw_1, \dots, kw_n))^{\SL(2)}$.  (Note that we
use the notation $R^{(k)}$ for the $k$ piece of a graded ring $R$.)  The ring
of invariants $R_w$ turns out to be generated in its lowest nonzero degree, so
the GIT quotient has a natural projective embedding.  We denote the weight
$(1,1,\ldots,1)$ by $1^n$.
\Ncom{$R_w, R_w^{(k)}$}

\begin{theorem}[Main Theorem, informal version]
\label{informal}
If $w \neq 1^6$, the ideal of relations between lowest degree invariants is
generated by quadratics.
\end{theorem}

Detailed motivation and background for this problem are given in the
announcement \cite{hmsvCR}.  We describe there how small cases have long been
known to yield beautiful classical geometry.  In this paper, we show that this
rich structure extends to any number of points with any weighting: the
relations for the moduli space are generated by a particularly simple type of
quadratic, with the single exception of the Segre cubic for $6$ points.  In a
precise sense, the ideal is cut out by essentially one equation, inherited
from the $n=8$ case.  If Kempe's theorem is the analogue of Weyl's ``First Main
Theorem'' for $\SL_2$ (see \cite{W}), then this is the analogue of his ``Second Main Theorem.''  

\subsection{Integrality issues}
\label{ss:integrality}

For simplicity, we prove our results over $\Q$, but our results apply more
generally, as we show in \cite{integral}.  In preparation for this, we
prove intermediate results over more general base rings, which we hope will not
distract readers interested in characteristic $0$.  In particular, this
paper proves Theorem~\ref{informal} over $\Z[\tfrac{1}{12!}]$, but the precise
version (Theorem~\ref{mainthm}) only over $\Q$.

\subsection{The graphical formalism}
\label{ss:gf}

We use a graphical interpretation of the ring of invariants, which allows us to
deal effectively with both the $\mf{S}_n$-symmetries and the broken symmetries
of toric degenerations. To a directed graph $\Gamma$ on vertices labeled $1$
through $n$ (in bijection with the points), with valence vector $kw=(kw_1,
\ldots, kw_n)$, we associate an invariant in $R_w^{(k)}$:
\begin{displaymath}
X_{\Gamma} = \prod_{\dedgescr{ij}} (x_i y_j - x_j y_i).
\end{displaymath}
Here the product is taken over the edges $\dedge{ij}$ of $\Gamma$ and $x_i$
and $y_i$ are projective coordinates for $\P^1$.  Note that if $\Gamma$ has a
\emph{loop} --- that is, an edge with the same source and target
--- then $X_{\Gamma}=0$.  (Throughout this paper, graphs are allowed to have
loops and multiple edges between the same vertices.)  A fundamental result from
classical invariant theory states that the $X_{\Gamma}$ span $R_w^{(k)}$.
Kempe showed that for $w=1^n$ and $n$ even the ring $R_w$ is generated in
degree 1 (Theorem~\ref{thm:kempe}).  Thus, in this situation, $R_w$ is
generated by those $X_{\Gamma}$ where $\Gamma$ is a \emph{matching}, that is,
a graph in which each vertex belongs to precisely one edge.  A similar result holds for any
$n$ and $w$.
\Ncom{$X_{\Gamma}, \text{loop, matching}$}

\subsection{Relations}

The theorem of Kempe mentioned above provides generators for the ring $R_w$.
The purpose of this paper is to determine the relations between these
generators.  A number of obvious relations exist; we catalog some of them
below.  An important phenomenon shows itself already in these simple
examples:  relations on $n$ points can be extended to give relations on
more than $n$ points.  This is a key theme in our treatment and is discussed
in detail in \S \ref{ss:outer}.

The most obvious relation is the \emph{sign relation}:  we have
$X_{\dedgescr{ab}}=-X_{\dedgescr{ba}}$.  This can be regarded as a relation
between two invariants with $n=2$ and $w=1^2$.  It extends to invariants with
arbitrary $n$ and $w$ as
follows:  if $\Gamma$ is any graph and $\Gamma'$ is obtained from $\Gamma$ by
switching the direction of a single edge then $X_{\Gamma}=-X_{\Gamma'}$.

The next most simple relation appears when $n=4$ and $w=1^4$:
\Ncom{$\text{s., Pl., S.c. r.}$}
\begin{displaymath}
\begin{xy}
(0, -4)*{}="A"; (0, 4)*{}="B"; (8, -4)*{}="C"; (8, 4)*{}="D";
{\ar "A"; "D"}; {\ar "C"; "B"};
\end{xy}
\qquad = \qquad
\begin{xy}
(0, -4)*{}="A"; (0, 4)*{}="B"; (8, -4)*{}="C"; (8, 4)*{}="D";
{\ar "A"; "B"}; {\ar "C"; "D"};
\end{xy}
\qquad + \qquad
\begin{xy}
(0, -4)*{}="A"; (0, 4)*{}="B"; (8, -4)*{}="C"; (8, 4)*{}="D";
{\ar "C"; "A"}; {\ar "B"; "D"};
\end{xy}
\end{displaymath}
This is the classical \emph{Pl\"ucker relation}.  It extends to each
$R^{(k)}_w$ as well:  given a graph $\Gamma$ with valence vector $kw$ and two
edges $\dedge{ab}$ and $\dedge{cd}$ of $\Gamma$ we have the identity
\begin{equation}
X_{\Gamma}=X_{\Gamma_1}+X_{\Gamma_2}
\end{equation}
where $\Gamma_1$ and $\Gamma_2$ are the graphs obtained from $\Gamma$ by
replacing $\{ \dedge{ab},\dedge{cd} \}$ with $\{ \dedge{ad}, \dedge{cb} \}$ and
$\{ \dedge{ac}, \dedge{bd} \}$ respectively.  We will use the phrase ``to
Pl\"ucker two edges $\dedge{ab}$ and $\dedge{bc}$ of a graph $\Gamma$''  to
mean ``to replace $X_{\Gamma}$ by $X_{\Gamma_1} + X_{\Gamma_2}$.''  The sign
and Pl\"ucker relations are both linear relations; in fact, they span all
linear relations.
\Ncom{$\text{Plucker (v)}$}

Some higher degree relations are conveniently thought of in terms of colored
graphs.  By a \emph{$k$-colored graph} we mean a graph in which each edge
has been assigned one of $k$ colors.  By a \emph{multi-matching} of degree $k$
we mean a $k$-colored graph in which each vertex appears in precisely one
edge of each color.  Let $\Gamma$ be a multi-matching of degree $k$ on $n$
vertices.  We define the element $X_{\Gamma}$ of $(R^{(1)}_{1^n})^{\otimes
k}$ to be the pure tensor $\bigotimes X_{\Gamma(i)}$ where the product is
over the colors $i$ and $\Gamma(i)$ is the subgraph of $\Gamma$ on the
edges of color $i$.  In terms of colored graphs, the map
$(R_{1^n}^{(1)})^{\otimes k} \to R_{1^n}^{(k)}$ ``forgets the color.''
\Ncom{$\text{colored graph}$}

The \emph{Segre cubic relation} is described with colored graphs as follows:
\begin{equation}
\begin{xy}
(-4, 6.93)*{}="A"; (4, 6.93)*{}="B"; (8, 0)*{}="C";
(4, -6.93)*{}="D"; (-4, -6.93)*{}="E"; (-8, 0)*{}="F";
{\ar@[red]@{--} "A"; "B"};
{\ar@[red]@{--} "C"; "F"};
{\ar@[red]@{--} "D"; "E"};
{\ar@[blue]@{..}@[|(2)] "A"; "D"};
{\ar@[blue]@{..}@[|(2)] "B"; "C"};
{\ar@[blue]@{..}@[|(2)] "F"; "E"};
{\ar@[green]@{-} "A"; "F"};
{\ar@[green]@{-} "B"; "E"};
{\ar@[green]@{-} "C"; "D"};
\end{xy}
\qquad = \qquad
\begin{xy}
(-4, 6.93)*{}="A"; (4, 6.93)*{}="B"; (8, 0)*{}="C";
(4, -6.93)*{}="D"; (-4, -6.93)*{}="E"; (-8, 0)*{}="F";
{\ar@[red]@{--} "C"; "F"};
{\ar@[red]@{--} "A"; "D"};
{\ar@[red]@{--} "B"; "E"};
{\ar@[blue]@{..}@[|(2)] "A"; "B"};
{\ar@[blue]@{..}@[|(2)] "C"; "D"};
{\ar@[blue]@{..}@[|(2)] "F"; "E"};
{\ar@[green]@{-} "A"; "F"};
{\ar@[green]@{-} "B"; "C"};
{\ar@[green]@{-} "D"; "E"};
\end{xy}
\label{e:Segre}
\end{equation}
(For those readers for whom color is not available, we will follow the
convention that the solid lines are green, the dashed red and the dotted
blue.)  Each edge should be directed in the same way on both sides of the
equation.  This relation holds because the graphs on each side have the same
set of edges --- only the colors are different.  The Segre cubic relation
extends to cubic relations on $R^{(1)}_{1^n}$ for any even $n>6$.  For
example, we have the following relation on eight points starting from the
Segre cubic on six points:
\begin{equation}
\label{Segre8}
\begin{xy}
(-4, 9.93)*{}="A"; (4, 9.93)*{}="B"; (8, 3)*{}="C";
(4, -3.93)*{}="D"; (-4, -3.93)*{}="E"; (-8, 3)*{}="F";
(-4, -9.93)*{}="G"; (4, -9.93)*{}="H";
{\ar@[red]@{--} "A"; "B"};
{\ar@[red]@{--} "C"; "F"};
{\ar@[red]@{--} "D"; "E"};
{\ar@[blue]@{..}@[|(2)] "A"; "D"};
{\ar@[blue]@{..}@[|(2)] "B"; "C"};
{\ar@[blue]@{..}@[|(2)] "F"; "E"};
{\ar@[green]@{-} "A"; "F"};
{\ar@[green]@{-} "B"; "E"};
{\ar@[green]@{-} "C"; "D"};
{\ar@<.4ex>@[red]@{--} "G"; "H"};
{\ar@[green]@{-} "G"; "H"};
{\ar@<-.4ex>@[blue]@{..}@[|(2)] "G"; "H"};
\end{xy}
\qquad = \qquad
\begin{xy}
(-4, 9.93)*{}="A"; (4, 9.93)*{}="B"; (8, 3)*{}="C";
(4, -3.93)*{}="D"; (-4, -3.93)*{}="E"; (-8, 3)*{}="F";
(-4, -9.93)*{}="G"; (4, -9.93)*{}="H";
{\ar@[red]@{--} "C"; "F"};
{\ar@[red]@{--} "A"; "D"};
{\ar@[red]@{--} "B"; "E"};
{\ar@[blue]@{..}@[|(2)] "A"; "B"};
{\ar@[blue]@{..}@[|(2)] "C"; "D"};
{\ar@[blue]@{..}@[|(2)] "F"; "E"};
{\ar@[green]@{-} "A"; "F"};
{\ar@[green]@{-} "B"; "C"};
{\ar@[green]@{-} "D"; "E"};
{\ar@<.4ex>@[red]@{--} "G"; "H"};
{\ar@[green]@{-} "G"; "H"};
{\ar@<-.4ex>@[blue]@{..}@[|(2)] "G"; "H"};
\end{xy}
\end{equation}

There is a ``new'' relation for $n=8$, binomial and quadratic:
\begin{equation}\label{e:simplest}
\begin{xy}
(0, 4)*{}="A"; (8, 4)*{}="B"; (16, 4)*{}="C"; (24, 4)*{}="D";
(0, -4)*{}="E"; (8, -4)*{}="F"; (16, -4)*{}="G"; (24, -4)*{}="H";
{\ar@[red]@{--} "A"; "B"};
{\ar@[red]@{--} "E"; "F"};
{\ar@[green]@{-} "A"; "E"};
{\ar@[green]@{-} "B"; "F"};
{\ar@[red]@{--} "C"; "D"};
{\ar@[red]@{--} "G"; "H"};
{\ar@[green]@{-} "C"; "G"};
{\ar@[green]@{-} "D"; "H"};
\end{xy}
\qquad = \qquad
\begin{xy}
(0, 4)*{}="A"; (8, 4)*{}="B"; (16, 4)*{}="C"; (24, 4)*{}="D";
(0, -4)*{}="E"; (8, -4)*{}="F"; (16, -4)*{}="G"; (24, -4)*{}="H";
{\ar@[red]@{--} "A"; "B"};
{\ar@[red]@{--} "E"; "F"};
{\ar@[green]@{-} "A"; "E"};
{\ar@[green]@{-} "B"; "F"};
{\ar@[green]@{-} "C"; "D"};
{\ar@[green]@{-} "G"; "H"};
{\ar@[red]@{--} "C"; "G"};
{\ar@[red]@{--} "D"; "H"};
\end{xy}
\end{equation}
As with all previously discussed relations, the above relation extends to
relations on more points.  One way to extend the above relation is to add some
number of doubled edges to each side; we call such relations the \emph{simplest
binomial relations}.  Examples are given in \cite[\S 4]{hmsvCR}.
\Ncom{$\text{simpl. bin. r.}$}

\subsection{The main theorem}

We now state Theorem~\ref{informal} more precisely in the main case
of $n$ even and unit weights.

\begin{theorem}[Main theorem, main case]
\label{mainthm}
For even $n \ne 6$ the ideal $I_{1^n}$ of relations (the kernel of
$\Sym{R_{1^n}^{(1)}} \twoheadrightarrow R_{1^n}$) is generated by the simplest
binomial relations.  The symmetric group $\mf{S}_n$ acts transitively on these
relations, so any one of them generates $I_{1^n}$ as an $\mf{S}_n$-ideal. 
\end{theorem}

(The ideal $I_{1^6}$ is principal and generated by the Segre cubic relation
\eqref{e:Segre} over $\Z$.) 

\begin{remark}
\label{formal}
The discussion of \cite[\S 2.17]{HMSV} explains how to reduce the case of
arbitrary weight to the ``main case.''  Thus as a corollary we have
Theorem~\ref{informal}, and more precisely, the quadratics are explicitly given
by ``clumping vertices'' (loc.\ cit.).  Thus for the remainder of the paper,
we will deal only with this ``main case'' of $w=1^n$ and $n$ even.
\end{remark}

We essentially conjectured Theorems~\ref{informal} and~\ref{mainthm} in
\cite[\S 1.5]{HMSV}.  We saw the two main theorems of that paper as evidence:
first, that a class of relations called the ``simple binomial relations,''
containing the simplest binomial relations, cuts out the quotient
scheme-theoretically, and second, that the ideal is generated by relations of
degree at most four.  These two results are subsumed by Theorem~\ref{mainthm}
in the main case, and Theorem~\ref{informal} (in its more precise form,
Remark~\ref{formal}) in general.

\begin{remark}
Given that Theorem~\ref{mainthm} states that $I_{1^n}$, for even $n \ne 6$, is
generated by a single quadratic up to $\mf{S}_n$-symmetry, one may wonder
whether the same holds when $n$ is odd.  A short representation-theoretic
argument shows that if $n$ is 5, 7 or 9 then this is indeed the case (see
\cite[Fig.~ 3(h)]{hmsvCR} for a simple generator with $n=5$), but if $n$ is
odd and at least 11 then the space of quadratic relations is not a cyclic
$\mf{S}_n$-module, so $I_{1^n}$ is not principal as an $\mf{S}_n$-ideal.
\end{remark}

\subsection{Representation-theoretic description}
\label{repthy}

Here is a representation-theoretic description of the quadratics (in the
main case) in characteristic $0$ that is both striking and relevant.  The
$\mf{S}_n$-representation on $R^{(1)}_{1^n}$ is irreducible and corresponds to
the partition $n/2+n/2$.  The representation $\Sym^2{R^{(1)}_{1^n}}$ is
multiplicity free and contains those irreducibles corresponding to partitions
with at most four parts, all even (Proposition \ref{prop:pfil2rep}).  The
space of quadratic relations is the subspace of $\Sym^2{R^{(1)}_{1^n}}$
spanned by those irreducibles corresponding to partitions with
\emph{precisely} four parts.  Being multiplicity-free, this is necessarily a
cyclic $\mf{S}_n$-module.  The reader may wonder why we privilege a
particular generator; the answer is that this relation is in some imprecise
sense forced upon us by the graphical formalism.

\subsection{Outline of proof}
\label{outline}

We now outline the proof, noting where the arguments are ad hoc or less
satisfactory.  The challenge is to relate three structures which often operate
at cross purposes: the generation of new relations from relations on fewer
points; the action of $\mf{S}_n$ on everything; and the graphical description
of the algebra, including the use of colored graphs to describe relations.

In \S \ref{s:ring}, we set the stage by giving our preferred description of
the invariant ring.  We replace the integer $n$ by a finite set $L$ of
cardinality $n$, as this makes many constructions more transparent.  In
\S \ref{s:gentor}--\ref{s:ytree} we use a Speyer-Sturmfels toric degeneration
to get some control on the degrees and types of generators, temporarily
breaking the $\mf{S}_L$-symmetry.  We show that the degenerated ring is
generated in degree one and that the relations between the degree one
generators are generated by quadratic relations and certain explicit cubic
relations.  In \S \ref{s:gsc} we lift these explicit cubic relations to the
original invariant ring.  Having deduced that the ideal of relations is
generated by quadratics and these particular (``small generalized Segre'')
cubics, we are done with the toric degeneration.  Our next goal is to show
that the particular cubics lie in the ideal $Q_L$ generated by quadratic
relations.

In order to take advantage of the $\mf{S}_L$-action, in \S \ref{s:symvl} we
study the tensor powers of the degree one invariants as representations,
introducing a useful ``partition filtration.''  Our results will (for
example) allow us to write invariants and relations in terms of highly
disconnected graphs, which is the key to our later inductive arguments.
The last result of this section is disappointing:  it is the only place in the
article where computer calculation is used.  However, the calculations are
quite mild --- they concern cubic invariants on six points and amount to 
simple linear algebra problems in vector spaces of dimension at most 35
--- and we feel that a dedicated human being could perform them in a matter
of hours.

In \S \ref{s:retro}, we show that for $n \ge 10$ all relations are induced
from those on fewer points (precisely, $n-2$, $n-4$ or $n-6$ points), modulo
quadratic relations.  The cases $n \ge 12$ are direct and structural, but the
case $n=10$ is ad hoc and inelegant because ``the graphs are too small'' to
apply the structural techniques.  As a consequence, we find that if the
ideal is generated by quadratics for $n-2$, $n-4$ and $n-6$ points then it
is for $n$ points as well.

In \S \ref{s:base}, we show that the ideal is generated by quadratics
when $n \neq 6$.  This implies Theorem~\ref{informal} by Remark~\ref{formal}.
Thanks to the previous section, showing generation by quadratics reduces by
induction to showing the result for the three ``base cases'' where $n$ is
8, 10 and 12.
We accomplish this by further reducing the 10 and 12 point
cases to the 8 point case and then appealing to earlier work for the 8
point case.  The reduction from the 10 and 12 point case to the 8 point
case is one of the most difficult and least conceptual parts of the paper, so
the reader may wish to skip this section on a first reading.  One might
hope that these results, being finite computations, could be relegated to a
computer, but the computations are large enough so that this is not possible
with current technology using naive algorithms.

Finally, in \S \ref{s:quad} we show that the quadratic relations are
spanned by the simplest binomial relations using the representation theory
of $\mf{S}_L$, completing the proof of Theorem~\ref{mainthm}.

\subsection{The projective coordinate ring of $X^n \cq G$}

\label{andrew}The third author has observed that many of the formal concepts in this paper
--- such as outer multiplication and the simple binomial relations --- in
fact apply to the study of the projective coordinate ring of $X^n \cq G$ for
any projective variety $X$ with an action of a group $G$.  He has constructed a
formalism for dealing with the resulting structures and formulated a few
general finiteness conjectures.  One of the more surprising realizations is
that the graphical formalism discussed above applies to any variety, in a
certain sense.  This general point of view may even shed more light on the
present case:  in this formalism, Theorem~\ref{mainthm} can be reinterpreted
as stating that a certain ``ring,'' made up of the rings $R_{1^n}$ with
varying $n$, is finitely presented.  If such finitely presented ``rings''
were coherent (a weakening of the noetherian property) then one would
immediately obtain universal degree bounds on syzygies, as we have for
relations.  This work will appear in a forthcoming paper.

\subsection{Notation and conventions}

We follow some conventions in an attempt to make the notation less onerous.
Throughout, $L$ will be a finite set.  By an \emph{even set} we mean finite set
of even cardinality.  Semigroups will be in script, e.g., $\ms{G}, \ms{S},
\ms{R}$.  Graphs are denoted by uppercase Greek letters, e.g., $\Gamma$, 
$\Delta$.  Trivalent trees will be denoted by $\Xi$.  An edge of
a directed (resp.\ undirected) graph from vertex $x$ to $y$ is
denoted $\dedge{xy}$ (resp.\ $\uedge{xy}$).
In general, $S$ (in
various fonts) will refer to constructions involving general directed graphs,
and $R$ will refer to regular graphs.  We work over $\Z$ (see \S
\ref{ss:integrality}) in general.

\subsection*{Acknowledgments.}

We thank Shrawan Kumar, Chris Manon and Lawrence O'Neil for useful
discussions.


\section{The invariant ring $R_L$}
\label{s:ring}

In \S \ref{s:ring} we define the ring of invariants $R_L$ for a finite set $L$, and give some of its properties.

\subsection{The semi-group $\ms{G}_L$ and the rings $S_L$ and $R_L$}
\label{ss:rings}

Let $L$ be a finite set.  Denote by $\ms{G}_L$ the set of directed graphs on
$L$.   Give $\ms{G}_L$ the structure of a semi-group by defining $\Gamma \cdot
\Gamma'$ to be the graph on $L$ whose edge set is the disjoint union of the
edge sets of $\Gamma$ and $\Gamma'$.  For an element $\Gamma$ of $\ms{G}_L$ we
denote the corresponding element of the semi-group algebra $\Z[\ms{G}_L]$ by
$X_{\Gamma}$.  (Readers interested in characteristic $0$ may freely replace
any occurrence of $\Z$ or $\Z[1/n]$ by $\Q$ throughout.)

For $a,b \in L$ let $\dedge{ab}$ denote the graph in $\ms{G}_L$ with a single directed edge
from $a$ to $b$.  Clearly $\ms{G}_L$ is the free commutative semi-group on the
$\dedge{ab}$, and $\Z[\ms{G}_L]$ is the polynomial ring on the
$X_{\dedgescr{ab}}$.
\Ncom{$X_{\Gamma} \in \Z[\ms{G}_L], L, \ms{G}_L, \dedge{ab}$}

Define the ring $S_L$ as the quotient of $\Z[\ms{G}_L]$ by the following three
types of relations, described in \S \ref{ss:gf}.
\begin{itemize}
\item \emph{Loop relation:} If $\Gamma$ has a loop then $X_{\Gamma}=0$.
\item \emph{Sign relation:} If $\Gamma$ is obtained from $\Gamma'$ by reversing
the direction of an edge then $X_{\Gamma}=-X_{\Gamma'}$.
\item \emph{Pl\"ucker relation:} If $a$, $b$, $c$ and $d$ are elements of $L$,
then $X_{\dedgescr{ab}} X_{\dedgescr{cd}}=X_{\dedgescr{ad}} X_{\dedgescr{cb}}
+X_{\dedgescr{ac}} X_{\dedgescr{bd} }$.
\end{itemize}
The sign relation implies the loop relation when 2 is inverted.  We still write
$X_{\Gamma}$ for the image of $X_{\Gamma}$ in the ring $S_L$.
\Ncom{$\text{s., Pl. r.}$} 

Recall that a graph $\Gamma$ on $L$ is said to be \emph{regular of degree $d$}
if each vertex of $\Gamma$ belongs to precisely $d$ edges.  Define $R_L$ to be the
subgroup of $S_L$ generated by the $X_{\Gamma}$ with $\Gamma$ regular.  Clearly
$R_L$ is a subring of $S_L$.  Grade $R_L$ by declaring $X_{\Gamma}$ to be of
degree $d$ if $\Gamma$ is regular of degree $d$.  If $|L|$ is odd, then every
regular graph on $L$ has even degree, so $R_L$ is concentrated in even degrees. 
(The ``first and second main theorems of invariant theory,'' mentioned
implicitly in  \S \ref{ss:gf}, imply that $S_L$ is the $\SL_2$-invariant part of
the Cox ring of $(\P^1)^L$ and that $R_L$ is the $\SL_2$-invariant part of
the projective coordinate ring of $(\P^1)^L$ with respect to the line bundle
$\ms{O}(1)^{\boxtimes L}$.)
\Ncom{$R_L \subset S_L=\Z[\ms{G}_L]/\text{s,P}, \text{$d$-regular}$}

These constructions are functorial in $L$:  A map of sets $\phi:L \to L'$
induces a homomorphism $\ms{G}_L \to \ms{G}_{L'}$ of semi-groups and thus a
homomorphism $\Z[\ms{G}_L] \to \Z[\ms{G}_{L'}]$ of semi-group rings.  This
ring homomorphism respects the sign and Pl\"ucker relations and so induces
a ring homomorphism $S_L \to S_{L'}$.  If the fibers of $L$ all have the
same cardinality then $R_L$ is mapped into $R_{L'}$, and we thus obtain a map
$R_L \to R_{L'}$.  As a special case, we see that $\mf{S}_L=\Aut(L)$ acts on
$R_L$.

\subsection{Translation from directed graphs to undirected graphs}
\label{ss:orient}

To avoid confusion with signs, it will often be convenient to translate from
directed graphs to undirected graphs.  Let $L$ be an even set.  We denote by
$\mc{M}_L$ the set of directed matchings on $L$.   An \emph{orientation} on $L$
is defined as a map $\epsilon:\mc{M}_L \to \{ \pm 1\}$ satisfying
$\epsilon(\sigma \Gamma)=\sgn(\sigma) \epsilon(\Gamma)$ for $\sigma \in
\mf{S}_L$ and $\Gamma \in \mc{M}_L$.  There are two orientations on $L$.  
\Ncom{$\mc{M}_L, \epsilon:\mc{M}_L \to \{\pm 1\}, \text{orientation}$}

Fix an orientation of $L$.  For an \emph{undirected} matching $\Gamma$ we put
$Y_{\Gamma}=\epsilon(\wt{\Gamma}) X_{\wt{\Gamma}}$ where $\wt{\Gamma}$ is any
directed matching with underlying undirected matching $\Gamma$.  The
$Y_{\Gamma}$ span the space $R_L^{(1)}$ and satisfy the loop relation
and the ``undirected Pl\"ucker relation''
\begin{displaymath}
Y_{\Gamma_1}+Y_{\Gamma_2}+Y_{\Gamma_3}=0,
\end{displaymath}
whenever $\Gamma_2$ and $\Gamma_3$ are obtained by modifying two edges of
$\Gamma_1$ appropriately.  We often prefer to work with the $Y_{\Gamma}$
instead of the $X_{\Gamma}$ since there are no directions to keep track of.
However, one must keep in mind
that the action of the symmetric group on the $Y_{\Gamma}$ is twisted by the
sign character from the most obvious action:  $\sigma Y_{\Gamma}=\sgn(\sigma)
Y_{\sigma \Gamma}$ for $\sigma \in \mf{S}_L$.  We often use the $Y_{\Gamma}$
without explicitly mentioning the choice of an orientation.
\Ncom{$Y_{\Gamma}$}

\subsection{Kempe's generation theorem}

The purpose of this paper is to give a presentation for the ring $R_L$.  To
do this we must first find a set of generators.  This problem was solved
(for the main case) by Kempe \cite{Kempe}.

\begin{theorem}[Kempe]
\label{thm:kempe}
Let $L$ be an even set.  Then the ring $R_L$ is generated in degree one.
Equivalently, the $Y_{\Gamma}$ (or $X_{\Gamma}$) with $\Gamma$ a matching
generate $R_L$.
\end{theorem}

\begin{proof}
We recall the proof of \cite[Theorem~2.3]{HMSV}, which is simpler than Kempe's
original proof, to motivate later arguments.  For any regular graph $\Gamma$ on
$L$, we express $Y_{\Gamma}$ as a polynomial in elements of the form
$Y_{\Gamma'}$ with $\Gamma'$ a matching.  Partition $L$ arbitrarily into
two sets of equal cardinality, one called ``positive'' and the other
``negative.''  We then have three types of edges:  positive (both vertices
positive), negative (both negative) and neutral (one positive and one
negative).  After applying the Pl\"ucker relation to a positive and a negative
edge, one is left with only neutral edges:
\begin{displaymath}
\begin{xy}
(0, 5)*{}="A"; (0, -5)*{}="B"; (10, -5)*{}="C"; (10, 5)*{}="D";
"A"; "B"; **\dir{-};
"C"; "D"; **\dir{-};
(0, 8)*{\ss +}; (0, -8)*{\ss +}; (10, -8)*{\ss -}; (10, 8)*{\ss -};
\end{xy}
\qquad = \qquad
\begin{xy}
(0, 5)*{}="A"; (0, -5)*{}="B"; (10, -5)*{}="C"; (10, 5)*{}="D";
"A"; "C"; **\dir{-};
"B"; "D"; **\dir{-};
(0, 8)*{\ss +}; (0, -8)*{\ss +}; (10, -8)*{\ss -}; (10, 8)*{\ss -};
\end{xy}
\qquad + \qquad
\begin{xy}
(0, 5)*{}="A"; (0, -5)*{}="B"; (10, -5)*{}="C"; (10, 5)*{}="D";
"A"; "D"; **\dir{-};
"B"; "C"; **\dir{-};
(0, 8)*{\ss +}; (0, -8)*{\ss +}; (10, -8)*{\ss -}; (10, 8)*{\ss -};
\end{xy}
\end{displaymath}
(We have neglected signs in the above identity.)
As $\Gamma$ is regular, it has a positive edge if and only if it has a negative
one.  Thus by repeatedly applying Pl\"ucker relations to positive and negative
edges we end up with an expression $Y_{\Gamma}=\sum \pm Y_{\Gamma_i}$ where the
$\Gamma_i$ have only neutral edges, and are hence bipartite.  Hall's marriage
theorem states that in a regular bipartite graph one can find a matching.  Thus
each $\Gamma_i$ can be factored into matchings, which completes the proof.
\end{proof}

\subsection{Kempe's basis theorem}
\label{ss:kempebasis}

Fix an embedding of $L$ into the unit circle in the plane.  We say that a
graph $\Gamma$ on $L$ is \emph{non-crossing} if no two of its edges cross
when drawn as chords.  The following well-known theorem of Kempe (also from
\cite{Kempe}) will be used in the proofs of Proposition~\ref{prop:torpres} and
Lemma~\ref{lem:iota}.
\Ncom{$\text{(non-)crossing}$}

\begin{theorem}[Kempe]
\label{thm:kempe2}
The $X_{\Gamma}$ with $\Gamma$ non-crossing span $S_L$.  The only linear
relations among these elements are the sign and loop relations.  Thus if one
chooses for each undirected loop-free non-crossing graph a direction on the
edges then the corresponding $X_{\Gamma}$ form a basis for $S_L$.  The same is
true for $R_L$ if one considers regular non-crossing graphs.
\end{theorem}

In fact, there is a procedure called the \emph{straightening algorithm} which
expresses an arbitrary $X_{\Gamma}$ in terms of the non-crossing basis.
The algorithm is simple:  take any pair of edges in $\Gamma$ which cross and
Pl\"ucker them.  The algorithm terminates because the total lengths of the
edges in each the two graphs resulting from a Pl\"ucker operation is less
than that in the original graph.  This nearly proves the theorem; for details
see \cite[Propositions\ 2.5, 2.6]{HMSV}.

\subsection{Some definitions}

We now define some notation that will be used constantly:
\begin{itemize}
\item $R_L$ is the ring of invariants, as defined above.
\item $V_L$ is the first graded piece $R_L^{(1)}$ of $R_L$; it is spanned by
the $Y_{\Gamma}$ with $\Gamma$ a matching.
\item $I_L$ is the ideal of relations, that is, the kernel of the map
$\Sym(V_L) \to R_L$.
\item $Q_L$ is the ideal of $\Sym(V_L)$ generated by $I_L^{(2)}$; it is a
sub-ideal of $I_L$.
\end{itemize}


\section{The toric degenerations $\gr_{\Xi} S_L$ and $\gr_{\Xi} R_L$}
\label{s:gentor}

In \S \ref{s:gentor}, we discuss toric degenerations of the rings $S_L$ and
$R_L$.  These were first described in \cite{SpeyerSturmfels}, and one was used
in \cite{HMSV}.  By a ``toric ring'' we mean a ring isomorphic to a
semi-group algebra, where the semi-group is the set of lattice points in a
strictly convex rational polyhedral cone; by a ``toric degeneration'' of a ring
we mean a toric ring obtained as the associated graded of a filtration on the
original ring.  The main points of \S \ref{s:gentor} are the following:
\begin{enumerate}
\item To each trivalent tree $\Xi$ with leaf set $L$ we give a toric
degeneration of the rings $S_L$ and $R_L$, denoted $\gr_{\Xi} S_L$ and
$\gr_{\Xi} R_L$ respectively. 
\Ncom{$\gr_{\Xi} S_L, \gr_{\Xi} R_L, \text{leaf set}$}
\item We give a presentation of the ring $\gr_{\Xi} S_L$.
\item We discuss the theory of weightings on trivalent trees.
\item We identify the rings $\gr_{\Xi} S_L$ and $\gr_{\Xi} R_L$ with semi-group
algebras of ``weightings.''
\end{enumerate}

\subsection{Trivalent trees}

By a \emph{trivalent tree} we mean a connected undirected graph $\Xi$, without
cycles, all of whose vertices have valence one or three.  We call vertices of
valence three \emph{trinodes}, and vertices of valence one \emph{leaves}.  We
say leaves $x$ and $y$ form a \emph{matched pair} if they share a neighbor
(trinode).  We say that a trivalent tree is \emph{matched} if it has more than
four vertices and every leaf belongs to a (necessarily unique) matched pair.  A
matched trivalent tree necessarily has an even number of leaves.
\Ncom{$\Xi, \text{trinode, matched pair/tree}$}\label{d:matched}

\begin{figure}[!ht]
\begin{displaymath}
\begin{xy}
(-17, 7)*{}="A"; (-17, -7)*{}="B";  (-10, 0)*{}="C"; (0, 0)*{}="D";
(0, 7)*{}="E"; (10, 0)*{}="F"; (17, 7)*{}="G"; (17, -7)*{}="H";
"A"; "C"; **\dir{-};
"B"; "C"; **\dir{-};
"C"; "D"; **\dir{-};
"D"; "E"; **\dir{-};
"D"; "F"; **\dir{-};
"F"; "G"; **\dir{-};
"F"; "H"; **\dir{-};
"A"*{\bullet}; "B"*{\bullet}; "C"*{\bullet}; "D"*{\bullet};
"E"*{\bullet}; "F"*{\bullet}; "G"*{\bullet}; "H"*{\bullet};
(-17, 10)*{\ss a}; (-17, -9)*{\ss b}; (-10, -2)*{\ss c}; (0, -2)*{\ss d};
(0, 10)*{\ss e}; (10, -2)*{\ss f}; (17, 10)*{\ss g}; (17, -9)*{\ss h};
\end{xy}
\end{displaymath}
\caption{A trivalent tree.  Vertices $c$, $d$ and $f$ are trinodes; the rest
are leaves.  The vertices $a$ and $b$ form a matched pair, as do the
vertices $g$ and $h$.  
The tree is not matched because $e$ does not belong to a matched pair.}
\end{figure}
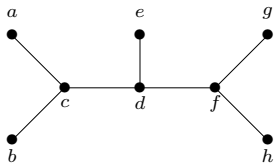

\subsection{The toric rings $\gr_{\Xi} S_L$ and $\gr_{\Xi} R_L$}

Let $\Xi$ be a trivalent tree with leaf set $L$.  For a graph $\Gamma$ on $L$,
define the \emph{level} of $\Gamma$ (relative to $\Xi$) as 
\Ncom{$\lev_{\Xi}$}
\begin{displaymath}
\lev_{\Xi}{\Gamma}=\sum_{\dedgescr{ab}} \left( \text{the distance from $a$ to
$b$ in $\Xi$} \right) 
\end{displaymath}
where the sum is over the edges of $\Gamma$, and distance is the
number of edges in the ``geodesic.''  Clearly $\lev_{\Xi}$ induces a
semi-group morphism $\lev_{\Xi}:\ms{G}_L \rightarrow \Z_{\ge 0}$.  Define an
increasing filtration $F_{\Xi}$ on $\Z[\ms{G}_L]$ by letting $F^i_{\Xi}
\Z[\ms{G}_L]$ be the
subspace of $\Z[\ms{G}_L]$ spanned by the $X_\Gamma$ with $\lev_{\Xi}{\Gamma}
\le i$. (The notation $\lev_{\Xi}$ will not be used further.)  Let $F^i_{\Xi}
S_L$ be the image of $F^i_{\Xi} \Z[\ms{G}_L]$ under the surjection
$\Z[\ms{G}_L] \to S_L$, giving a filtration of the ring $S_L$.  Let $\gr_{\Xi}
S_L$ denote the associated graded ring.  We will show that $\gr_{\Xi} S_L$ is a
toric ring (Proposition~\ref{prop:toric}).  For a graph $\Gamma$ of level $n$,
let $\ol{X}_{\Gamma}$ denote the image of $X_{\Gamma}$ in $F^n_{\Xi} S_L/
F^{n-1}_{\Xi} S_L$.  Clearly the $\ol{X}_{\Gamma}$ span $\gr_{\Xi} S_L$.
\Ncom{$F^i_{\Xi} \Z[\ms{G}_L] \to F^i_{\Xi} S_L, \text{tor.\ Gr.\ ring},
\ol{X}_{\Gamma}$}

Let $F^i_{\Xi} R_L$ be the filtration on $R_L$ induced from its inclusion into
$S_L$.  Let $\gr_{\Xi} R_L$ be the associated graded ring.  It is naturally the
subring of $\gr_{\Xi} S_L$ spanned by the $\ol{X}_{\Gamma}$ for which $\Gamma$
is regular.
\Ncom{$F^i R_L , \gr_{\Xi} R_L$}

\subsection{Presentation of the ring $\gr_{\Xi} S_L$}
\label{ss:torpres}

Let $J_{\Xi}$ denote the ideal in $\Z[\ms{G}_L]$ generated by the following
(cf.\ \S \ref{ss:rings}):
\Ncom{$J_{\Xi}=\Z[\ms{G}_L]/ \langle s,  tP \rangle$}
\begin{itemize}
\item \emph{Loop relation:} If $\Gamma$ has a loop then $X_{\Gamma}=0$.
\item \emph{Sign relation:} If $\Gamma'$ is obtained from $\Gamma$ by reversing
the direction of an edge then $X_{\Gamma}=-X_{\Gamma'}$.
\item \emph{Toric Pl\"ucker relation:} If $a$, $b$, $c$ and $d$ are elements of
$L$ satisfying:
\begin{equation}
\label{eq-overlap-cond}
\begin{array}{l}
\textrm{the path from $a$ to $b$ in $\Xi$ meets the path from $c$ to $d$,} \\
\textrm{and the path from $a$ to $c$  meets the path from $b$ to $d$
(see Figure~\ref{f:torpluck})}
\end{array}
\end{equation}
then $X_{\dedgescr{ab}} X_{\dedgescr{cd}} = X_{\dedgescr{ac}}
X_{\dedgescr{bd}}$.
\end{itemize}
(The notation $J_{\Xi}$ is only used in \S \ref{ss:torpres}.)

\begin{figure}[!ht]
\begin{displaymath}
\begin{xy}
(0, 5)*{}="A"; (0, -5)*{}="B"; (5, 0)*{}="C"; (20, 0)*{}="D";
(25, 5)*{}="E"; (25, -5)*{}="F";
"A"*{\bullet}; "B"*{\bullet}; "C"*{\bullet}; "D"*{\bullet};
"E"*{\bullet}; "F"*{\bullet};
"A"; "C"; **\dir{-};
"B"; "C"; **\dir{-};
"E"; "D"; **\dir{-};
"F"; "D"; **\dir{-};
"C"; (10, 0); **\dir{-};
"D"; (15, 0); **\dir{-};
(10, 0); (15, 0); **\dir{.};
(0, 7)*{\ss a}; (0, -7)*{\ss d}; (25, 7)*{\ss b}; (25, -7)*{\ss c};
(2, 5); (5, 2); **\dir{-};
(5, 2); (10, 2); **\dir{-};
(10, 2); (15, 2); **\dir{.};
(15, 2); (20, 2); **\dir{-};
(20, 2); (23, 5); **\dir{-};
(2, -5); (5, -2); **\dir{-};
(5, -2); (10, -2); **\dir{-};
(10, -2); (15, -2); **\dir{.};
(15, -2); (20, -2); **\dir{-};
(20, -2); (23, -5); **\dir{-};
\end{xy}
\qquad = \qquad
\begin{xy}
(0, 5)*{}="A"; (0, -5)*{}="B"; (5, 0)*{}="C"; (20, 0)*{}="D";
(25, 5)*{}="E"; (25, -5)*{}="F";
"A"*{\bullet}; "B"*{\bullet}; "C"*{\bullet}; "D"*{\bullet};
"E"*{\bullet}; "F"*{\bullet};
"A"; "C"; **\dir{-};
"B"; "C"; **\dir{-};
"E"; "D"; **\dir{-};
"F"; "D"; **\dir{-};
"C"; (10, 0); **\dir{-};
"D"; (15, 0); **\dir{-};
(10, 0); (15, 0); **\dir{.};
(0, 7)*{\ss a}; (0, -7)*{\ss d}; (25, 7)*{\ss b}; (25, -7)*{\ss c};
(2, 5); (10, -2); **\dir{-};
(10, 2); (15, 2); **\dir{.};
(15, 2); (20, 2); **\dir{-};
(20, 2); (23, 5); **\dir{-};
(2, -5); (10, 2); **\dir{-};
(10, -2); (15, -2); **\dir{.};
(15, -2); (20, -2); **\dir{-};
(20, -2); (23, -5); **\dir{-};
\end{xy}
\end{displaymath}
\caption{The toric Pl\"ucker relation $X_{\protect\dedgescr{ab}}
X_{\protect\dedgescr{cd}}=X_{\protect\dedgescr{ac}} X_{\protect\dedgescr{bd}}$.
Both pairs of geodesics overlap on the horizontal edge of the trivalent tree
$\Xi$.  Note that $\protect\dedge{ad}$ and $\protect\dedge{bc}$ do not
overlap.}
\label{f:torpluck}
\end{figure}
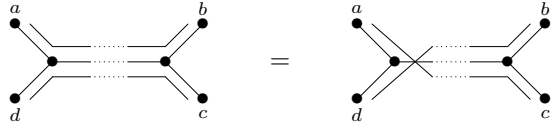

The purpose of this section is to prove the following, used in \S
\ref{ss:toric}  to identify $\gr_{\Xi} S_L$ with a semi-group algebra of
``weightings.''

\begin{proposition}
\label{prop:torpres}
The map $\Z[\ms{G}_L] \to \gr_{\Xi} S_L$ given by $X_{\Gamma} \mapsto
\ol{X}_{\Gamma}$ is surjective with kernel $J_{\Xi}$.
\end{proposition}

The kernel $I$ of the map $\Z[\ms{G}_L] \to S_L$ is generated by the sign and
Pl\"ucker relations.  The kernel of the map $\Z[\ms{G}_L] \to \gr_{\Xi} S_L$ is
the ideal generated by the leading terms of all elements of $I$.  In general,
of course, this is not the same as the ideal generated by the leading terms
of a generating set of $I$.  Proposition~\ref{prop:torpres} says that
in this situation, however, this is the case.

\begin{proof}
The map is clearly surjective and contains the sign relation in its kernel.  We
now check that the toric Pl\"ucker relation lies in its kernel as well.  Let
$a$, $b$, $c$ and $d$ belong to $L$ and satisfy \eqref{eq-overlap-cond}.  The
equation
\begin{displaymath}
X_{\dedgescr{ab}} X_{\dedgescr{cd}}= X_{\dedgescr{ac}} X_{\dedgescr{bd}}
+ X_{\dedgescr{ad}} X_{\dedgescr{bc}}
\end{displaymath}
holds in $S_L$ (the normal Pl\"ucker relation).  The two graphs $\dedge{ab}
\cdot \dedge{cd}$ and $\dedge{ac} \cdot \dedge{bd}$ have the same level, say
$n$, since when drawn in $\Xi$ they use the same edges with the same
multiplicity (see Figure~\ref{f:torpluck}).  The remaining graph $\dedge{ad}
\cdot \dedge{bc}$ has level less than $n$ (again, see Figure~\ref{f:torpluck}).
Thus all terms in the above relation lie in $F^n_{\Xi} S_L$.  Reducing modulo
$F^{n-1}_{\Xi} S_L$ we obtain
\begin{displaymath}
\ol{X}_{\dedgescr{ab}} \ol{X}_{\dedgescr{cd}} = \ol{X}_{\dedgescr{ac}}
\ol{X}_{\dedgescr{bd}}
\end{displaymath}
which shows that the toric Pl\"ucker relation lies in $J_{\Xi}$.

We now show that non-crossing graphs span $\Z[\ms{G}_L]/J_{\Xi}$.  Embed $L$
into the unit circle in the plane in such a way that $\Xi$ can be drawn
inside the circle without any crossings.  Let $\Gamma$ be a graph on $L$
which contains crossing edges $\dedge{ab}$ and $\dedge{cd}$.  The paths
$ab$ and $cd$ then meet in $\Xi$.  The same reasoning as in the
non-toric case now applies:  applying the toric Pl\"ucker relation to this
pair of edges yields an identity $\ol{X}_{\Gamma}=\ol{X}_{\Gamma'}$
where the total lengths of edges in $\Gamma'$ is less than that of $\Gamma$
(here length is computed as distance in the plane, not the trivalent tree).
Continuing in this manner, we get an expression $\ol{X}_{\Gamma}=
\ol{X}_{\Gamma'}$ where $\Gamma'$ is non-crossing. 

It is now formal to conclude that $\Z[\ms{G}_L]/J_{\Xi} \to \gr_{\Xi} S_L$ is
an isomorphism.  We elaborate on this.  Choose a set $Z$ of directed
non-crossing graphs such that for each undirected non-crossing graph $\Gamma$
there is a unique way to direct the edges of $\Gamma$ such that the resulting
graph belongs to $Z$.  The previous paragraph shows that if $\Gamma$ is any
graph of level $n$ we can find $\Gamma' \in Z$ such that $X_{\Gamma}=\pm
X_{\Gamma'}+Y$ holds in $S_L$, where $Y \in F^{n-1}_{\Xi} S_L$.  Applying this
result to $Y$ repeatedly, we find that the $X_{\Gamma}$ with $\Gamma$ in $Z$
and level at most $n$ form a basis of $F^n_{\Xi} S_L$ (we already know they are
linearly independent).  It thus follows that the $\ol{X}_{\Gamma}$ with
$\Gamma$ in $Z$ are linearly independent in $\gr_{\Xi} S_L$.  Since the
surjection $\Z[\ms{G}_L]/J_{\Xi} \to \gr_{\Xi} S_L$ takes a spanning set
to a set of linearly independent vectors, it must be an isomorphism.
\end{proof}

\subsection{Weightings}
\label{ss:weightings}

A \emph{weighting} $\xi$ on a trivalent tree $\Xi$ is an assignment of a
non-negative integer to each edge of $\Xi$.  Define the \emph{weight triple} of
a trinode $v$ of $\Xi$ to be the weights of the three edges connected to $v$.
(We write it as an ordered triple even though it is not ordered.)  Consider the
following two equivalent conditions on weight triples $(a, b, c)$:
\Ncom{$\text{weighting on $\Xi$, weight triple}$}
\begin{list}{}{}
\item[(W1)] The triple $(a, b, c)$ satisfies the triangle inequalities ($a$,
$b$ and $c$ can form the sides of a triangle) and $a+b+c$ is even.
\item[(W2)] There exists a triple $(x, y, z)$ of non-negative integers such
that $a=x+y$, $b=x+z$ and $c=y+z$.
\end{list}
Note that the triple $(x, y, z)$ is uniquely determined by $(a, b, c)$.  We say
that a weighting is \emph{admissible} if the weight triple at each trinode
satisfies these conditions.
\Ncom{$\text{adm.\ weighting, wt triple}$}

Let $\Gamma$ be an undirected graph on the leaves of $\Xi$.  We say that an
edge $e$ of $\Xi$ \emph{meets} an edge $\uedge{ij}$ of $\Gamma$ if $e$ occurs
in the geodesic joining $i$ and $j$.  We define a weighting $\xi_{\Gamma}$ of
$\Xi$ by assigning to an edge of $\Xi$ the number of edges of $\Gamma$ which
it meets.  We call $\xi_{\Gamma}$ the weighting of $\Xi$ \emph{associated} to
$\Gamma$ (see Figure~\ref{f:weight}). 
\Ncom{$\text{meets, associated}, \xi_{\Gamma}$}

Define \emph{(toric) Pl\"ucker equivalence} to be the equivalence relation
$\sim$ on the semi-group of undirected graphs $\ms{G}^{\un}_L$ on $L$ generated
by the following two conditions:
\Ncom{$\ms{G}^{\un}_L, \text{(t) Pl equiv} \sim$}
\begin{itemize}
\item Given $a$, $b$, $c$, $d$ in $L$ satisfying \eqref{eq-overlap-cond},
$\uedge{ab} \cdot \uedge{cd} \sim \uedge{ac} \cdot \uedge{bd}$.
\item If $\Gamma \sim \Gamma'$ and $\Gamma''$ is any graph then $\Gamma
\cdot \Gamma'' \sim \Gamma' \cdot \Gamma''$.
\end{itemize}

\begin{figure}[!ht]
\begin{displaymath}
\begin{xy}
(-6, 6)*{}="A"; (-6, -6)*{}="B"; (17, 14)*{}="C";
(24, 7)*{}="D"; (17, -14)*{}="E"; (24, -7)*{}="F";
"A"; "B"; **\dir{-};
"A"; "C"; **\dir{-};
"B"; "C"; **\dir{-};
"D"; "E"; **\dir{-};
"E"; "F"; **\dir{-};
"D"; "F"; **\dir{-};
"A"*{\bullet}; "B"*{\bullet}; "C"*{\bullet};
"D"*{\bullet}; "E"*{\bullet}; "F"*{\bullet};
(-6, 9)*{\ss 2}; (-6, -8)*{\ss 1}; (17, 17)*{\ss 3};
(24, 10)*{\ss 4}; (17, -16)*{\ss 6}; (24, -9)*{\ss 5};
\end{xy}
\qquad \qquad
\begin{xy}
(-6, 6)*{}="A"; (-6, -6)*{}="B"; (17, 14)*{}="C";
(24, 7)*{}="D"; (17, -14)*{}="E"; (24, -7)*{}="F";
(0, 0)*{}="P1"; (10, 0)*{}="P2"; (17, 7)*{}="P3"; (17, -7)*{}="P4";
"A"; "P1"; **\dir{-};
"B"; "P1"; **\dir{-};
"C"; "P3"; **\dir{-};
"D"; "P3"; **\dir{-};
"E"; "P4"; **\dir{-};
"F"; "P4"; **\dir{-};
"P1"; "P2"; **\dir{-};
"P3"; "P2"; **\dir{-};
"P4"; "P2"; **\dir{-};
"A"*{\bullet}; "B"*{\bullet}; "C"*{\bullet};
"D"*{\bullet}; "E"*{\bullet}; "F"*{\bullet};
"P1"*{\bullet}; "P2"*{\bullet}; "P3"*{\bullet}; "P4"*{\bullet};
(-3, 3)*{2}; (-3, -3)*{2}; (5, 0)*{2};
(13.5, 3.5)*{4}; (17, 10.5)*{2}; (20.5, 7)*{2};
(13.5, -3.5)*{2}; (17, -10.5)*{2}; (20.5, -7)*{2};
(-6, 9)*{\ss 2}; (-6, -8)*{\ss 1}; (17, 17)*{\ss 3};
(24, 10)*{\ss 4}; (17, -16)*{\ss 6}; (24, -9)*{\ss 5};
\end{xy}
\end{displaymath}
\caption{A graph on the leaves of a trivalent tree and its associated
weighting.}
\label{f:weight}
\end{figure}
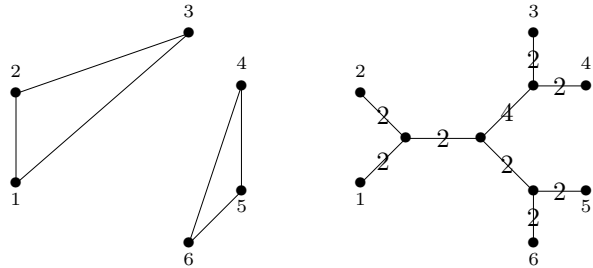

The following result gives a correspondence between graphs and weightings.

\begin{proposition}
\label{prop-wt-gr}
Associating a weighting of $\Xi$ to a graph induces a bijection between
the Pl\"ucker classes of undirected graphs on the leaves of $\Xi$ and the
admissible weightings of $\Xi$.
\end{proposition}

\begin{proof}
Let $\Gamma$ be a graph and let $\xi$ be its associated weighting.  We first
show that $\xi$ is admissible.  Let $v$ be a trinode of $\Xi$ and let
$(a, b, c)$ be its weight triple.  Consider the picture
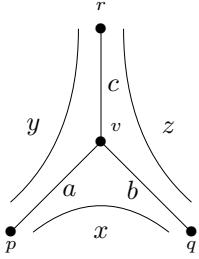
\begin{displaymath}
\begin{xy}
(0,0)*{}="V"; (-12, -12)*{}="A"; (12, -12)*{}="B"; (0, 15)*{}="C";
"V"; "A"; **\dir{-}; "V"; "B"; **\dir{-}; "V"; "C"; **\dir{-};
"V"*{\bullet}; "A"*{\bullet}; "B"*{\bullet}; "C"*{\bullet};
(2, 2)*{\ss v}; (-12, -14)*{\ss p}; (12, -14)*{\ss q}; (0, 18)*{\ss r};
(-4.2, -6.4)*{a}; (4.2, -6.4)*{b}; (1.6, 7.5)*{c};
(-12, -8); (-3, 15); **\crv{(-3, 0)};
(12, -8); (3, 15); **\crv{(3, 0)};
(-9, -12); (9, -12); **\crv{(0, -5)};
(-9, 2)*{y}; (9, 2)*{z}; (0, -12)*{x};
\end{xy}
\end{displaymath}
Here $a$ is the number of edges of $\Gamma$ that meet edge $\uedge{pv}$; $b$
and $c$ are defined similarly.  We let $x$ be the number of edges of $\Gamma$
which, when drawn as geodesics on $\Xi$, go through both $p$ and $q$; we let
$y$ and $z$ be the analogous quantities.  Clearly $a=x+y$, $b=x+z$, and
$c=y+z$.  Thus we have shown that the weighting $\xi$ satisfies condition (W2)
at each trinode and is therefore admissible.

If we apply a toric Pl\"ucker relation to $\Gamma$ then its associated
weighting does not change:  the two pairs of edges in the toric Pl\"ucker
relation contain the same edges when drawn as geodesics.  Thus associating a
weighting to a graph yields a well-defined map from the Pl\"ucker classes of
graphs to the set of admissible weightings.  We now show that this map is
bijective.

We first prove that it is surjective.  We are given an admissible
weighting $\xi$ and we must produce a graph $\Gamma$ such that $\xi$ is its
associated weighting.  To do this it suffices to prove the following:  given
an admissible weighting $\xi$ there exist two leaves $i$ and $j$ of $\Xi$ such
that when the geodesic joining $i$ and $j$ is subtracted from $\xi$ the
resulting weighting is still admissible.  For, if this is the case, then we
can let $\uedge{ij}$ be an edge of $\Gamma$, subtract the geodesic joining
$i$ and $j$ from $\xi$ and proceed by induction.

Thus let $\xi$ be an admissible weighting of $\Xi$.  Let $i$ be any leaf of
$\Xi$ for which $\xi$ does not vanish on the edge containing $i$.  We produce
the vertex $j$ by the following greedy algorithm.  Put $v_0=i$ and let
$v_1$ be the unique trinode connected to $i$.  Assume now that we have
defined $v_0$ through $v_k$.  If $v_k$ is a leaf then stop and put $j=v_k$.
Otherwise $v_k$ is connected to two vertices other than $v_{k-1}$.  Let
$v_{k+1}$ be the one for which the corresponding edge has higher weight; if
the two edges have the same weight then pick $v_{k+1}$ arbitrarily.  We have
thus produced a pair of leaves $i$ and $j$; note that the $v_k$ are the
trinodes in geodesic joining $i$ and $j$.  We must show that when this geodesic
is subtracted from $\xi$ the resulting weighting is still admissible.

Let $\ell$ be lie strictly between 0 and $k$ so that $v_{\ell}$ is a trinode.
We have the picture
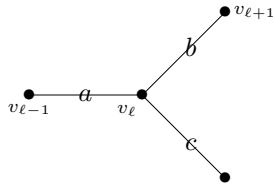
\begin{displaymath}
\begin{xy}
(-15, 0)*{}="A"; (0, 0)*{}="B"; (11, 11)*{}="C"; (11, -11)*{}="D";
"A"; "B"; **\dir{-}; "C"; "B"; **\dir{-}; "D"; "B"; **\dir{-};
"A"*{\bullet}; "B"*{\bullet}; "C"*{\bullet}; "D"*{\bullet};
(-15, -2)*{\ss v_{\ell-1}}; (-2, -2)*{\ss v_{\ell}}; (15, 11)*
{\ss v_{\ell+1}};
(-7.5, 0)*{a}; (6.5, 6.5)*{b}; (6.5, -6.5)*{c};
\end{xy}
\end{displaymath}
By definition of $v_{\ell+1}$ we have $b \ge c$.  We know $(a, b, c)$ satisfies
(W1).  Clearly $(a-1, b-1, c)$ still satisfies the parity condition.  We must
show that it still satisfies the triangle inequalities, which amounts to
proving  $c \le a+b-2$.  This inequality could fail in two ways:  1) $a=0$ and
$c$ equals $b$ or $b-1$; or 2) $a=1$ and $b=c$.  The first case is ruled out by
the way we selected $v_{\ell-1}$ and an easy induction argument.  The second
case is ruled out since $a+b+c$ is even.  This proves that the greedy algorithm
indeed works and completes the proof that our map is surjective.  Note that in
building the graph from the weighting there are many arbitrary choices.

We now prove that the map is injective, i.e.,  that if $\Gamma$ and $\Gamma'$
have the same weighting $\xi$ then they are Pl\"ucker equivalent.  To show this
we show that $\Gamma$ is Pl\"ucker equivalent to any of the graphs constructed
out of $\xi$ by using the greedy algorithm.  It suffices to prove that if
$v_0, \ldots, v_k$ is a sequence coming out of the greedy algorithm then we can
apply toric Pl\"ucker relations to $\Gamma$ so that $v_0$ and $v_k$ are
connected in $\Gamma$.  For then we may remove this edge from $\Gamma$ and the
corresponding path in $\xi$ and proceed by induction.

Thus let $v_0, \ldots, v_k$ come out of the greedy algorithm.  We prove by
induction on $\ell$ that $\Gamma$ is Pl\"ucker equivalent to a graph
containing an edge which passes through $v_0$ and $v_{\ell}$ (we say that an
edge of $\Gamma$ passes through two vertices of $\Xi$ if its corresponding
geodesic does).  This is clear for $\ell=1$.  Thus assume it is true for
$\ell$ and we show that it is true for $\ell+1$.  We have the picture
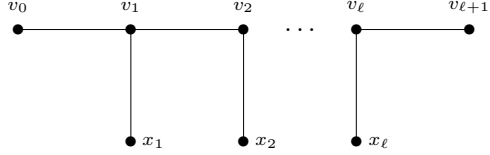
\begin{displaymath}
\begin{xy}
(0, 0)*{}="V0"; (15, 0)*{}="V1"; (30, 0)*{}="V2"; (45, 0)*{}="V3";
(60, 0)*{}="V4"; (15, -15)*{}="X1"; (30, -15)*{}="X2";
(45, -15)*{}="X3";
"V0"; "V1"; **\dir{-};
"V1"; "V2"; **\dir{-};
"V3"; "V4"; **\dir{-};
"V1"; "X1"; **\dir{-};
"V2"; "X2"; **\dir{-};
"V3"; "X3"; **\dir{-};
(37.5, 0)*{\ldots};
"V0"*{\bullet}; "V1"*{\bullet}; "V2"*{\bullet}; "V3"*{\bullet};
"V4"*{\bullet}; "X1"*{\bullet}; "X2"*{\bullet}; "X3"*{\bullet};
(0, 3)*{\ss v_0}; (15, 3)*{\ss v_1}; (30, 3)*{\ss v_2};
(45, 3)*{\ss v_{\ell}}; (60, 3)*{\ss v_{\ell+1}};
(18, -15)*{\ss x_1}; (33, -15)*{\ss x_2}; (48, -15)*{\ss x_{\ell}};
\end{xy}
\end{displaymath}
Here $x_i$ is the unique vertex connected to $v_i$ besides $v_{i-1}$ and
$v_{i+1}$.  We have assumed that $\Gamma$ contains an edge $e$ passing through
$v_0$ and $v_{\ell}$; we must show that $\Gamma$ is Pl\"ucker equivalent to a
graph containing an edge passing through $v_0$ and $v_{\ell+1}$.  Now, $e$
itself either passes through $v_{\ell+1}$ or $x_{\ell}$.  In the former case
we are done.  Thus we may assume that $e$ passes through $x_{\ell}$.

By the definition of the greedy algorithm we have
\begin{displaymath}
\xi(\uedge{ v_{\ell} x_{\ell}}) \le \xi(\uedge{v_{\ell} v_{\ell+1}}), \qquad
\xi(\uedge{ v_i v_{i+1}}) \ne 0.
\end{displaymath}
Thus there exists an edge of $\Gamma$ passing through $v_{\ell}$ and
$v_{\ell+1}$.  If every edge which passed through $v_{\ell}$ and $v_{\ell+1}$
also passed through $x_{\ell}$ then, by the inequality, every edge which passed
through $x_{\ell}$ and $v_{\ell}$ would also pass through $v_{\ell+1}$; it
would follow that no edge could pass through $v_{\ell-1}$ and $v_{\ell}$.
However, this would imply $\xi(\uedge{v_{\ell-1} v_{\ell}})=0$, a
contradiction.  Thus there exists an edge $e'$ of $\Gamma$ which passes
through $v_{\ell+1}$ and $v_{\ell}$ but \emph{not} through $x_{\ell}$.  If
$v_0$ is a vertex of $e'$ then we are done.  Otherwise, applying the toric
Pl\"ucker relation to $e$ and $e'$ yields a graph containing an edge passing
through $v_0$ and $v_{\ell+1}$.  This complete the proof.
\end{proof}

\subsection{The rings $\gr_{\Xi} S_L$ and $\gr_{\Xi} R_L$ as semi-group
algebras of weightings}
\label{ss:toric}

Let $\Xi$ be a trivalent tree with leaf set $L$.  Embed $L$ into the unit
circle in such a way that $\Xi$ can be drawn inside the unit circle without
crossings.  Choose a total order on $L$ which is compatible with its
embedding into the circle in the sense that if $a \le b \le c$ then one
encounters $b$ when traveling clockwise from $a$ to $c$.  For $a, b \in L$
define $\epsilon_{ab}$ to be 1 if $a<b$, $-1$ if $a>b$ and 0 if $a=b$.
For a directed graph $\Gamma$ on $L$ define $\epsilon_{\Gamma}$ to be the
product of the $\epsilon_{ab}$ over the edges $\dedge{ab}$ of $\Gamma$.  We
 write $\Gamma^{\un}$ for the undirected graph associated to $\Gamma$.
Finally, let $\ms{S}_{\Xi}$ denote the set of admissible weights on $\Xi$.
It is a semi-group since the sum of two admissible weightings is again
admissible.  We can now prove:

\begin{proposition}
\label{prop:toric}
There is a unique isomorphism of rings $\gr_{\Xi}{S_L} \to \Z[\ms{S}_{\Xi}]$
mapping $\ol{X}_{\Gamma}$ to $\epsilon_{\Gamma} \xi_{\Gamma^{\un}}$.
\end{proposition}

\begin{proof}
We define an auxiliary ring by modifying the sign relation in the presentation
of $\gr_{\Xi} S_L$ given in \S \ref{ss:torpres}.  Define the ideal
$J^{\un}_{\Xi}$ of $\Z[\ms{G}_L]$ by the following types of relations:
\begin{itemize}
\item \emph{Loop relation:} If $\Gamma$ has a loop then $X_{\Gamma}=0$.
\item \emph{Modified sign relation:} If $\Gamma'$ is obtained from $\Gamma$ by
reversing the direction of an edge then $X_{\Gamma}=X_{\Gamma'}$.
\item \emph{Toric Pl\"ucker relation:} As in \S \ref{ss:torpres}.
\end{itemize}
Denote the image of $X_{\Gamma}$ in $\Z[\ms{G}]/J_L^{\un}$ by
$\ol{X}_{\Gamma}^{\un}$.  Note that $\ol{X}_{\Gamma}^{\un}$ makes sense for
an \emph{undirected} graph $\Gamma$, and that if $\Gamma$ and $\Gamma'$ are
Pl\"ucker equivalent undirected graphs then the toric Pl\"ucker relation
implies $\ol{X}_{\Gamma}^{\un}=\ol{X}_{\Gamma'}^{\un}$.  By comparing the
definition of $J_L^{\un}$ to the presentation of $\gr_{\Xi} S_L$ given in
Proposition~\ref{prop:torpres}, we find that the map
\begin{displaymath}
\gr_{\Xi}{S_L}=\Z[\ms{G}_L]/J_{\Xi} \to \Z[\ms{G}_L]/J_{\Xi}^{\un}, \qquad
\ol{X}_{\Gamma} \mapsto \epsilon_{\Gamma} \ol{X}_{\Gamma}^{\un}
\end{displaymath}
is well-defined and an isomorphism.  From the equivalence of weightings and
Pl\"ucker classes of undirected graphs given in Proposition~\ref{prop-wt-gr},
we find that the map
\begin{displaymath}
\Z[\ms{G}_L]/J_{\Xi}^{\un} \to \Z[\ms{S}_{\Xi}], \qquad
\ol{X}_{\Gamma}^{\un} \mapsto \xi_{\Gamma^{\un}}
\end{displaymath}
is a well-defined isomorphism.  The proposition now follows.
\end{proof}

We translate this result to the regular case.  Call a weighting $\xi$ on $\Xi$
\emph{regular of degree $d$} if for each leaf $v$ we have $\xi(e_v)=d$, where
$e_v$ is the unique edge containing $v$.  Let $\ms{R}_{\Xi}$ be the semi-group
of admissible regular weightings on $\Xi$.  We then have:
\Ncom{$\ms{R}_{\Xi}$}

\begin{corollary}
\label{cor:toric}
There is  a unique isomorphism of graded rings $\gr_{\Xi} R_L \to
\Z[\ms{R}_{\Xi}]$ which takes $\ol{X}_{\Gamma}$ to $\epsilon_{\Gamma} \cdot
\xi_{\Gamma^{\un}}$.
\end{corollary}

Proposition~\ref{prop:toric} and Corollary~\ref{cor:toric} show that $\gr_{\Xi}
S_L$ and $\gr_{\Xi} R_L$ are semi-group algebras and therefore toric rings.
(It is not difficult to see that $\ms{S}_{\Xi}$ and $\ms{R}_{\Xi}$ are the
set of lattice points in a strictly convex rational polyhedral cone,
but 
 this will not be of importance to us.)

\subsection{Reduced weightings}

We close \S \ref{s:gentor} with a discussion of reduced weightings, which we
use in \S \ref{s:ytree}--\ref{s:gsc}.  Let $\Xi$ be a trivalent tree.  By a
\emph{reduced weighting} on $\Xi$ we simply mean a weighting on $\Xi$ --- the
terms are synonymous but used to distinguish the usage of ``admissible.''  We
say that a reduced weighting is \emph{admissible} if it satisfies the triangle
inequality (as in (W1)) at each trinode; the parity condition is not enforced.
We say that a reduced weighting $\xi$ is \emph{regular of degree $\le d$} if
$\xi(e_v) \le d$ for all leaves $v$, where $e_v$ is the edge meeting leaf $v$.
Let $\ol{\ms{R}}_{\Xi}$ be the set of all ordered pairs $(\xi, d)$ with $d$ a
non-negative integer and $\xi$ an admissible reduced weighting which is
regular of degree $\le d$.  We define the \emph{degree} of $(\xi, d)$ to be
$d$.
\Ncom{$\text{(adm.) red.\ wting}, \ol{\ms{R}}_{\Xi}$}
\Rnts{reduced weightings used in \S \ref{ss:tor-segre}}

Let $\Xi$ be a matched trivalent tree and let $\Xi^-$ be the trivalent tree
obtained by deleting the leaves of $\Xi$ and the edges that they touch.  We
call $\Xi^-$ the \emph{truncation} of $\Xi$.  
\Ncom{$\text{truncation} \Xi^-$}

\begin{proposition}
\label{p2-trun-wt}
With notation as above, there is a canonical isomorphism of semi-groups
$\ms{R}_{\Xi} \to \ol{\ms{R}}_{\Xi^-}$ preserving degree.  The image of a
weighting $\xi$ on $\Xi$ is the pair $(\xi', d)$ where $d$ is the degree of
$\xi$ and $\xi'$ is the weighting on $\Xi^-$ given by $\xi'(x)=\half \xi(x)$.
\end{proposition}

The proof is easy.  See Figure~\ref{f:trun-wt} for an illustration.

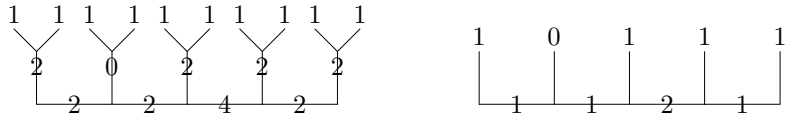
\begin{figure}[!ht]
\begin{displaymath}
\begin{xy}
(-20, 0)*{}="A"; (-10, 0)*{}="B"; (0, 0)*{}="C"; (10, 0)*{}="D";
(20, 0)*{}="E";
(-20, 7)*{}="A1"; (-10, 7)*{}="B1"; (0, 7)*{}="C1"; (10, 7)*{}="D1";
(20, 7)*{}="E1";
(-23, 10)*{}="A2"; (-17, 10)*{}="A3";
(-13, 10)*{}="B2"; (-7, 10)*{}="B3";
(-3, 10)*{}="C2"; (3, 10)*{}="C3";
(7, 10)*{}="D2"; (13, 10)*{}="D3";
(17, 10)*{}="E2"; (23, 10)*{}="E3";
"A"; "B"; **\dir{-};
"B"; "C"; **\dir{-};
"C"; "D"; **\dir{-};
"D"; "E"; **\dir{-};
"A"; "A1"; **\dir{-};
"B"; "B1"; **\dir{-};
"C"; "C1"; **\dir{-};
"D"; "D1"; **\dir{-};
"E"; "E1"; **\dir{-};
"A1"; "A2"; **\dir{-};
"A1"; "A3"; **\dir{-};
"B1"; "B2"; **\dir{-};
"B1"; "B3"; **\dir{-};
"C1"; "C2"; **\dir{-};
"C1"; "C3"; **\dir{-};
"D1"; "D2"; **\dir{-};
"D1"; "D3"; **\dir{-};
"E1"; "E2"; **\dir{-};
"E1"; "E3"; **\dir{-};
(-23, 12)*{1}; (-17, 12)*{1};
(-13, 12)*{1}; (-7, 12)*{1};
(-3, 12)*{1}; (3, 12)*{1};
(13, 12)*{1}; (7, 12)*{1};
(23, 12)*{1}; (17, 12)*{1};
(-20, 5)*{2}; (-10, 5)*{0}; (0, 5)*{2}; (10, 5)*{2}; (20, 5)*{2};
(-15, 0)*{2}; (-5, 0)*{2}; (5, 0)*{4}; (15, 0)*{2};
\end{xy}
\qquad \qquad
\begin{xy}
(-20, 0)*{}="A"; (-10, 0)*{}="B"; (0, 0)*{}="C"; (10, 0)*{}="D";
(20, 0)*{}="E";
(-20, 7)*{}="A1"; (-10, 7)*{}="B1"; (0, 7)*{}="C1"; (10, 7)*{}="D1";
(20, 7)*{}="E1";
"A"; "B"; **\dir{-};
"B"; "C"; **\dir{-};
"C"; "D"; **\dir{-};
"D"; "E"; **\dir{-};
"A"; "A1"; **\dir{-};
"B"; "B1"; **\dir{-};
"C"; "C1"; **\dir{-};
"D"; "D1"; **\dir{-};
"E"; "E1"; **\dir{-};
(-20, 9)*{1}; (-10, 9)*{0}; (0, 9)*{1}; (10, 9)*{1}; (20, 9)*{1};
(-15, 0)*{1}; (-5, 0)*{1}; (5, 0)*{2}; (15, 0)*{1};
\end{xy}
\end{displaymath}
\caption{An illustration of Proposition~\ref{p2-trun-wt}. 
The admissible weighting $\xi$ on the left is regular of degree 1, the 
associated reduced admissible weighting $\xi^-$ on the right is regular
of degree $\le 1$.}
\label{f:trun-wt}
\end{figure}


\section{The toric ideal is generated by quadratics and 
toric generalized Segre cubics}
\label{s:ytree}

We have described a family of toric degenerations of $R_L$, depending on a
choice of trivalent tree.  The purpose of \S \ref{s:ytree} is to choose a
specific family of trees (the Y-trees) to ensure the degenerated ring (i)
is generated in degree one, (ii) has relations generated in degrees two and
three, and (iii) is such that there is a precise description of the degree
three relations:

\begin{theorem}
\label{thm:ytree}
Let $\Xi$ be a Y-tree (defined in \S \ref{def:Ytree}).  Then $\gr_{\Xi} R_L$ is
generated in degree one, and the relations between degree one elements are
generated by quadratic relations and the generalized toric Segre cubic
relations (defined in \S \ref{ss:tor-segre}).
\end{theorem}

(Manon \cite{Manon} independently solved the presentation problem for a large
class of weighted trivalent trees, including this one; however, we will need
the form of (iii).) 

\subsection{The Y- and caterpillar trees}
\label{def:Ytree}

For $r \ge 3$, define the $r$th \emph{Y-tree} as:
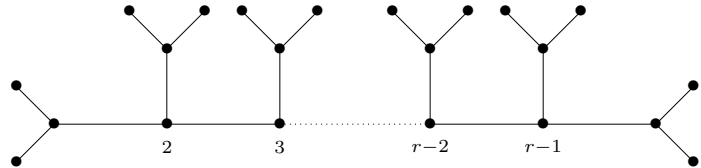
\Ncom{$\text{Y-tree}$}
\begin{displaymath}
\begin{xy}
(-25, 0)*{}="A"; (-10, 0)*{}="B"; (10, 0)*{}="C"; (25, 0)*{}="D";
(-40, 0)*{}="P"; (40, 0)*{}="Q";
(-25, 10)*{}="A1"; (-10, 10)*{}="B1"; (10, 10)*{}="C1"; (25, 10)*{}="D1";
(-30, 15)*{}="A2"; (-20, 15)*{}="A3";
(-15, 15)*{}="B2"; (-5, 15)*{}="B3";
(5, 15)*{}="C2"; (15, 15)*{}="C3";
(20, 15)*{}="D2"; (30, 15)*{}="D3";
(-45, 5)*{}="P2"; (-45, -5)*{}="P3";
(45, 5)*{}="Q2"; (45, -5)*{}="Q3";
"A"*{\bullet}; "B"*{\bullet}; "C"*{\bullet}; "D"*{\bullet};
"P"*{\bullet}; "Q"*{\bullet};
"A1"*{\bullet}; "B1"*{\bullet}; "C1"*{\bullet}; "D1"*{\bullet};
"A2"*{\bullet}; "A3"*{\bullet};
"B2"*{\bullet}; "B3"*{\bullet};
"C2"*{\bullet}; "C3"*{\bullet};
"D2"*{\bullet}; "D3"*{\bullet};
"P2"*{\bullet}; "P3"*{\bullet};
"Q2"*{\bullet}; "Q3"*{\bullet};
"P"; "A"; **\dir{-};
"A"; "B"; **\dir{-};
"B"; "C"; **\dir{..};
"C"; "D"; **\dir{-};
"D"; "Q"; **\dir{-};
"A"; "A1"; **\dir{-};
"B"; "B1"; **\dir{-};
"C"; "C1"; **\dir{-};
"D"; "D1"; **\dir{-};
"A1"; "A2"; **\dir{-};
"A1"; "A3"; **\dir{-};
"B1"; "B2"; **\dir{-};
"B1"; "B3"; **\dir{-};
"C1"; "C2"; **\dir{-};
"C1"; "C3"; **\dir{-};
"D1"; "D2"; **\dir{-};
"D1"; "D3"; **\dir{-};
"P"; "P2"; **\dir{-};
"P"; "P3"; **\dir{-};
"Q"; "Q2"; **\dir{-};
"Q"; "Q3"; **\dir{-};
(-25, -3)*{\ss 2};
(-10, -3)*{\ss 3};
(10, -3)*{\ss r-2};
(25, -3)*{\ss r-1};
\end{xy}
\end{displaymath}
There are $r$ ``Y's'' in the tree, and $2r$ leaves.  We call the vertex at $i$
the $i$th \emph{base vertex}.  We call an edge between two base vertices a
\emph{base edge}.  By a \emph{stalk} we mean one of the internal edges in one
of the Y's.  By the $i$th stalk, for $2 \le i \le r-1$, we mean the one above
the $i$th base vertex.  We call the two remaining stalks at either end of the
tree the 1st and $r$th stalk.  The Y-tree is a matched tree (\S
\ref{d:matched}).  Sometimes we bend the first and last horizontal edges so
that all the Y's are in a row, as in Figure~\ref{f:trun-wt}.  The $r$th
\emph{caterpillar tree} is the truncation of the $r$th Y-tree:
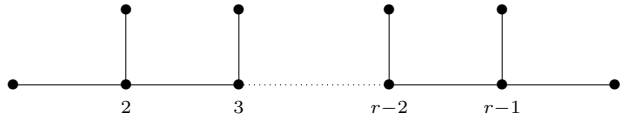
\Ncom{$\text{base vertex, base edge, stalk, caterpillar}$} 
\begin{displaymath}
\begin{xy}
(-25, 0)*{}="A"; (-10, 0)*{}="B"; (10, 0)*{}="C"; (25, 0)*{}="D";
(-25, 10)*{}="A1"; (-10, 10)*{}="B1"; (10, 10)*{}="C1"; (25, 10)*{}="D1";
(-40, 0)*{}="P"; (40, 0)*{}="Q";
"A"*{\bullet}; "B"*{\bullet}; "C"*{\bullet}; "D"*{\bullet};
"P"*{\bullet}; "Q"*{\bullet};
"A1"*{\bullet}; "B1"*{\bullet}; "C1"*{\bullet}; "D1"*{\bullet};
"P"; "A"; **\dir{-};
"A"; "B"; **\dir{-};
"B"; "C"; **\dir{..};
"C"; "D"; **\dir{-};
"D"; "Q"; **\dir{-};
"A"; "A1"; **\dir{-};
"B"; "B1"; **\dir{-};
"C"; "C1"; **\dir{-};
"D"; "D1"; **\dir{-};
(-25, -3)*{\ss 2};
(-10, -3)*{\ss 3};
(10, -3)*{\ss r-2};
(25, -3)*{\ss r-1};
\end{xy}
\end{displaymath}
We use the same terminology (base edges, stalks, etc.) for caterpillar trees.
Note that the third caterpillar tree is just a trinode.

\subsection{Generators of $\gr_{\Xi} R_L$}

We now prove the first half of Theorem~\ref{thm:ytree}.

\begin{proposition}
\label{prop:ytree-gen}
Let $\Xi$ be a Y-tree.  Then the ring $\gr_{\Xi} R_L$ is generated in degree
one.
\Rnts{used in \S \ref{ss:tor-segre}}
\end{proposition}

\begin{proof}
(This proof essentially applies to any matched tree.)  We follow the spirit of
our proof of Kempe's theorem (Theorem~\ref{thm:kempe}).  Assign to each leaf of
$\Xi$ a sign in such a way that in each matched pair of leaves one leaf is
positive and the other is negative.  Suppose $\Gamma$ is a regular undirected
graph of degree $k$.  Each edge of $\Gamma$ is then either positive, negative
or neutral.  We will show that $\Gamma$ is Pl\"ucker equivalent (\S
\ref{ss:weightings}) to a graph with only neutral edges.  As in the proof of
Theorem~\ref{thm:kempe}, this implies that it is a product of matchings,
expressing $\ol{X}_{\Gamma}$ as a product of degree one elements.  This proof
is more difficult than that of Theorem~\ref{thm:kempe} because we are now only
allowed to apply the Pl\"ucker relation to pairs of edges of $\Gamma$ which
meet (i.e., whose geodesics meet) in the tree $\Xi$.

{\it Step 1:}  Given a positive and negative edge of $\Gamma$ which meet
in $\Xi$, apply the toric Pl\"ucker relation.  In the resulting graph the
two new edges are both neutral and so the number of non-neutral edges has
decreased.  We can thus continue this process until we have reached a
graph which is Pl\"ucker equivalent to $\Gamma$ and contains no overlapping
positive and negative edges.

{\it Step 2:}
Suppose there are still non-neutral paths, but no two of opposite parity meet
each other.  Let $e$ be the leftmost base edge of $\Xi$ which separates
non-neutral edges of $\Gamma$ of opposite parity and let $e'$ be the edge
to the left of $e$.  The edge $e'$ will be a base edge unless $e$ is the first
base edge, in which case it will be the first stalk.  All non-neutral edges
of $\Gamma$ which are entirely to the left of $e$ are of the same type,
say positive.  All edges of $\Gamma$ which pass through $e$ are neutral.

There is a positive edge $\uedge{ab}$ of
$\Gamma$ which contains $e'$.  We claim there is an edge $\uedge{cd}$ of $\Gamma$
which contains $e$ and where $d$ is positive and to the right of $e$.
Assuming the claim (which we prove below), then $\uedge{ab}$ and $\uedge{cd}$
must meet (either at $e'$ or at the stalk between $e'$ and $e$); now apply the
toric Pl\"ucker relation to this pair.  The resulting positive path contains
$e$.  This brings the leftmost cluster of positive paths closer to the
negative paths.  By continuing this process, we will eventually cause the
two clusters to meet, at which point we  return to Step 1 and reduce the
number of non-neutral edges.  Continuing in this manner, we will 
eventually remove all non-neutral edges.

{\it Proof of claim:}
We now prove the claim that there must exist an edge $\uedge{cd}$ of $\Gamma$
containing $e$ and for which $d$ is positive and to the right of $e$.
Suppose there are $2m$ leaves to the left of $e$.  Let $A$ (resp.\ $B$, $C$)
be the number of positive (resp.\ negative, neutral) edges of $\Gamma$
entirely to the left of $e$.  Let $D$ (resp.\ $E$) be the number of neutral
edges containing $e$ for which the positive (resp.\ negative) leaf is to the
left of $e$.  Then
\begin{displaymath}
2A+C+D=mk, \qquad 2B+C+E=mk
\end{displaymath}
as the first (resp.\ second) counts $k$ times each positive (resp.\ negative)
leaf to the left of $e$, and there are $m$ such leaves.  We thus find $2A+D
=2B+E$ and as $B=0$ and $A \ne 0$ by assumption we conclude $E \ne 0$, as
claimed.
\end{proof}

\subsection{Relations in a semi-group algebra}

We now turn our attention to relations in $\gr_{\Xi} R_L$.  We begin with a
general discussion of relations in a semi-group algebra.  Let $\ms{R}$ be a
semi-group equipped with a homomorphism $\deg:\ms{R} \to \Z_{\ge 0}$, so that
the semi-group algebra $\Z[\ms{R}]$ is graded.  Assume that $\ms{R}$ is
generated by $\ms{V}=\deg^{-1}(1)$, so that the natural map $\Sym(\Z\langle
\ms{V} \rangle) \to \Z$ is a surjection (here $\Sym(\Z\langle \ms{V} \rangle)$
is the polynomial ring in indeterminates $\ms{V}$ while $\Z[\ms{R}]$ is the
semi-group algebra of $\ms{R}$).  Call the kernel $I$ of this surjection the
\emph{ideal of relations} of $\Z[\ms{R}]$.
\Ncom{$\text{ideal of relations}$}

Let $\ul{\xi}=(\xi_1, \ldots, \xi_n)$ and $\ul{\xi}'=(\xi'_1, \ldots, \xi'_n)$
be two elements of $\ms{V}^n$ such that $\sum \xi_i=\sum \xi_i'$.  We
then write $\ul{\xi} \sim \ul{\xi}'$ and say that $\ul{\xi}$ and $\ul{\xi}'$
are \emph{related}.  For an element $\xi$ of $\ms{V}$ let $[\xi]$ denote the
corresponding element of $\Sym(\Z\langle \ms{V} \rangle)$; for an element
$\ul{\xi}$ of $\ms{V}^n$ let $[\ul{\xi}]$ denote the monomial $[\xi_1]
\cdots [\xi_n]$.  Given related $\ul{\xi}$ and $\ul{\xi}'$ the element
$[\ul{\xi}]-[\ul{\xi}']$ belongs to $I$.  We call such relations
\emph{binomial relations}.  One easily verifies that $I$ is generated by
binomial relations.  
\Ncom{$\text{related, bin.\ rels.}$}

We say that a relation $\ul{\xi} \sim \ul{\xi}'$ has \emph{degree $\le k$} if
$\xi_i=\xi'_i$ holds for all but $k$ indices.  We say that $\ul{\xi} \sim
\ul{\xi}'$ has \emph{essentially degree $\le k$} if there exists a sequence of
relations 
\Ncom{$\text{ess.\ deg., ess.\ quad.}$}
\begin{displaymath}
\ul{\xi}=\ul{\xi}^{(0)} \sim \ul{\xi}^{(1)} \sim \ldots \sim \ul{\xi}^{(p)}
=\ul{\xi}'
\end{displaymath}
for which each $\ul{\xi}^{(i)} \sim \ul{\xi}^{(i+1)}$ has degree $\le k$.
This notation is only relevant to us when $k$ is 2 or 3 and we then use the
terms ``essentially quadratic'' and ``essentially cubic.'' If $\ul{\xi} \sim
\ul{\xi}'$ has essential degree $\le k$ then the element $[\ul{\xi}]-
[\ul{\xi}']$ of $I$ lies in the ideal generated by the $k$th graded piece of
$I$.

Note that we consider our tuples as ordered, so that $(\xi_1, \xi_2)=
(\xi_2, \xi_1)$ constitutes a non-trivial relation.  However, as the symmetric
group is generated by transpositions, it follows that if $\ul{\xi}'$ is a
permutation of the tuple $\ul{\xi}$ then the relation $\ul{\xi} \sim
\ul{\xi}'$ is essentially quadratic.

\subsection{The toric generalized Segre cubic relation}
\label{ss:tor-segre}

We have the following relation between degree one reduced weightings on the
third caterpillar tree:
\begin{equation}
\label{eq:torsegre}
\begin{xy}
(-8, -3)*{}="A"; (0, -3)*{}="B"; (8, -3)*{}="C"; (0, 2)*{}="B1";
"A"; "B"; **\dir{-};
"B"; "C"; **\dir{-};
"B"; "B1"; **\dir{-};
(-4, -3)*{1}; (0, 4)*{0}; (4, -3)*{1};
\end{xy}
\quad
\begin{xy}
(-8, -3)*{}="A"; (0, -3)*{}="B"; (8, -3)*{}="C"; (0, 2)*{}="B1";
"A"; "B"; **\dir{-};
"B"; "C"; **\dir{-};
"B"; "B1"; **\dir{-};
(-4, -3)*{0}; (0, 4)*{1}; (4, -3)*{1};
\end{xy}
\quad
\begin{xy}
(-8, -3)*{}="A"; (0, -3)*{}="B"; (8, -3)*{}="C"; (0, 2)*{}="B1";
"A"; "B"; **\dir{-};
"B"; "C"; **\dir{-};
"B"; "B1"; **\dir{-};
(-4, -3)*{1}; (0, 4)*{1}; (4, -3)*{0};
\end{xy}
\qquad = \qquad
\begin{xy}
(-8, -3)*{}="A"; (0, -3)*{}="B"; (8, -3)*{}="C"; (0, 2)*{}="B1";
"A"; "B"; **\dir{-};
"B"; "C"; **\dir{-};
"B"; "B1"; **\dir{-};
(-4, -3)*{1}; (0, 4)*{1}; (4, -3)*{1};
\end{xy}
\quad
\begin{xy}
(-8, -3)*{}="A"; (0, -3)*{}="B"; (8, -3)*{}="C"; (0, 2)*{}="B1";
"A"; "B"; **\dir{-};
"B"; "C"; **\dir{-};
"B"; "B1"; **\dir{-};
(-4, -3)*{1}; (0, 4)*{1}; (4, -3)*{1};
\end{xy}
\quad
\begin{xy}
(-8, -3)*{}="A"; (0, -3)*{}="B"; (8, -3)*{}="C"; (0, 2)*{}="B1";
"A"; "B"; **\dir{-};
"B"; "C"; **\dir{-};
"B"; "B1"; **\dir{-};
(-4, -3)*{0}; (0, 4)*{0}; (4, -3)*{0};
\end{xy}
\end{equation}
(We will omit $+$ signs in such equations, interpreting them as binomial
relations in $\Z[\ms{R}]$.)  One obtains this relation by
converting the usual graphical Segre cubic \eqref{e:Segre} into a relation
between weightings on the third Y-tree and then passing to the associated
reduced weighting on the third caterpillar tree.  One may verify by hand that
this single relation generates all relations among admissible reduced
weightings on the third caterpillar tree.

We now introduce a class of toric relations that generalize
\eqref{eq:torsegre}.  We call a reduced weighting of $\Xi^-$ of degree
$\le 1$ a \emph{reduced matching}.  Recall that the set $\ms{V}_\Xi$ of
reduced matchings generates $\ms{R}_\Xi$ (Proposition~\ref{prop:ytree-gen}).
Let $X$, $Y$ and $Z$ be reduced matchings on the $r$th caterpillar tree such
that $X$ and $Y$ take value 1 on the $r$th stalk and $Z$ takes value 0 on the
$r$th stalk.  Let $X'$, $Y'$ and $Z'$ be reduced matchings on the $s$th
caterpillar tree such that $X'$ and $Y'$ take value 1 on the first stalk and
$Z'$ takes value 0 on the first stalk.  We then have the relation 
\Ncom{$\text{reduced matching }  \ms{V}_\Xi$}
\begin{displaymath}
\begin{split}
&
\begin{xy}
(-8, -3)*{}="A"; (0, -3)*{}="B"; (8, -3)*{}="C"; (0, 2)*{}="D";
"A"; "B"; **\dir{-};
"B"; "C"; **\dir{-};
"B"; "D"; **\dir{-};
(-4, -3)*{1}; (0, 4)*{1}; (4, -3)*{1};
(-8, -6); (-8, 0); **\dir{-};
(-8, 0); (-14, 0); **\dir{-};
(-14, 0); (-14, -6); **\dir{-};
(-14, -6); (-8, -6); **\dir{-};
(-11, -3)*{X};
(8, -6); (8, 0); **\dir{-};
(8, 0); (14, 0); **\dir{-};
(14, 0); (14, -6); **\dir{-};
(14, -6); (8, -6); **\dir{-};
(11, -3)*{X'};
\end{xy}
\qquad
\begin{xy}
(-8, -3)*{}="A"; (0, -3)*{}="B"; (8, -3)*{}="C"; (0, 2)*{}="D";
"A"; "B"; **\dir{-};
"B"; "C"; **\dir{-};
"B"; "D"; **\dir{-};
(-4, -3)*{1}; (0, 4)*{1}; (4, -3)*{1};
(-8, -6); (-8, 0); **\dir{-};
(-8, 0); (-14, 0); **\dir{-};
(-14, 0); (-14, -6); **\dir{-};
(-14, -6); (-8, -6); **\dir{-};
(-11, -3)*{Y};
(8, -6); (8, 0); **\dir{-};
(8, 0); (14, 0); **\dir{-};
(14, 0); (14, -6); **\dir{-};
(14, -6); (8, -6); **\dir{-};
(11, -3)*{Y'};
\end{xy}
\qquad
\begin{xy}
(-8, -3)*{}="A"; (0, -3)*{}="B"; (8, -3)*{}="C"; (0, 2)*{}="D";
"A"; "B"; **\dir{-};
"B"; "C"; **\dir{-};
"B"; "D"; **\dir{-};
(-4, -3)*{0}; (0, 4)*{0}; (4, -3)*{0};
(-8, -6); (-8, 0); **\dir{-};
(-8, 0); (-14, 0); **\dir{-};
(-14, 0); (-14, -6); **\dir{-};
(-14, -6); (-8, -6); **\dir{-};
(-11, -3)*{Z};
(8, -6); (8, 0); **\dir{-};
(8, 0); (14, 0); **\dir{-};
(14, 0); (14, -6); **\dir{-};
(14, -6); (8, -6); **\dir{-};
(11, -3)*{Z'};
\end{xy}
\\ = \qquad &
\begin{xy}
(-8, -3)*{}="A"; (0, -3)*{}="B"; (8, -3)*{}="C"; (0, 2)*{}="D";
"A"; "B"; **\dir{-};
"B"; "C"; **\dir{-};
"B"; "D"; **\dir{-};
(-4, -3)*{1}; (0, 4)*{1}; (4, -3)*{0};
(-8, -6); (-8, 0); **\dir{-};
(-8, 0); (-14, 0); **\dir{-};
(-14, 0); (-14, -6); **\dir{-};
(-14, -6); (-8, -6); **\dir{-};
(-11, -3)*{X};
(8, -6); (8, 0); **\dir{-};
(8, 0); (14, 0); **\dir{-};
(14, 0); (14, -6); **\dir{-};
(14, -6); (8, -6); **\dir{-};
(11, -3)*{Z'};
\end{xy}
\qquad
\begin{xy}
(-8, -3)*{}="A"; (0, -3)*{}="B"; (8, -3)*{}="C"; (0, 2)*{}="D";
"A"; "B"; **\dir{-};
"B"; "C"; **\dir{-};
"B"; "D"; **\dir{-};
(-4, -3)*{0}; (0, 4)*{1}; (4, -3)*{1};
(-8, -6); (-8, 0); **\dir{-};
(-8, 0); (-14, 0); **\dir{-};
(-14, 0); (-14, -6); **\dir{-};
(-14, -6); (-8, -6); **\dir{-};
(-11, -3)*{Z};
(8, -6); (8, 0); **\dir{-};
(8, 0); (14, 0); **\dir{-};
(14, 0); (14, -6); **\dir{-};
(14, -6); (8, -6); **\dir{-};
(11, -3)*{X'};
\end{xy}
\qquad
\begin{xy}
(-8, -3)*{}="A"; (0, -3)*{}="B"; (8, -3)*{}="C"; (0, 2)*{}="D";
"A"; "B"; **\dir{-};
"B"; "C"; **\dir{-};
"B"; "D"; **\dir{-};
(-4, -3)*{1}; (0, 4)*{0}; (4, -3)*{1};
(-8, -6); (-8, 0); **\dir{-};
(-8, 0); (-14, 0); **\dir{-};
(-14, 0); (-14, -6); **\dir{-};
(-14, -6); (-8, -6); **\dir{-};
(-11, -3)*{Y};
(8, -6); (8, 0); **\dir{-};
(8, 0); (14, 0); **\dir{-};
(14, 0); (14, -6); **\dir{-};
(14, -6); (8, -6); **\dir{-};
(11, -3)*{Y'};
\end{xy}
\end{split}
\end{displaymath}
We call these  \emph{toric generalized Segre cubic} relations.
\Ncom{$\text{gen.\ t. S c rels}$}

\subsection{The type vector of a triple}

Let $\ul{\xi}=(\xi_1, \xi_2, \xi_3)$ be a triple of reduced matchings on the
$r$th caterpillar tree.  Define the \emph{type} of $\xi$ at the $i$th base
vertex to be one of $A$, $B$ or $\emptyset$, as follows.  We call $\ul{\xi}$
type $A$ at $i$ if it looks like 
\Ncom{$\text{type of triple}$}
\begin{displaymath}
\begin{xy}
(-8, -3)*{}="A"; (0, -3)*{}="B"; (8, -3)*{}="C"; (0, 2)*{}="D";
"A"; "B"; **\dir{-};
"B"; "C"; **\dir{-};
"B"; "D"; **\dir{-};
(-4, -3)*{1}; (0, 4)*{1}; (4, -3)*{1};
\end{xy}
\qquad
\begin{xy}
(-8, -3)*{}="A"; (0, -3)*{}="B"; (8, -3)*{}="C"; (0, 2)*{}="D";
"A"; "B"; **\dir{-};
"B"; "C"; **\dir{-};
"B"; "D"; **\dir{-};
(-4, -3)*{1}; (0, 4)*{1}; (4, -3)*{1};
\end{xy}
\qquad
\begin{xy}
(-8, -3)*{}="A"; (0, -3)*{}="B"; (8, -3)*{}="C"; (0, 2)*{}="D";
"A"; "B"; **\dir{-};
"B"; "C"; **\dir{-};
"B"; "D"; **\dir{-};
(-4, -3)*{0}; (0, 4)*{0}; (4, -3)*{0};
\end{xy}
\end{displaymath}
at the $i$th vertex (the order of the triple is irrelevant).  We call
$\ul{\xi}$ type $B$ at $i$ if it looks like
\begin{displaymath}
\begin{xy}
(-8, -3)*{}="A"; (0, -3)*{}="B"; (8, -3)*{}="C"; (0, 2)*{}="D";
"A"; "B"; **\dir{-};
"B"; "C"; **\dir{-};
"B"; "D"; **\dir{-};
(-4, -3)*{1}; (0, 4)*{1}; (4, -3)*{0};
\end{xy}
\qquad
\begin{xy}
(-8, -3)*{}="A"; (0, -3)*{}="B"; (8, -3)*{}="C"; (0, 2)*{}="D";
"A"; "B"; **\dir{-};
"B"; "C"; **\dir{-};
"B"; "D"; **\dir{-};
(-4, -3)*{0}; (0, 4)*{1}; (4, -3)*{1};
\end{xy}
\qquad
\begin{xy}
(-8, -3)*{}="A"; (0, -3)*{}="B"; (8, -3)*{}="C"; (0, 2)*{}="D";
"A"; "B"; **\dir{-};
"B"; "C"; **\dir{-};
"B"; "D"; **\dir{-};
(-4, -3)*{1}; (0, 4)*{0}; (4, -3)*{1};
\end{xy}
\end{displaymath}
there.  In all other cases we call $\ul{\xi}$ type $\emptyset$ at $i$.  We
define the \emph{type vector} of $\ul{\xi}$, denoted $t(\ul{\xi})$, to be the
ordered tuple of the $r-2$ types of $\ul{\xi}$.  The type of a triple is a
quadratic invariant --- if $\ul{\xi} \sim \ul{\xi}'$ is an essentially
quadratic relation between triples, then $t(\ul{\xi})=t(\ul{\xi}')$.   The
toric generalized Segre cubic changes the type.
\Ncom{$\text{type vec.}$}

\subsection{Reformulation of Theorem~\ref{thm:ytree}}

Rather than proving the statement about relations in Theorem~\ref{thm:ytree}
directly, we will prove the following: 
\Ncom{$\text{ideal of relations} I_{\Xi}$}

\begin{proposition}
\label{prop:ytree-rel}
Let $\Xi$ be the $r$th caterpillar tree ($r \ge 3$).  Then the ideal of
relations $I_{\Xi}$ of $\gr_{\Xi} R_L=\Z[\ms{R}_{\Xi}]$ is
generated by relations of degree two and three.  Furthermore, ``type'' is the
only quadratic invariant on $\ms{V}_{\Xi}^3$, that is, if $\ul{\xi} \sim
\ul{\xi}'$ is a relation with $\ul{\xi}, \ul{\xi}' \in \ms{V}_{\Xi}^3$ then
$\ul{\xi} \sim \ul{\xi}'$ is essentially quadratic if and only if
$t(\ul{\xi})=t(\ul{\xi}')$.
\end{proposition}

Theorem~\ref{thm:ytree} follows easily from this, since generalized toric
Segre cubics allow one to switch the type of a triple between $A$
and $B$ at any base vertex.

We now give an overview of the proof of Proposition~\ref{prop:ytree-rel}.
First of all we split relations into two types defined in \S
\ref{ss:noncritical}, ``breakable'' and ``unbreakable.''  Breakable means that
some degree one piece of the relation vanishes at some base edge.  The first
idea in the proof is that, given a relation, one can chop the tree into pieces
such that the relation is unbreakable on each piece.  One can then use a
gluing argument to reduce to the unbreakable case.  We then prove that an
unbreakable relation is essentially quadratic.  We do this by introducing a
normal form for monomials and showing that any monomial can be changed into
normal form by a series of quadratic relations.  We use the notion of a
``balanced'' monomial, a tuple of matchings which assume roughly the same value
on each edge of the tree.  A key result (Proposition  \ref{keyresult}) is that
any tuple can be balanced by quadratic relations.
\Ncom{$\text{breakable rel.}$}

\subsection{Balancing}

We say that a tuple of integers $(x_i)$ is \emph{balanced} if $\vert x_i - x_j
\vert$ is always 0 or 1.  We say that a tuple $\ul{\xi} \in \ms{V}^n$ is
\emph{balanced} if for each base edge $e$ of $\Xi$ the tuple of integers
$(\xi_i(e))$ is balanced.
\Ncom{$\text{balanced tuple}$}

\begin{proposition}
Given any tuple $\ul{\xi}$ in $\mc{V}^n$ there exists an essentially quadratic
relation $\ul{\xi} \sim \ul{\xi}'$ with $\ul{\xi}'$ balanced. \label{keyresult}
\end{proposition}

\begin{proof}
It suffices to prove the proposition for $n=2$, as one can repeatedly balance
pairs of integers to balance a set of integers.  Thus suppose that $n=2$.
First we show that the proposition holds for the third caterpillar tree.  In
the general case of the $r$th caterpillar, we break $\Xi$ into its trinodes,
balance the weightings on each of these separately and then ``glue.''

{\it Third caterpillar:} We indicate a reduced matching $\xi$ on the third
caterpillar tree by a triple $(a,b,c)$ where $a$ is the weight of the first
(left) stalk, $b$ is the weight of the second (vertical) stalk and $c$ is the
weight of the third (right) stalk.  The triple $(a,b,c)$ satisfies the triangle
inequalities and has $b \le 1$.  If $b=0$ then $a=c$.  If $b=1$ then one of $a$
or $c$ is nonzero, and $\vert a-c \vert \le 1$.  Suppose $\xi_1=(a_1, b_1,
c_1)$ and $\xi_2=(a_2, b_2, c_2)$ are reduced matchings and $\vert a_1-a_2
\vert \ge 2$ or $\vert c_1-c_2 \vert \ge 2$.  Without loss of generality we
take $a_1 + 2 \le a_2$.  Since $c_1 \le a_1+1$ and $c_2 \ge a_2-1$ we have
$c_1 \le c_2$.  We know that $c_2 \ge 1$ since $a_2 \ge 2$.  

If $c_1<c_2$ then define $\xi_1'=(a_1+1, b_1, c_1+1)$ and $\xi_2'=(a_2-1, b_2,
c_2-1)$.  Then $\xi_1'+\xi_2'=\xi_1+\xi_2$ so $(\xi_1',\xi_2') \sim
(\xi_1, \xi_2)$.  The resulting pair $(\xi_1', \xi_2')$ is now closer to being
balanced since $\vert (a_1+1) - (a_2-1) \vert = \vert a_1-a_2 \vert-2$
and $\vert (c_1+1) - (c_2-1) \vert \le \vert c_1-c_2 \vert$.  

If $c_1 = c_2$ then $a_2 = a_1+2$ and $c_1 = c_2 = a_1 + 1$.  Define 
$\xi_1' = (a_1+1,b_1,c_1)$ and $\xi_2' = (a_2-1,b_2,c_2)$.  Then $\xi_1'$ and
$\xi_2'$ are reduced admissible matchings and $(\xi_1',\xi_2') \sim (\xi_1,
\xi_2)$.  Again, we get strictly closer to a balanced pair with such an
assignment.

A finite number of steps as above will related the original pair $(\xi_1,
\xi_2)$ to a balanced pair.

{\it Breaking and gluing:}  Suppose now that we have a pair $(\xi_1,\xi_2)$
of reduced admissible matchings on the $r$th caterpillar.  Break the
caterpillar up into an ordered tuple of $r-2$ caterpillar trees with three
vertices by cutting each base edge in two.  These are arranged from left to
right and indexed as $2, \ldots, r-1$.  The matchings $\xi_i$ on the original
caterpillar define matchings on each copy of the third caterpillar.  Let $e_j$,
$f_j$ and $g_j$ denote the first, second and third stalk on the $j$th trinode
and define $\xi_{i,j} = (\xi_i(e_j), \xi_i(f_j), \xi_i(g_j))$ for $i=1, 2$ and
$2 \le j \le r-1$.  Now apply our previous result on the third caterpillar ---
each pair $(\xi_{1,j},\xi_{2,j})$ is equivalent to a balanced pair
$(\xi_{1,j}', \xi_{2,j}')$.    

Define $t_2=1$.  Now, because the pairs are balanced, for $3 \le j \le r-1$
we have either $c_{1,j}'=a_{1,j+1}'$ and $c_{2,j}'=a_{2,j+1}'$, or
$c_{1,j}'=a_{2,j+1}'$ and $c_{2,j}' = a_{1,j+1}'$.  In the former case set
$t_j=t_{j-1}$ while in the latter case set $t_j=3-t_{j-1}$.  We will use the
tuple $(t_2,\ldots,t_{r-1})$ of 1's and 2's to glue these balanced pairs on 
individual trinodes to obtain a balanced pair on the $r$th caterpillar. 

We now define a pair of admissible weightings $(\xi_1',\xi_2')$ on the $r$th
caterpillar by a gluing procedure.  Let the edges of the $r$th caterpillar be
labeled as $s_1$, $s_2$, $b_2$, $s_3$, $b_3$, \ldots, $b_{r-2}$, $s_{r-1}$,
$s_{r}$, where $s_j$ is the $j$th stalk and $b_k$ is the base edge between the
$k$th and $(k+1)$st base vertices.  To begin, we define $\xi_1'$
(resp.\ $\xi_2'$) on $s_1, s_2, b_2$ to agree with $\xi_{1,2}'$
(resp.\ $\xi_{2,2}'$).  For $3 \le j \le r-3$ set
\begin{displaymath}
(\xi_1'(b_{j-1}),\xi_1'(s_j),\xi_1'(b_j)) = \xi_{t_j}', 
\quad (\xi_2'(b_{j-1}),\xi_2'(s_j),\xi_2'(b_j)) = \xi_{2-t_j}'.
\end{displaymath}
Finally define 
\begin{displaymath}
(\xi_1'(b_{r-2}),\xi_1'(s_{r-1}),\xi_1'(s_r)) = \xi_{t_{r-1}}', 
\quad (\xi_2'(b_{r-2}),\xi_2'(s_{r-1}),\xi_2'(s_{r})) = \xi_{2-t_{r-1}}'.
\end{displaymath}
The above assignments are well-defined, $(\xi_1',\xi_2')$ is balanced and
$(\xi_1',\xi_2') \sim (\xi_1, \xi_2)$.
\end{proof}

\subsection{Reduction of Proposition~\ref{prop:ytree-rel} to the unbreakable
case}
\label{ss:noncritical}

We say that a matching $\xi$ on $\Xi$ is \emph{breakable} if there exists
a base edge $e$ with $\xi(e)=0$, and \emph{unbreakable} otherwise.  We say that
a tuple of matchings $\ul{\xi} \in \ms{V}_{\Xi}^n$ is \emph{unbreakable} if
each $\xi_i$ is.  In \S \ref{ss:non-crit} we will prove the following
proposition: 
\Ncom{$\text{breakable}$}

\begin{proposition}
\label{prop:non-crit}
Let $\ul{\xi} \sim \ul{\xi}'$ be a relation with $\ul{\xi}, \ul{\xi}' \in
\ms{V}_{\Xi}^n$ and $\ul{\xi}$ unbreakable.  Then  $\ul{\xi} \sim \ul{\xi}'$
is essentially quadratic.
\end{proposition}

In this section we prove the following:

\begin{proposition}
Proposition~\ref{prop:non-crit} implies Proposition~\ref{prop:ytree-rel}.
\end{proposition}

\begin{proof}
As remarked in \S \ref{ss:tor-segre},
Proposition~\ref{prop:ytree-rel} is true for the third caterpillar tree.  This
will be the base case of an inductive argument.

Let $x \sim y$ be a relation of length $n$.  Using quadratic relations, we
may assume that both $x$ and $y$ are balanced.  If each $x_i$ is unbreakable
then each $y_i$ is as well (since $x$ and $y$ are balanced) and we are done.
Assume then that there is a base edge $e$ for which $x_i$ is breakable at $e$
for some $i$.

Cut the edge $e$ in half to produce two new trees $\Xi'$ and $\Xi''$.  We
regard $e$ as an edge of both of these trees.  Both of these trees can be
regarded as smaller caterpillar trees.  Also, giving a weighting on $\Xi$ is
equivalent to giving weightings on $\Xi'$ and $\Xi''$
which agree at $e$.

Now, let $x'$ and $x''$ be the restrictions of $x$ to $\Xi'$ and $\Xi''$ (and
similarly for $y$).  The key point is that because $e$ is breakable and $x$ and
$y$ are balanced these restricted weightings are \emph{matchings}, that is,
they assign $e$ either 0 or 1.  In other words, we have $x', y' \in
\mc{V}_{\Xi'}^n$ and $x'', y'' \in \mc{V}_{\Xi''}^n$.

Proceeding by induction, we can assume that all relations are essentially
cubic on $\Xi'$ and $\Xi''$, or essentially quadratic if the types agree.  We can then pick
a sequence of cubic (resp.\ quadratic) relations between $x'$ and $y'$ and
between $x''$ and $y''$ and concatenate them to form a sequence of cubic
(resp.\ quadratic) relations between $x$ and $y$.   By ``concatenate,'' we mean
that one should first order the tuples so that those taking value zero at the
edge $e$ should be glued together (the order does not matter), and those taking
value one at $e$ should be glued together (the order does not matter).  In the
final step --- that is, after $x'$ has been replaced with $y'$ (up to
permutation)
and $x''$ has been replaced with $y''$ (up to permutation) --- one can finally
permute the $y''$ matchings taking value zero at $e$, and permute those taking
value one at $e$, and lastly permute the concatenated matchings, so that the
result is equal to $y$. (Recall that permutations are essentially quadratic
since they are generated by 2-cycles.)
\end{proof}

\subsection{Proof of Proposition~\ref{prop:non-crit}}
\label{ss:non-crit}

Let $\xi$ be an unbreakable matching.  There are four possibilities for $\xi$
at one of the internal trinodes.  We label them as follows:
\begin{displaymath}
\begin{xy}
(0, 0); (16, 0); **\dir{-};
(8, 0); (8, 8); **\dir{-};
(10, 4)*{0}; (4, -2)*{n}; (12, -2)*{n};
(8, -5)*{A_n};
\end{xy}
\qquad
\begin{xy}
(0, 0); (16, 0); **\dir{-};
(8, 0); (8, 8); **\dir{-};
(10, 4)*{1}; (4, -2)*{n}; (12, -2)*{n};
(8, -5)*{B_n};
\end{xy}
\qquad
\begin{xy}
(0, 0); (16, 0); **\dir{-};
(8, 0); (8, 8); **\dir{-};
(10, 4)*{1}; (4, -2)*{n}; (12, -2)*{n+1};
(8, -5)*{C_n};
\end{xy}
\qquad
\begin{xy}
(0, 0); (16, 0); **\dir{-};
(8, 0); (8, 8); **\dir{-};
(10, 4)*{1}; (4, -2)*{n+1}; (12, -2)*{n};
(8, -5)*{D_n};
\end{xy}
\end{displaymath}
In all cases $n$ is non-zero.  We order them as follows: $A_n \le B_n \le C_n
\le D_n \le A_{n+1}$.  There are also four possibilities for $\xi$ at the
trinodes on the left end:
\begin{displaymath}
\begin{xy}
(0, 0); (0, 16); **\dir{-};
(0, 8); (8, 8); **\dir{-};
(-2, 4)*{0}; (-2, 12)*{1}; (4, 11)*{1};
(4, -3)*{E};
\end{xy}
\qquad
\begin{xy}
(0, 0); (0, 16); **\dir{-};
(0, 8); (8, 8); **\dir{-};
(-2, 4)*{1}; (-2, 12)*{0}; (4, 11)*{1};
(4, -3)*{F};
\end{xy}
\qquad
\begin{xy}
(0, 0); (0, 16); **\dir{-};
(0, 8); (8, 8); **\dir{-};
(-2, 4)*{1}; (-2, 12)*{1}; (4, 11)*{1};
(4, -3)*{G};
\end{xy}
\qquad
\begin{xy}
(0, 0); (0, 16); **\dir{-};
(0, 8); (8, 8); **\dir{-};
(-2, 4)*{1}; (-2, 12)*{1}; (4, 11)*{2};
(4, -3)*{H};
\end{xy}
\end{displaymath}
The right end is similar (flip each over).  We order these:  $E \le F \le G
\le H$. 

We can express an unbreakable matching as a string using the above alphabet and
some obvious rules.  Let $\xi$ and $\xi'$ be two unbreakable matchings.  We
define a partial ordering on unbreakable matchings, by $\xi \le \xi'$ if this
is the case when restricted to each trinode.  Equivalently, $\xi \le \xi'$ if
the following holds:
\begin{itemize}
\item Both ends of $\xi$ are at most  the respective ends of $\xi'$. 
\item For each base edge $e$ we have $\xi(e) \le \xi'(e)$.
\item If $e$ and $e'$ are consecutive base edges for which $\xi(e)=\xi'(e)$
and $\xi(e')=\xi'(e')$ then $\xi(s) \le \xi'(s)$ where $s$ is the stalk in
between $e$ and $e'$.
\end{itemize}
We say that an unbreakable tuple of matchings $\ul{\xi} \in \ms{V}_{\Xi}^n$ is
in \emph{normal form} if $\ul{\xi}$ is balanced and $\ul{\xi}_1 \le \ul{\xi}_2
\le \cdots \le \ul{\xi}_n$.  Proposition~\ref{prop:non-crit} now follows from
the following proposition:

\begin{proposition}
\label{prop-norm}
Every unbreakable element of $\ms{V}_{\Xi}^n$ is related to a unique normal
form, and this relation is essentially quadratic.
\end{proposition}

The following is the key lemma:

\begin{lemma}
\label{lemma-key}
Let $(\xi, \xi') \in \ms{V}_{\Xi}^2$ be unbreakable and balanced.  Then there
exists $(\eta, \eta') \in \ms{V}_{\Xi}^2$ balanced and unbreakable such that
$\xi+\xi'=\eta +\eta'$ and for each internal edge $e$ we have
\begin{displaymath}
\eta(e)=\min(\xi(e), \xi'(e)), \qquad \eta'(e)=\max(\xi(e), \xi'(e)).
\end{displaymath}
\end{lemma}

\begin{proof}
The idea is the same as in the proof of Proposition~\ref{keyresult}: break apart at a base edge (or at
an end), permute, and glue back together to achieve the desired order.
\end{proof}

We now prove Proposition~\ref{prop-norm}.

\begin{proof}[Proof of Proposition~\ref{prop-norm}]
Let $\ul{\xi} \in \ms{V}_{\Xi}^n$ be a given unbreakable tuple.  We may
assume that $\ul{\xi}$ is balanced.  By repeatedly using Lemma~\ref{lemma-key}
we find that $\ul{\xi}$ is quadratically related to an unbreakable element
$\ul{\xi}' \in \ms{V}_{\Xi}^n$ which has the property that $\xi'_i(e) \le
\xi'_j(e)$ for $i \le j$ and all internal edges $e$.  

For uniqueness, it suffices to show uniqueness for third caterpillars and for
the ends.  Consider the third caterpillar at a base vertex.  Suppose the sum
of the $n$ weightings $(a_i,b_i,c_i)$, $1 \le i \le n$, on the third
caterpillar is equal to $(a,b,c)$ (left, stalk, right).  Suppose these
weightings are increasing in the order we have defined.  This means the $a_i$'s
are increasing, starting as the floor of $a/n$ and ending as the ceiling
of $a/n$.  This determines the value of each $a_i$.  Similarly the value of
each $c_i$ is determined (beginning with floor of $c/n$, ending with ceiling
of $c/n$).  Wherever $a_i \ne c_i$ we must have $b_i = 1$, so these $b_i$'s
are determined.  However if $a_i = c_i$, then $b_i$ could be either $0$ or
$1$.  The set $\{i \mid a_i = c_i\}$ consists of at most two intervals $I, J$.
These intervals are determined by the $a_i, c_i$ (which are determined by the
normality condition).  Within the interval $I$, the $b_i$'s must be increasing.
Similarly in the interval $J$, the $b_i$'s must increase.  Thus the $b_i$'s
are all determined by the normality condition.  The argument for the ends is
similar.
\end{proof}


\section{The ideal is generated by quadratics and generalized Segre cubics}
\label{s:gsc}

In \S \ref{s:gsc}, we lift the generalized Segre toric cubics to the
ring $R_L$ and prove the following:

\begin{theorem}
\label{thm:gsc}
For any even set $L$ the ideal $I_L$ is generated over $\Z$ by quadratics and
the small generalized Segre cubic relations.
\end{theorem}

\Rnts{used in \ref{thm:retro}}

We will introduce the generalized Segre cubics, and the small generalized
Segre cubics, shortly.  Theorem~\ref{thm:gsc} will follow easily from our toric
results once we make these definitions.  In one of our ad hoc arguments in
the 10 point case we will need a slightly refined version of
Theorem~\ref{thm:gsc} given in Remark~\ref{rem:gsc}.  The proof of this theorem
is the only place we use the toric results.

\subsection{Brief additional comments on colored graphs}

We will use the language of colored graphs, introduced in \S \ref{s:intro}.
We consider both directed and undirected colored graphs.  If $\Gamma$ is a
directed multi-matching on $L$ whose edges have been colored with colors from
the set $C$ then $X_{\Gamma}$ is defined as an element of $V_L^{\otimes C}$
as in \S \ref{s:intro}.  If $L$ is oriented and $\Gamma$ is a regular
undirected multi-matching one can also make sense of $Y_{\Gamma}$ as an element
of $V_L^{\otimes C}$.  Clearly the $X_{\Gamma}$ (or $Y_{\Gamma}$) span
$V_L^{\otimes C}$.  Furthermore, the $X_{\Gamma}$ satisfy the sign and
``colored'' Pl\"ucker relations, and these generate all the linear relations
among them.  (The colored Pl\"ucker relation is just the usual Pl\"ucker
relation on a pair of edges, with the restriction that these two edges be of
the same color.)  We have thus given a description of $V_L^{\otimes C}$ in
terms of colored graphs.  There is a similar description for $\Sym^k(V_L)$.
The only difference is that in $\Sym^k(V_L)$ the particular color of an edge
is not relevant.  What matters is whether two edges have the same color --- two
colored graphs represent the same element of $\Sym^k(V_L)$ if one is obtained
from the other by permuting the colors.

\subsection{Generalized Segre cubic data}

Let $L$ be an even set.  By a \emph{generalized Segre datum} we mean a pair
$\Sigma=(\Gamma, \ms{U})$ where $\Gamma$ is an undirected graph on $L$
whose edges have been colored one of red, green or blue and $\ms{U}=\{U_R, U_G,
U_B\}$ is a partition of $L$ into three even subsets (called \emph{parts}),
such that
\begin{itemize}
\item Every vertex of $\Gamma$ has valence one for each of the three colors.
\item Between any two parts there are either two edges or no edges.  If there
are two edges then these edges have the color of the ``opposite'' part.
For instance, any edge from $U_G$ to $U_B$ must be red.
\end{itemize}
See Figure~\ref{f:genseg} for a schematic presentation.  We call the edges
between the parts \emph{special}.  We call $\Sigma$ \emph{small} if $U_R$,
$U_G$ or $U_B$ has cardinality two.
\Ncom{$\text{gen.\ S. datum, special edges, small}$}

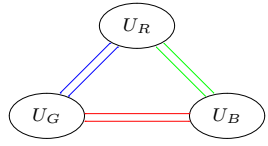
\begin{figure}[!ht]
\begin{displaymath}
\begin{xy}
(0, 0)*\xycircle(5, 3){-};
(-12, -12)*\xycircle(5, 3){-};
(12, -12)*\xycircle(5, 3){-};
(0, 0)*{\ss U_R}; (-12, -12)*{\ss U_G}; (12, -12)*{\ss U_B};
(-.5, .75)*{}="A";
{\ar@[blue]@{-} "A"+(-9.7, -9.7); "A"+(-3.1, -3.1)};
(-.5, -.75)*{}="A";
{\ar@[blue]@{-} "A"+(-8.6, -8.6); "A"+(-2.1, -2.1)};
(.5, .75)*{}="A";
{\ar@[green]@{-} "A"+(9.7, -9.7); "A"+(3.1, -3.1)};
(.5, -.75)*{}="A";
{\ar@[green]@{-} "A"+(8.6, -8.6); "A"+(2.1, -2.1)};
(-12, -11.7)*{}="A";
{\ar@[red]@{-} "A"+(5.1, 0); "A"+(18.9, 0)};
(-12, -12.7)*{}="A";
{\ar@[red]@{-} "A"+(5.1, 0); "A"+(18.9, 0)};
\end{xy}
\end{displaymath}
\caption{Schematic presentation of a generalized Segre datum.}
\label{f:genseg}
\end{figure}

Let $\Sigma$ be a generalized Segre datum.  Suppose $e$ and $e'$ are edges of
$\Gamma$ of the same color and have a vertex in a common part.  Given such a
pair, let $Y_{\Gamma}+Y_{\Gamma'}+Y_{\Gamma''}=0$ be the colored Pl\"ucker
relation on $e$ and $e'$.  Then it is easily verified that $\Sigma'=(\Gamma',
\ms{U})$ and $\Sigma''=(\Gamma'', \ms{U})$ are both generalized Segre data.
Define the \emph{space of generalized Segre data} to be the $\Z$-module spanned
by generalized Segre data modulo relations of the form $\Sigma+\Sigma'+\Sigma''
=0$.

\subsection{Generalized Segre cubic relations}

Let $\Sigma=(\Gamma, \ms{U})$ be a generalized Segre datum.  Let $\wt{\Gamma}$
be a directed colored graph with underlying undirected graph $\Gamma$.  We let
$\epsilon(\wt{\Gamma})$ be the product of the $\epsilon(\wt{\Gamma}_i)$, where
$\wt{\Gamma}_i$ is the directed matching of color $i$ in $\wt{\Gamma}$.  (Here
$\epsilon$ is a chosen orientation on $L$, see \S \ref{ss:orient}.)
Form a new directed colored graph $\wt{\Gamma}'$ as
follows.  The graph $\wt{\Gamma}'$ will be a recoloring of $\wt{\Gamma}$, so we
just specify a new color for each edge.  We  use the colors purple
and black.  The red (resp.\ blue, green) edges of $\wt{\Gamma}$ in $U_R$
(resp.\ $U_B$, $U_G$) are black in $\wt{\Gamma}'$ and all other edges of
$\wt{\Gamma}$ are purple in $\wt{\Gamma}'$.  It is clear that every vertex has
black valence one and thus  purple valence two.

Define $Y_{\Sigma} \in \Sym^3(V_L)$ and $Y_{\Sigma}' \in R_L^{(1)} \otimes
R_L^{(2)}$ by
\begin{displaymath}
Y_{\Sigma}=Y_{\Gamma}=\epsilon(\wt{\Gamma}) X_{\wt{\Gamma}}, \qquad
Y_{\Sigma}'=\epsilon(\wt{\Gamma}) \left( X_{\wt{\Gamma}'_{\mathrm{black}}}
\otimes X_{\wt{\Gamma}'_{\mathrm{purple}}} \right).
\end{displaymath}
These only depend on $\Sigma$ and not the choice of $\wt{\Gamma}$.  There is a
well-defined map $R_L^{(1)} \otimes R_L^{(2)} \to \Sym^3(V_L)/Q_L^{(3)}$
given by writing the element of $R_L^{(2)}$ in terms of degree one elements and
then formally multiplying to get an element of $\Sym^3(V_L)$.  (This is
only defined modulo quadratic relations $Q_L$, because of the choice of how to write $R_L^{(2)}$ in
terms of degree one elements.)  We may therefore regard both $Y_{\Sigma}$ and
$Y_{\Sigma}'$ as elements of $\Sym^3(R_L^{(1)})/Q_L$.  Define
\begin{displaymath}
\Rel(\Sigma)=Y_{\Sigma}-Y_{\Sigma}',
\end{displaymath}
regarded as an element of $\Sym^3(V_L)/Q_L^{(3)}$.  As the two terms in
$\Rel(\Sigma)$ are recolorings of the same graph, $\Rel(\Sigma)$ is a relation,
that is, it maps to zero in $R_L$.  We call such relations \emph{generalized
Segre cubic relations}.  We use the same name for lift of some $\Rel(\Sigma)$
to  $I^{(3)}_L \subset \Sym^3(V_L)$.  A generalized Segre cubic relation is
shown in Figure~\ref{f:genseg6}.  Note that for $\sigma \in
\mf{S}_L$ we have $\sigma \Rel(\Sigma)=\sgn(\sigma) \Rel(\sigma \Sigma)$.
\Ncom{$\text{gen.\ S. c. rels.}$}

\begin{figure}[!ht]
\begin{displaymath}
\begin{xy}(-10, -12)*{}="A1"; (-8, -3)*{}="A2";
(0, 3)*{}="B1"; (2, 12)*{}="B2";
(10, -12)*{}="C1"; (12, -3)*{}="C2";
{\ar@[blue]@{..}@[|(2)] "A1"; "B1"};
{\ar@[red]@{--} "A1"; "C1"};
{\ar@[green]@{-} "B1"; "C1"};
{\ar@[blue]@{..}@[|(2)] "A2"; "B2"};
{\ar@[red]@{--} "A2"; "C2"};
{\ar@[green]@{-} "B2"; "C2"};
{\ar@[blue]@{..}@[|(2)] "C1"; "C2"};
{\ar@[red]@{--} "B1"; "B2"};
{\ar@[green]@{-} "A1"; "A2"};
(-12, -12)*{\ss 2}; (-10, -3)*{\ss 1};
(2, 3)*{\ss 4}; (4, 12)*{\ss 3};
(12, -12)*{\ss 6}; (14, -3)*{\ss 5};
\end{xy}
\qquad = \qquad
\begin{xy}
(-10, -12)*{}="A1"; (-8, -3)*{}="A2";
(0, 3)*{}="B1"; (2, 12)*{}="B2";
(10, -12)*{}="C1"; (12, -3)*{}="C2";
{\ar@[Thistle]@{-} "A1"; "B1"};
{\ar@[Thistle]@{-} "A1"; "C1"};
{\ar@[Thistle]@{-} "B1"; "C1"};
{\ar@[Thistle]@{-} "A2"; "B2"};
{\ar@[Thistle]@{-} "A2"; "C2"};
{\ar@[Thistle]@{-} "B2"; "C2"};
{\ar@[black]@{-}@[|(3)] "C1"; "C2"};
{\ar@[black]@{-}@[|(3)] "B1"; "B2"};
{\ar@[black]@{-}@[|(3)] "A1"; "A2"};
\end{xy}
\end{displaymath}
\caption{A generalized Segre relation on six points.  Here $U_G=\{1, 2\}$,
$U_R=\{3, 4\}$ and $U_B=\{5, 6\}$.  This is equal to the image of the usual
Segre cubic relation in $I_L/Q_L$.}
\label{f:genseg6}
\end{figure}
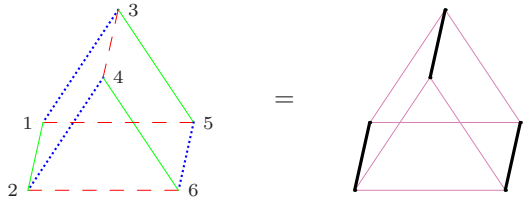

One easily verifies that $\Rel$ gives a homomorphism
\begin{displaymath}
\Rel:\{ \textrm{the space of generalized Segre cubic data} \} \to
I_L^{(3)}/Q_L^{(3)}.
\end{displaymath}
This allows us to interpret relations as graphs with extra structure, which
has two advantages.  First, it gives a source of linear relations between
(also thought of as operations we can perform on) generalized Segre relations:
Pl\"ucker relations on the generalized Segre data which respect the extra
structure.  Second, one can read off certain properties of the generalized
Segre relation from the original graph.  For instance, if $\Sigma$ is a
generalized Segre cubic datum for which $\Gamma$ is disconnected, then the
associated generalized Segre relation arises from a relation on fewer points
(see Proposition~\ref{prop:gsd-discon} for a precise statement).

\subsection{Degenerate Segre cubic relations}
\label{ss:gsd-degen}

We call a generalized Segre datum $\Sigma$ \emph{degenerate} if:
\begin{enumerate}
\item one of the pairs of special edges is missing; or
\item in one of the parts the two pairs of special edges do not connect.
\end{enumerate}
By (2) we mean that one part, say $U_G$, can be partitioned into two pieces
$A$ and $B$ such that no edges go between $A$ and $B$, the blue special edges
go into $A$ and the red special edges go into $B$.  A generalized Segre
relation is called degenerate if it comes from a degenerate generalized Segre
datum. 
\Ncom{$\text{degenerate gSdatum}$}

\begin{proposition}
\label{prop:gsd-degen}
A degenerate generalized Segre relation lies in $Q_L$.
\end{proposition}

\begin{proof}
First consider a degenerate datum satisfying (1).  Say that the red special
edges are missing.  We show how to move from $Y_{\Sigma}$ to $Y_{\Sigma}'$
using quadratic relations.  First switch the red edges and the green edges of
$\Gamma$ which occur in $U_G$ (this relation lies in $Q_L$).  Then switch the
red edges and blue edges of $\Gamma$ in $U_B$.  The resulting graph now looks
like $\Gamma'$ if we make the red edges black and the other edges purple.

Now say that $\Sigma$ is a degenerate datum satisfying (2).  Say that in $U_G$
the red and blue special edges do not connect and let $U_G=A \amalg B$ be a
decomposition as described above.  Let $\Gamma'$ be the graph obtained from
$\Gamma$ by switching the red and green colors in $A$, let $U_G'=B$, let
$U_R'=U_R \cup A$ and let $U_B'=U_B$.  Let $\Sigma'$ be the generalized Segre
datum $(\Gamma', \{U_R', U_G', U_B'\})$.  Then $\Sigma'$ is degenerate of type
(1) and so $\Rel(\Sigma') \in Q_L$.  Consider the difference
\begin{displaymath}
\Rel(\Sigma)-\Rel(\Sigma')=(Y_{\Sigma}-Y_{\Sigma'})+(Y'_{\Sigma}-Y'_{\Sigma'}).
\end{displaymath}
Now $Y'_{\Sigma}=Y'_{\Sigma'}$.  Also, $Y_{\Sigma}- Y_{\Sigma'} \in Q_L$ (the blue subgraphs of $Y_{\Sigma}$
and $Y_{\Sigma'}$ are identical, and the remainder of the graph gives
a quadratic relation).  Thus $\Rel(\Sigma)\in Q_L$.
\end{proof}

\subsection{Proof of Theorem~\ref{thm:gsc}}

It follows from Theorem~\ref{thm:ytree} that $I_L$ is generated by quadratics
and lifts of the generalized toric Segre cubics.  Thus to prove
Theorem~\ref{thm:gsc}, it suffices to show that the generalized Segre cubics
are lifts of the generalized toric Segre cubics, which is what we now do.

Let $L$ be an even set.  Pick a Y-tree $\Xi$ with leaf set $L$ and let $\Xi^-$
be the associated reduced (caterpillar) tree.  Consider a generalized toric
Segre cubic relation on $\Xi^-$:
\begin{equation}
\label{eq:tgsc}
\begin{split}
&
\begin{xy}
(-8, -3)*{}="A"; (0, -3)*{}="B"; (8, -3)*{}="C"; (0, 2)*{}="D";
"A"; "B"; **\dir{-};
"B"; "C"; **\dir{-};
"B"; "D"; **\dir{-};
(-4, -3)*{1}; (0, 4)*{1}; (4, -3)*{0};
(-8, -6); (-8, 0); **\dir{-};
(-8, 0); (-14, 0); **\dir{-};
(-14, 0); (-14, -6); **\dir{-};
(-14, -6); (-8, -6); **\dir{-};
(-11, -3)*{X};
(8, -6); (8, 0); **\dir{-};
(8, 0); (14, 0); **\dir{-};
(14, 0); (14, -6); **\dir{-};
(14, -6); (8, -6); **\dir{-};
(11, -3)*{Z'};
\end{xy}
\qquad
\begin{xy}
(-8, -3)*{}="A"; (0, -3)*{}="B"; (8, -3)*{}="C"; (0, 2)*{}="D";
"A"; "B"; **\dir{-};
"B"; "C"; **\dir{-};
"B"; "D"; **\dir{-};
(-4, -3)*{0}; (0, 4)*{1}; (4, -3)*{1};
(-8, -6); (-8, 0); **\dir{-};
(-8, 0); (-14, 0); **\dir{-};
(-14, 0); (-14, -6); **\dir{-};
(-14, -6); (-8, -6); **\dir{-};
(-11, -3)*{Z};
(8, -6); (8, 0); **\dir{-};
(8, 0); (14, 0); **\dir{-};
(14, 0); (14, -6); **\dir{-};
(14, -6); (8, -6); **\dir{-};
(11, -3)*{X'};
\end{xy}
\qquad
\begin{xy}
(-8, -3)*{}="A"; (0, -3)*{}="B"; (8, -3)*{}="C"; (0, 2)*{}="D";
"A"; "B"; **\dir{-};
"B"; "C"; **\dir{-};
"B"; "D"; **\dir{-};
(-4, -3)*{1}; (0, 4)*{0}; (4, -3)*{1};
(-8, -6); (-8, 0); **\dir{-};
(-8, 0); (-14, 0); **\dir{-};
(-14, 0); (-14, -6); **\dir{-};
(-14, -6); (-8, -6); **\dir{-};
(-11, -3)*{Y};
(8, -6); (8, 0); **\dir{-};
(8, 0); (14, 0); **\dir{-};
(14, 0); (14, -6); **\dir{-};
(14, -6); (8, -6); **\dir{-};
(11, -3)*{Y'};
\end{xy}
\\ = \qquad &
\begin{xy}
(-8, -3)*{}="A"; (0, -3)*{}="B"; (8, -3)*{}="C"; (0, 2)*{}="D";
"A"; "B"; **\dir{-};
"B"; "C"; **\dir{-};
"B"; "D"; **\dir{-};
(-4, -3)*{1}; (0, 4)*{1}; (4, -3)*{1};
(-8, -6); (-8, 0); **\dir{-};
(-8, 0); (-14, 0); **\dir{-};
(-14, 0); (-14, -6); **\dir{-};
(-14, -6); (-8, -6); **\dir{-};
(-11, -3)*{X};
(8, -6); (8, 0); **\dir{-};
(8, 0); (14, 0); **\dir{-};
(14, 0); (14, -6); **\dir{-};
(14, -6); (8, -6); **\dir{-};
(11, -3)*{X'};
\end{xy}
\qquad
\begin{xy}
(-8, -3)*{}="A"; (0, -3)*{}="B"; (8, -3)*{}="C"; (0, 2)*{}="D";
"A"; "B"; **\dir{-};
"B"; "C"; **\dir{-};
"B"; "D"; **\dir{-};
(-4, -3)*{1}; (0, 4)*{1}; (4, -3)*{1};
(-8, -6); (-8, 0); **\dir{-};
(-8, 0); (-14, 0); **\dir{-};
(-14, 0); (-14, -6); **\dir{-};
(-14, -6); (-8, -6); **\dir{-};
(-11, -3)*{Y};
(8, -6); (8, 0); **\dir{-};
(8, 0); (14, 0); **\dir{-};
(14, 0); (14, -6); **\dir{-};
(14, -6); (8, -6); **\dir{-};
(11, -3)*{Y'};
\end{xy}
\qquad
\begin{xy}
(-8, -3)*{}="A"; (0, -3)*{}="B"; (8, -3)*{}="C"; (0, 2)*{}="D";
"A"; "B"; **\dir{-};
"B"; "C"; **\dir{-};
"B"; "D"; **\dir{-};
(-4, -3)*{0}; (0, 4)*{0}; (4, -3)*{0};
(-8, -6); (-8, 0); **\dir{-};
(-8, 0); (-14, 0); **\dir{-};
(-14, 0); (-14, -6); **\dir{-};
(-14, -6); (-8, -6); **\dir{-};
(-11, -3)*{Z};
(8, -6); (8, 0); **\dir{-};
(8, 0); (14, 0); **\dir{-};
(14, 0); (14, -6); **\dir{-};
(14, -6); (8, -6); **\dir{-};
(11, -3)*{Z'};
\end{xy}
\end{split}
\end{equation}
Implicit in this diagram is a decomposition of $\Xi^-$ into three pieces:
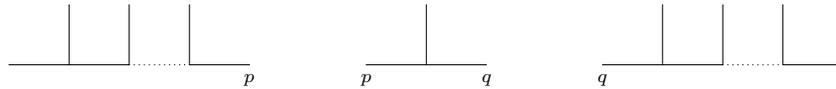
\begin{displaymath}
\begin{xy}
(0, 0)*{}="A"; (8, 0)*{}="B"; (16, 0)*{}="C"; (24, 0)*{}="D"; (32, 0)*{}="E";
(8, 8)*{}="B1"; (16, 8)*{}="C1"; (24, 8)*{}="D1";
(32, -2)*{\ss p};
"A"; "B"; **\dir{-};
"B"; "C"; **\dir{-};
"C"; "D"; **\dir{..};
"D"; "E"; **\dir{-};
"B"; "B1"; **\dir{-};
"C"; "C1"; **\dir{-};
"D"; "D1"; **\dir{-};
\end{xy}
\qquad\qquad
\begin{xy}
(-8, 0)*{}="A"; (0, 0)*{}="B"; (8, 0)*{}="C"; (0, 8)*{}="B1";
(-8, -2)*{\ss p}; (8, -2)*{\ss q};
"A"; "B"; **\dir{-};
"B"; "C"; **\dir{-};
"B"; "B1"; **\dir{-};
\end{xy}
\qquad\qquad
\begin{xy}
(0, 0)*{}="A"; (8, 0)*{}="B"; (16, 0)*{}="C"; (24, 0)*{}="D"; (32, 0)*{}="E";
(8, 8)*{}="B1"; (16, 8)*{}="C1"; (24, 8)*{}="D1";
(0, -2)*{\ss q};
"A"; "B"; **\dir{-};
"B"; "C"; **\dir{-};
"C"; "D"; **\dir{..};
"D"; "E"; **\dir{-};
"B"; "B1"; **\dir{-};
"C"; "C1"; **\dir{-};
"D"; "D1"; **\dir{-};
\end{xy}
\end{displaymath}
The labels indicate how these three pieces are glued together (after gluing,
$p$ and $q$ disappear).  Corresponding to this decomposition is a decomposition
of $L$ into three pieces, the left piece $P$, the right piece $P'$ and the
center piece, which has two elements $x$ and $y$ (the two vertices of
the Y-tree $\Xi$ which connect to the stalk of the trinode in the center of the
above diagram).  We regard $X$, $Y$ and $Z$ as reduced degree one weightings on
the left caterpillar tree, where $X$ and $Y$ take value 1 on $p$ while $Z$
takes value 0 there.  Similarly, we regard $X'$, $Y'$ and $Z'$ as reduced
weightings on the right caterpillar tree, where $X'$ and $Y'$ take value 1 on
$q$ while $Z'$ takes value 0 there.

Now, we can regard $X$, $Y$ and $Z$ as non-reduced degree one weightings on the
Y-tree with vertex set $P \cup \{p^+, p^-\}$.  Similarly, we can regard $X'$,
$Y'$ and $Z'$ as non-reduced degree one weightings on the Y-tree with vertex
set $P' \cup \{q^+, q^-\}$.  We lift these six weightings to matchings, which
we denote by $\wt{X}$, etc.\  In the graph $\wt{X}$ there are two elements of
$P$, say $a$ and $b$, which connect to $p^+$ and $p^-$.  We let $\wt{X}_0$ be
the restriction of $\wt{X}$ to $P \setminus \{ a, b\}$.  We let $c$ and $d$ be
the two vertices in $\wt{Y}$ connecting to $p^+$ and $p^-$ and let $\wt{Y}_0$
be the restriction of $\wt{Y}$ to $P \setminus \{c, d\}$.  In $\wt{Z}$ the
vertices $p^+$ and $p^-$ are connected to each other.  We similarly define
primed versions.

We now define a small generalized Segre cubic datum $\Sigma$.  The partition is
given by $U_B=P'$, $U_R=\{x, y\}$ and $U_G=P$.  The graph $\Gamma$ is given as
follows:
\Ncom{$\text{sm. gSc datum}$}
\begin{itemize}
\item The blue graph is the union of $\wt{X}_0$, $\wt{Z}'$ and the edges
$\uedge{ax}$ and $\uedge{by}$.
\item The green graph is the union of $\wt{Z}$, $\wt{X}_0'$ and the edges
$\uedge{a' x}$ and $\uedge{b' y}$.
\item The red graph is the union of $\wt{Y}_0$, $\wt{Y}_0'$ and the edges
$\uedge{xy}$, $\uedge{cc'}$ and $\uedge{dd'}$.
\end{itemize}
These three matchings are lifts for the three reduced weightings appearing on
the left side of \eqref{eq:tgsc}.  It is clear that this is a valid small
generalized Segre cubic datum.

We must now show that the generalized Segre cubic relation $\Rel(\Sigma)$
associated to the generalized Segre cubic datum defined above lifts the
relation \eqref{eq:tgsc}.  Now, the relation associated to the datum is
\begin{displaymath}
\begin{split}
&
\left( \wt{X}_0 \cdot \uedge{ax} \cdot \uedge{by} \cdot \wt{Z}' \right)
\left( \wt{Z} \cdot \uedge{a'x} \cdot \uedge{b'y} \cdot \wt{X}_0' \right)
\left( \wt{Y} \cdot \uedge{cc'} \cdot \uedge{dd'} \cdot \uedge{xy}
\cdot \wt{Y}' \right) 
\\ 
=&
\left( \wt{X}_0 \cdot \wt{Y}_0 \cdot \uedge{ax} \cdot \uedge{a'x} \cdot
\uedge{by} \cdot \uedge{b'y} \cdot \uedge{cc'} \cdot \uedge{dd'} \cdot
\wt{X}_0' \cdot \wt{Y}_0' \right)
\left( \wt{Z} \cdot \uedge{xy} \cdot \wt{Z}' \right).
\end{split}
\end{displaymath}
Here the parentheses should each be interpreted as single graphs --- for
example, the first parenthetical is the concatenation of $\wt{X}_0$,
$\uedge{ax}$, $\uedge{by}$ and $\wt{Z}'$.  The quadratic term $\Delta$ on the
right side would be the purple subgraph in the graphical notation.  Now, the
above relation belongs to $I_L^{(3)}/Q_L^{(3)}$.  To get an element of
$I_L^{(3)}$, we must rewrite $\Delta$ in terms of degree one elements.  Because
the toric ring $\gr_{\Xi} R_L$ is generated in degree one, we can write
\begin{displaymath}
\Delta=\sum_{i=1}^n \Phi_i \Phi_i'
\end{displaymath}
where each $\Phi_i$ and $\Phi_i'$ has degree one, $\Phi_1 \Phi_1'$ has toric
weight equal to that of $\Delta$ and $\Phi_i \Phi_i'$ has toric weight strictly
less than that of $\Delta$ for $i>1$.  We may take
$\Phi_1=\wt{X}_0 \cdot \uedge{ax} \cdot \uedge{a' y} \cdot
\uedge{bb'} \cdot \wt{X}_0'$ and
$\Phi_1'=\wt{Y}_0 \cdot \uedge{cx} \cdot \uedge{c'y} \cdot
\uedge{dd'} \cdot \wt{Y}_0'$,
since the product of these two graphs is equal to $\Delta$ in the toric ring.  

It now follows that $\Rel(\Sigma)$ is represented by the  relation
\begin{displaymath}
\left( \wt{X}_0 \cdot \uedge{ax} \cdot  \uedge{by} \cdot \wt{Z}' \right)
\left( \wt{Z} \cdot \uedge{a'x} \cdot \uedge{b'y} \cdot \wt{X}_0' \right)
\left( \wt{Y} \cdot \uedge{cc'} \cdot \uedge{dd'}  \cdot \uedge{xy} \cdot
\wt{Y}' \right) 
=
\left( \sum_{i=1}^n \Phi_i \Phi_i' \right) 
\left( \wt{Z}  \cdot \uedge{xy} \cdot \wt{Z}' \right).
\end{displaymath}
The leading term (in terms of the grading) of this relation is 
\begin{displaymath}
\left( \wt{X}_0 \cdot \uedge{ax} \cdot \uedge{by} \cdot \wt{Z}' \right)
\left( \wt{Z} \cdot \uedge{a'x} \cdot \uedge{b'y} \cdot \wt{X}_0' \right)
\left( \wt{Y} \cdot \uedge{cc'} \cdot \uedge{dd'} \cdot \uedge{xy} \cdot
\wt{Y}' \right) \\
=
\left( \Phi_1 \Phi_1' \right) 
\left( \wt{Z} \cdot \uedge{xy} \cdot \wt{Z}' \right).
\end{displaymath}
The two sides are  the same as the two sides of \eqref{eq:tgsc}.  This shows
that the leading term of our generalized Segre relation $\Rel({\Sigma})$ is
equal to the generalized toric Segre relation we started with, which proves
Theorem~\ref{thm:gsc}.

\begin{remark}
\label{rem:gsc}
It follows from the proof that, if we totally order $L$, then $I_L$ is
generated by quadratic relations and those small generalized Segre cubic
relations coming from data for which $U_G<U_R<U_B$ and $|U_R|=2$.  We will
use this stronger form of Theorem~\ref{thm:gsc} in the proof of
Proposition~\ref{prop:retro10b}.
\end{remark}


\section{The structure of $V_L$ and its tensor powers}
\label{s:symvl}

In \S \ref{s:symvl}, we study the partition filtration on $\bigotimes^n V_L$,
$\Sym^n(V_L)$ and $\bwg{n}{V_L}$ for $n \le 3$ and the $\mf{S}_L$-action on the
associated graded pieces.  This provides us with essential structural
properties of the Pl\"ucker relation; for example: the elements $(Y_{\Gamma})^2$
with $\Gamma$ a matching span $\Sym^2(V_L)$ (see the discussion following
Proposition~\ref{prop:pfil2}) and useful generalizations.

\subsection{The partition filtration}

Let $L$ be an even set.  By a \emph{partition} of $L$ we mean a collection of
non-empty disjoint subsets of $L$ whose union is $L$.  We say that a partition
is \emph{into even parts} if each of the subsets is even.  We similarly speak
of partitions of $|L|$ into even parts.  Given a partition $\ms{U}$ of $L$ we
denote by $|\ms{U}|$ the corresponding partition of $|L|$.  We partially order
the set of partitions of $L$ and $|L|$ by refinement.  For example $2+2+2+2$ is
smaller than $4+2+2$, but $4+4$ and $6+2$ are not comparable.
\Ncom{$\text{into even parts}$}

Given a regular colored graph $\Gamma$ on $L$ we obtain a partition
$\ms{U}_{\Gamma}$ of $L$ by taking the vertex sets of the connected
components of $\Gamma$.  The partition $\ms{U}_{\Gamma}$ necessarily has even
parts.  We use this to define a filtration, indexed by the partitions of $|L|$
into even parts, on any $\Z$-module $U$ which is spanned by graphs.  We denote
this filtration, which we call the \emph{partition filtration}, by $F_pU$.
(Warning:  this filtration is indexed by a partially ordered set, not a totally
ordered set.) For instance, if $p$ is a partition of $|L|$ then $F_p
\Sym^k(V_L)$ is the span of the $X_{\Gamma}$ (or $Y_{\Gamma}$) for which
$|\ms{U}_{\Gamma}| \le p$.  We denote the associated graded subquotient by
$\gr_pU$.  To be precise, $\gr_pU$ is the quotient of $F_pU$ by the span of
the $F_{p'}U$ with $p'<p$.  The partition filtration is preserved by the action
of $\mf{S}_L$, so this group naturally acts on $\gr_p U$.
\Ncom{$\text{part.\ filt.}, F_p U, gr_p U$}

\subsection{Degree one spaces}

The $n=1$ case is easy:  the only non-zero piece of the partition filtration
on $V_L$ occurs for $p=2+\cdots+2$ and then $\gr_p=V_L$ is the irreducible
representation of $\mf{S}_L$ corresponding to the partition $n/2+n/2$ (\S
\ref{repthy}). 

\subsection{Degree two spaces}

We now study the partition filtration on $V_L^{\otimes 2}$, $\Sym^2(V_L)$ and
$\bw{V_L}$.

\Rnts{Used in Proposition~\ref{prop:simp-binom} and mentioned in
Proposition~\ref{prop:sr}}

\begin{proposition}
\label{prop:pfil2}
We have, over $\Z[\half]$:
\begin{displaymath}
\gr_p(V_L^{\otimes 2})=\begin{cases}
\Sym^2(V_L) & \textrm{if $p=2+\cdots+2$,} \\
\bw{V_L} & \textrm{if $p=4+2+\cdots+2$,} \\
0 & \textrm{otherwise.}
\end{cases}
\end{displaymath}
\end{proposition}

The space $V_L^{\otimes 2}$ is spanned by regular 2-colored graphs.  Such
graphs are disjoint unions of cycles of even size.
Proposition~\ref{prop:pfil2} says that $V_L^{\otimes 2}$ is spanned by graphs
which are unions of 2-cycles and at most one 4-cycle.  Furthermore,
$\Sym^2(V_L)$ is spanned by graphs which are unions of 2-cycles:  the elements
$(Y_{\Gamma})^2$ with $\Gamma$ a matching span $\Sym^2(V_L)$.

We begin our proof of Proposition~\ref{prop:pfil2} with the following result:

\begin{lemma}
\label{prop:pfil2-a}
The space $V_L^{\otimes 2}$ is spanned over $\Z[\half]$ by graphs which are
unions of 2-cycles and 4-cycles.
\end{lemma}

\begin{proof}
It suffices, by induction, to show that every regular 2-colored graph on at
least six vertices can be written as a sum of graphs which are not connected.
In other words, letting $U_L$ be the subspace of $V_L^{\otimes 2}$ spanned by
disconnected graphs, it suffices to show that $V_L^{\otimes 2}/U_L$ is zero.
Thus let $\Gamma$ be a graph on $L$, which we can assume to be connected
(otherwise it already belongs to $U_L$).  We must show $Y_{\Gamma}=0$ in
$V_L^{\otimes 2}/U_L$.

Pick four consecutive vertices $a$, $b$, $c$ and $d$ of $\Gamma$.  The
Pl\"ucker relation on the edges $\uedge{ab}$ and $\uedge{cd}$ is
\begin{displaymath}
0 \qquad = \qquad
\begin{xy}
(-10, -4.33)*{}="A"; (-5, 4.33)*{}="B"; (5, 4.33)*{}="C";
(10, -4.33)*{}="D";
{\ar@[blue]@{..}@[|(2)] "A"; "B"};
{\ar@[black]@{-} "B"; "C"};
{\ar@[blue]@{..}@[|(2)] "C"; "D"};
(-10, -6.33)*{\ss a}; (-5, 6.33)*{\ss b}; (5, 6.33)*{\ss c};
(10, -6.33)*{\ss d};
\end{xy}
\qquad + \qquad
\begin{xy}
(-10, -4.33)*{}="A"; (-5, 4.33)*{}="B"; (5, 4.33)*{}="C";
(10, -4.33)*{}="D";
{\ar@[blue]@{..}@[|(2)] "A"; "B"};
{\ar@[black]@{-} "B"; "C"};
{\ar@[blue]@{..}@[|(2)] "C"; "D"};
(-10, -6.33)*{\ss a}; (-5, 6.33)*{\ss c}; (5, 6.33)*{\ss b};
(10, -6.33)*{\ss d};
\end{xy}
\qquad + \qquad
\begin{xy}
(-10, -4.33)*{}="A"; (-5, 4.33)*{}="B"; (5, 4.33)*{}="C";
(10, -4.33)*{}="D";
{\ar@[blue]@{..}@[|(2)] "A"; "D"};
{\ar@<-.2ex>@[black]@{-} "B"; "C"};
{\ar@<.2ex>@[blue]@{..}@[|(2)] "B"; "C"};
(-10, -6.33)*{\ss a}; (-5, 6.33)*{\ss b}; (5, 6.33)*{\ss c};
(10, -6.33)*{\ss d};
\end{xy}
\end{displaymath}
The rightmost term belongs to $U_L$.  We therefore have $Y_{\Gamma}
=-Y_{\Gamma'}$ in $V_L^{\otimes 2}/U_L$ where $\Gamma'$ is obtained by
transposing two consecutive vertices in $\Gamma$.  In words, we may transpose
consecutive vertices at the cost of a sign, modulo $U_L$.

Now consider six consecutive vertices $a$, $b$, $c$, $d$, $e$ and $f$ in
$\Gamma$.  Applying the Pl\"ucker relation on the edges $\uedge{a b}$ and
$\uedge{e f}$ we find
\begin{displaymath}
0 \qquad = \qquad
\begin{xy}
(-4, -6.928)*{}="A"; (-8, 0)*{}="B"; (-4, 6.928)*{}="C";
(4, 6.928)*{}="D"; (8, 0)*{}="E"; (4, -6.928)*{}="F";
{\ar@[blue]@{..}@[|(2)] "A"; "B"};
{\ar@[black]@{-} "B"; "C"};
{\ar@[blue]@{..}@[|(2)] "C"; "D"};
{\ar@[black]@{-} "D"; "E"};
{\ar@[blue]@{..}@[|(2)] "E"; "F"};
(-4, -8.928)*{\ss a}; (-10, 0)*{\ss b}; (-4, 8.928)*{\ss c};
(4, 8.928)*{\ss d}; (10, 0)*{\ss e}; (4, -8.928)*{\ss f};
\end{xy}
\qquad + \qquad
\begin{xy}
(-4, -6.928)*{}="A"; (-8, 0)*{}="B"; (-4, 6.928)*{}="C";
(4, 6.928)*{}="D"; (8, 0)*{}="E"; (4, -6.928)*{}="F";
{\ar@[blue]@{..}@[|(2)] "A"; "B"};
{\ar@[black]@{-} "B"; "C"};
{\ar@[blue]@{..}@[|(2)] "C"; "D"};
{\ar@[black]@{-} "D"; "E"};
{\ar@[blue]@{..}@[|(2)] "E"; "F"};
(-4, -8.928)*{\ss a}; (-10, 0)*{\ss e}; (-4, 8.928)*{\ss d};
(4, 8.928)*{\ss c}; (10, 0)*{\ss b}; (4, -8.928)*{\ss f};
\end{xy}
\qquad + \qquad
\begin{xy}
(-4, -6.928)*{}="A"; (-8, 0)*{}="B"; (-4, 6.928)*{}="C";
(4, 6.928)*{}="D"; (8, 0)*{}="E"; (4, -6.928)*{}="F";
{\ar@[blue]@{..}@[|(2)] "A"; "F"};
{\ar@[black]@{-} "B"; "C"};
{\ar@[blue]@{..}@[|(2)] "C"; "D"};
{\ar@[black]@{-} "D"; "E"};
{\ar@[blue]@{..}@[|(2)] "E"; "B"};
(-4, -8.928)*{\ss a}; (-10, 0)*{\ss b}; (-4, 8.928)*{\ss c};
(4, 8.928)*{\ss d}; (10, 0)*{\ss e}; (4, -8.928)*{\ss f};
\end{xy}
\end{displaymath}
Again, the rightmost term belongs to $U_L$ and thus can be discarded.  This
shows that we may pick four consecutive vertices and reverse their order at the
cost of picking up a sign.  However, we may now move the affected vertices back
to the original position using six transpositions, \emph{not} introducing a new
sign.  We thus find that $Y_{\Gamma}=-Y_{\Gamma}$ in $V_L^{\otimes 2}/U_L$,
which shows that $2 Y_{\Gamma}$ belongs to $U_L$.
\end{proof}

In the proof of Lemma~\ref{prop:pfil2-a} we only ever use six consecutive
vertices.  By keeping track of the graphs we discarded during the course of the
proof we obtain an identity, shown in Figure~\ref{f:id6}, that we will use
on a few later occasions.

\begin{figure}[!ht]
\begin{displaymath}
\begin{split}
(-2) \quad
\begin{xy}
(-4, -6.928)*{}="A"; (-8, 0)*{}="B"; (-4, 6.928)*{}="C";
(4, 6.928)*{}="D"; (8, 0)*{}="E"; (4, -6.928)*{}="F";
{\ar@[black]@{-} "A"; "B"};
{\ar@[blue]@{..}@[|(2)] "B"; "C"};
{\ar@[black]@{-} "C"; "D"};
{\ar@[blue]@{..}@[|(2)] "D"; "E"};
{\ar@[black]@{-} "E"; "F"};
\end{xy}
\quad &= \quad
\begin{xy}
(-4, -6.928)*{}="A"; (-8, 0)*{}="B"; (-4, 6.928)*{}="C";
(4, 6.928)*{}="D"; (8, 0)*{}="E"; (4, -6.928)*{}="F";
{\ar@[black]@{-} "A"; "F"};
{\ar@[blue]@{..}@[|(2)] "B"; "C"};
{\ar@[black]@{-} "C"; "D"};
{\ar@[blue]@{..}@[|(2)] "D"; "E"};
{\ar@[black]@{-} "E"; "B"};
\end{xy}
\quad + \quad
\begin{xy}
(-4, -6.928)*{}="A"; (-8, 0)*{}="B"; (-4, 6.928)*{}="C";
(4, 6.928)*{}="D"; (8, 0)*{}="E"; (4, -6.928)*{}="F";
{\ar@[black]@{-} "A"; "D"};
{\ar@[blue]@{..}@[|(2)] "D"; "C"};
{\ar@[black]@{-} "C"; "F"};
{\ar@<.2ex>@[blue]@{..}@[|(2)] "B"; "E"};
{\ar@<-.2ex>@[black]@{-} "B"; "E"};
\end{xy}
\quad + \quad
\begin{xy}
(-4, -6.928)*{}="A"; (-8, 0)*{}="B"; (-4, 6.928)*{}="C";
(4, 6.928)*{}="D"; (8, 0)*{}="E"; (4, -6.928)*{}="F";
{\ar@[black]@{-} "A"; "B"};
{\ar@[blue]@{..}@[|(2)] "B"; "E"};
{\ar@[black]@{-} "E"; "F"};
{\ar@<.2ex>@[blue]@{..}@[|(2)] "C"; "D"};
{\ar@<-.2ex>@[black]@{-} "C"; "D"};
\end{xy}
\\ & \\ &+ \quad
\begin{xy}
(-4, -6.928)*{}="A"; (-8, 0)*{}="B"; (-4, 6.928)*{}="C";
(4, 6.928)*{}="D"; (8, 0)*{}="E"; (4, -6.928)*{}="F";
{\ar@[black]@{-} "A"; "D"};
{\ar@[blue]@{..}@[|(2)] "D"; "E"};
{\ar@[black]@{-} "E"; "F"};
{\ar@<.2ex>@[blue]@{..}@[|(2)] "B"; "C"};
{\ar@<-.2ex>@[black]@{-} "B"; "C"};
\end{xy}
\quad + \quad
\begin{xy}
(-4, -6.928)*{}="A"; (-8, 0)*{}="B"; (-4, 6.928)*{}="C";
(4, 6.928)*{}="D"; (8, 0)*{}="E"; (4, -6.928)*{}="F";
{\ar@[black]@{-} "A"; "B"};
{\ar@[blue]@{..}@[|(2)] "B"; "C"};
{\ar@[black]@{-} "C"; "F"};
{\ar@<.2ex>@[blue]@{..}@[|(2)] "D"; "E"};
{\ar@<-.2ex>@[black]@{-} "D"; "E"};
\end{xy}
\quad + \quad
\begin{xy}
(-4, -6.928)*{}="A"; (-8, 0)*{}="B"; (-4, 6.928)*{}="C";
(4, 6.928)*{}="D"; (8, 0)*{}="E"; (4, -6.928)*{}="F";
{\ar@[black]@{-} "A"; "F"};
{\ar@<.2ex>@[blue]@{..}@[|(2)] "B"; "C"};
{\ar@<-.2ex>@[black]@{-} "B"; "C"};
{\ar@<.2ex>@[blue]@{..}@[|(2)] "D"; "E"};
{\ar@<-.2ex>@[black]@{-} "D"; "E"};
\end{xy}
\\ & \\ &- \quad
\begin{xy}
(-4, -6.928)*{}="A"; (-8, 0)*{}="B"; (-4, 6.928)*{}="C";
(4, 6.928)*{}="D"; (8, 0)*{}="E"; (4, -6.928)*{}="F";
{\ar@[black]@{-} "A"; "C"};
{\ar@[blue]@{..}@[|(2)] "C"; "E"};
{\ar@[black]@{-} "E"; "F"};
{\ar@<.2ex>@[blue]@{..}@[|(2)] "B"; "D"};
{\ar@<-.2ex>@[black]@{-} "B"; "D"};
\end{xy}
\quad - \quad
\begin{xy}
(-4, -6.928)*{}="A"; (-8, 0)*{}="B"; (-4, 6.928)*{}="C";
(4, 6.928)*{}="D"; (8, 0)*{}="E"; (4, -6.928)*{}="F";
{\ar@[black]@{-} "A"; "B"};
{\ar@[blue]@{..}@[|(2)] "B"; "D"};
{\ar@[black]@{-} "D"; "F"};
{\ar@<.2ex>@[blue]@{..}@[|(2)] "C"; "E"};
{\ar@<-.2ex>@[black]@{-} "C"; "E"};
\end{xy}
\quad - \quad
\begin{xy}
(-4, -6.928)*{}="A"; (-8, 0)*{}="B"; (-4, 6.928)*{}="C";
(4, 6.928)*{}="D"; (8, 0)*{}="E"; (4, -6.928)*{}="F";
{\ar@[black]@{-} "A"; "F"};
{\ar@<.2ex>@[blue]@{..}@[|(2)] "B"; "D"};
{\ar@<-.2ex>@[black]@{-} "B"; "D"};
{\ar@<.2ex>@[blue]@{..}@[|(2)] "C"; "E"};
{\ar@<-.2ex>@[black]@{-} "C"; "E"};
\end{xy}
\end{split}
\end{displaymath}
\caption{A graphical identity on 6 points.  More accurately, this should be
regarded as a family of identities between various tensors $Y_{\Gamma}
\otimes Y_{\Delta}$ in $V_L^{\otimes 2}$ for any $L$; we have only drawn the
edges on each side which are different.
\label{f:id6}}
\end{figure}
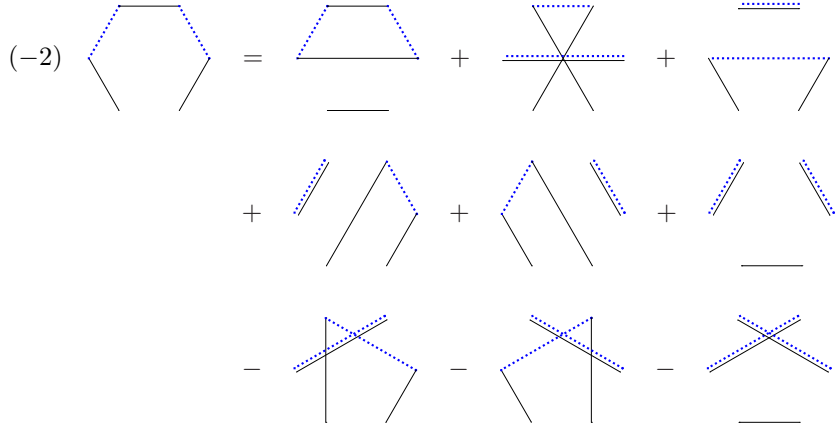

Lemma~\ref{prop:pfil2-a} states that $V_L^{\otimes 2}$ is spanned by
graphs with only 2- and 4-cycles.  To prove Proposition~\ref{prop:pfil2}, we
must therefore show that two 4-cycles can converted into two 2-cycles and one
4-cycle.  This is a question about graphs on eight points; in fact, it suffices
to show $\gr_{4, 4}(V_L^{\otimes 2})=0$ when $L$ has cardinality eight.  We
prove this after the following lemma.

\begin{lemma}
\label{prop:pfil2-b}
Let $L$ be an even set and let $\Gamma$ be a regular 2-colored graph on $L$.
Put $p=|\ms{U}_{\Gamma}|$.  Assume that $\Gamma$ has a 4-cycle and let
$\Gamma'$ be the graph obtained by switching the colors of $\Gamma$ on this
4-cycle.  Then $Y_{\Gamma}=-Y_{\Gamma'}$ holds in $\gr_p(V_L^{\otimes 2})$.
\end{lemma}

\begin{proof}
Let $\Gamma$ be given and say that the vertices in the 4-cycle are labeled
$\{1, 2, 3, 4\}$.  By squaring the identity
\begin{displaymath}
(-1) \quad
\begin{xy}
(0, -4)*{}="A"; (0, 4)*{}="B"; (8, -4)*{}="C"; (8, 4)*{}="D";
(-2, -4)*{\ss 1}; (-2, 4)*{\ss 2}; (10, 4)*{\ss 3}; (10, -4)*{\ss 4};
{\ar@{-} "A"; "D"}; {\ar@{-} "B"; "C"};
\end{xy}
\qquad = \qquad
\begin{xy}
(0, -4)*{}="A"; (0, 4)*{}="B"; (8, -4)*{}="C"; (8, 4)*{}="D";
{\ar@{-} "A"; "C"}; {\ar@{-} "B"; "D"};\end{xy}
\qquad + \qquad
\begin{xy}
(0, -4)*{}="A"; (0, 4)*{}="B"; (8, -4)*{}="C"; (8, 4)*{}="D";
{\ar@{-} "A"; "B"}; {\ar@{-} "C"; "D"};
\end{xy}
\end{displaymath}
we obtain
\begin{displaymath}
\begin{xy}
(0, -4)*{}="A"; (0, 4)*{}="B"; (8, -4)*{}="C"; (8, 4)*{}="D";
{\ar@<.2ex>@[blue]@{..}@[|(2)] "A"; "D"};
{\ar@<-.2ex>@[black]@{-} "A"; "D"};
{\ar@<.2ex>@[blue]@{..}@[|(2)] "B"; "C"};
{\ar@<-.2ex>@[black]@{-} "B"; "C"};
\end{xy}
\qquad = \qquad
\begin{xy}
(0, -4)*{}="A"; (0, 4)*{}="B"; (8, -4)*{}="C"; (8, 4)*{}="D";
{\ar@<.2ex>@[blue]@{..}@[|(2)] "A"; "C"};
{\ar@<-.2ex>@[black]@{-} "A"; "C"};
{\ar@<.2ex>@[blue]@{..}@[|(2)] "B"; "D"};
{\ar@<-.2ex>@[black]@{-} "B"; "D"};
\end{xy}
\qquad + \qquad
\begin{xy}
(0, -4)*{}="A"; (0, 4)*{}="B"; (8, -4)*{}="C"; (8, 4)*{}="D";
{\ar@<.2ex>@[blue]@{..}@[|(2)] "A"; "C"};
{\ar@<-.2ex>@[black]@{-} "A"; "B"};
{\ar@<.2ex>@[blue]@{..}@[|(2)] "B"; "D"};
{\ar@<-.2ex>@[black]@{-} "C"; "D"};
\end{xy}
\qquad + \qquad
\begin{xy}
(0, -4)*{}="A"; (0, 4)*{}="B"; (8, -4)*{}="C"; (8, 4)*{}="D";
{\ar@<.2ex>@[blue]@{..}@[|(2)] "A"; "B"};
{\ar@<-.2ex>@[black]@{-} "A"; "C"};
{\ar@<.2ex>@[blue]@{..}@[|(2)] "C"; "D"};
{\ar@<-.2ex>@[black]@{-} "B"; "D"};
\end{xy}
\qquad + \qquad
\begin{xy}
(0, -4)*{}="A"; (0, 4)*{}="B"; (8, -4)*{}="C"; (8, 4)*{}="D";
{\ar@<.2ex>@[blue]@{..}@[|(2)] "A"; "B"};
{\ar@<-.2ex>@[black]@{-} "A"; "B"};
{\ar@<.2ex>@[blue]@{..}@[|(2)] "C"; "D"};
{\ar@<-.2ex>@[black]@{-} "C"; "D"};
\end{xy}
\end{displaymath}
\Rnts{identity used in pf of L \ref{lem:retro-b}.}
The middle two terms on the right side of this identity are $\Gamma$ and
$\Gamma'$.  The other terms have a more refined partition and so map to 0 in
$\gr_p(V_L^{\otimes 2})$.  This proves the lemma.
\end{proof}

\begin{lemma}
\label{prop:pfil2-c}
Let $L$ be a set of cardinality eight.  Then $\gr_{4, 4}(V_L^{\otimes 2})=0$
over $\Z[\half]$.
\end{lemma}

\begin{proof}
In this proof we write $F_p$ for $F_p(V_L^{\otimes 2})$, and similarly for
$\gr_p$.  Let $\tau$ be the involution on $V_L^{\otimes 2}$ which transposes
factors.  In terms of colored graphs, $\tau$ switches the colors.  We will show
that $\tau$ acts by $+1$ and by $-1$ on $\gr_{4, 4}$, which will establish the
lemma.

We first show that $\tau$ acts by $-1$ on $\gr_{4, 4}$.  By
Lemma~\ref{prop:pfil2-a} we have $F_{4, 4}=V_L^{\otimes 2}$ and $F_{6, 2}=
F_{4, 2, 2}$.  Thus $\gr_{4, 4}$ can be described as the space of all degree
two multi-matchings on $L$ modulo the disconnected ones.  It therefore suffices
to show that $\tau$ acts by $-1$ on a connected graph.  Consider such a
graph $\Gamma$, say the following one:
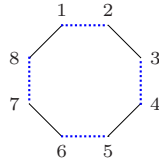
\begin{displaymath}
\begin{xy}
(3.06, 7.39)*{}="A"; (7.39, 3.06)*{}="B"; (7.39, -3.06)*{}="C";
(3.06, -7.39)*{}="D"; (-3.06, -7.39)*{}="E"; (-7.39, -3.06)*{}="F";
(-7.39, 3.06)*{}="G"; (-3.06, 7.39)*{}="H";
(3.06, 9.39)*{\ss 2}; (9.39, 3.06)*{\ss 3}; (9.39, -3.06)*{\ss 4};
(3.06, -9.39)*{\ss 5}; (-3.06, -9.39)*{\ss 6}; (-9.39, -03.06)*{\ss 7};
(-9.39, 3.06)*{\ss 8}; (-3.06, 9.39)*{\ss 1};
{\ar@{-}@[black] "A"; "B"}; {\ar@{-}@[black] "C"; "D"};
{\ar@{-}@[black] "E"; "F"}; {\ar@{-}@[black] "G"; "H"};
{\ar@{..}@[blue]@[|(2)] "B"; "C"}; {\ar@{..}@[blue]@[|(2)] "D"; "E"};
{\ar@{..}@[blue]@[|(2)] "F"; "G"}; {\ar@{..}@[blue]@[|(2)] "A"; "H"};
\end{xy}
\end{displaymath}
By the same reasoning as in the proof of Lemma~\ref{prop:pfil2-a} we may switch
consecutive vertices at the cost of a sign, when working modulo disconnected
graphs.  Thus we may pick vertex 1 and move it counterclockwise around the
cycle to the position of vertex 8.  This uses seven transpositions and so
introduces a sign.  The resulting cycle is the same as the original but with
the colors switched.  We thus have $\tau \Gamma=-\Gamma$ modulo disconnected
graphs, which shows that $\tau$ acts by $-1$ on $\gr_{4, 4}$.

We now show that $\tau$ acts by $+1$ on $\gr_{4, 4}$.  To do this, it suffices
to show that $\tau$ fixes the graph
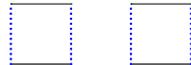
\begin{displaymath}
\begin{xy}
(0, 4)*{}="A"; (8, 4)*{}="B"; (16, 4)*{}="C"; (24, 4)*{}="D";
(0, -4)*{}="E"; (8, -4)*{}="F"; (16, -4)*{}="G"; (24, -4)*{}="H";
{\ar@{-}@[black] "A"; "B"};
{\ar@{-}@[black] "E"; "F"};
{\ar@{-}@[black] "C"; "D"};
{\ar@{-}@[black] "H"; "G"};
{\ar@{..}@[|(2)]@[blue] "A"; "E"};
{\ar@{..}@[|(2)]@[blue] "B"; "F"};
{\ar@{..}@[|(2)]@[blue] "C"; "G"};
{\ar@{..}@[|(2)]@[blue] "D"; "H"};
\end{xy}
\end{displaymath}
modulo $\gr_{4, 2, 2}$.
This follows immediately from Lemma~\ref{prop:pfil2-b}:  switching the color
in one square introduces a sign modulo $\gr_{4, 2, 2}$, so switching the color
in both squares introduces no sign.
\end{proof}

We now complete the proof of Proposition~\ref{prop:pfil2}.

\begin{proof}[Proof of Proposition~\ref{prop:pfil2}]
Combining Lemma~\ref{prop:pfil2-a} with Lemma~\ref{prop:pfil2-c} shows that
$\gr_p(V_L^{\otimes 2})=0$ unless $p=2+\cdots+2$ or $p=4+2+\cdots+2$.  We thus
have a filtration
\begin{displaymath}
0 \subset F_{2+\cdots+2}(V_L^{\otimes 2}) \subset
F_{4+2+\cdots+2}(V_L^{\otimes 2}) =V_L^{\otimes 2}.
\end{displaymath}
Let $\tau$ be the transposition of factors on $V_L^{\otimes 2}$, as in the
proof of Proposition~\ref{prop:pfil2-c}.  It is clear that $\tau$ acts as the
identity on $F_{2+\cdots+2}$ since this space is spanned by graphs which are
unions of 2-cycles.  On the other hand, Lemma~\ref{prop:pfil2-b} shows that
$\tau$ acts by $-1$ on $F_{4+2+\cdots+2}/F_{2+\cdots+2}$.  From this it follows
that $\Sym^2(V_L)=F_{2+\cdots+2}=\gr_{2+\cdots+2}$ and that the quotient map
$V_L^{\otimes 2} \to \gr_{4+2+\cdots+2}$ factors to give an isomorphism
$\bw{V_L} \to \gr_{4+2+\cdots+2}$.
\end{proof}

\subsection{The action of $\mf{S}_L$ on degree two spaces}

\label{action}Having described the spaces $\gr_p(V_L^{\otimes 2})$, we now turn to their
structure as $\mf{S}_L$-modules.  Our main result is the following.  Note that
characteristic $0$ notions and arguments from the representation theory of
$\mf{S}_L$ make sense over $\Z[1/{|L|!}]$ with the obvious changes.

\begin{proposition}
\label{prop:pfil2rep}
In the following table, each $\mf{S}_L$-module is multiplicity free.  The set
of irreducibles it contains corresponds to the given set of partitions.
\begin{center}
\rm
\begin{tabular}{c|c}
$\mf{S}_L$-module & Set of partitions of $|L|$ \\[.5ex]
\hline \\[-2ex]
$\Sym^2(V_L)$ & at most four parts, all even \\[1ex]
$\bw{V_L}$ & exactly four parts, all odd \\[1ex]
$V_L^{\otimes 2}$ & union of previous two sets \\[1ex]
$R^{(2)}_L$ & at most three parts, all even \\[1ex]
$I_L^{(2)}$ & exactly four parts, all even
\end{tabular}
\end{center}
These statements hold over $\Z[1/{|L|!}]$.
\end{proposition}

\Rnts{used in Cor~\ref{lastcor}} 
For instance, $I^{(2)}_L$ is a direct sum of those irreducible representations of
$\mf{S}_L$ corresponding to partitions of $|L|$ into exactly four even parts.
To prove the proposition it suffices to work over the complex numbers $\C$.  We
use Schur-Weyl theory: for a vector space $A$ we have a decomposition
\begin{displaymath}
A^{\otimes L} \cong \bigoplus_{\lambda} M_{\lambda} \otimes \bS_{\lambda}(A)
\end{displaymath}
where the sum is over all partitions $\lambda$ of $|L|$, $M_{\lambda}$ denotes
the irreducible representation of $\mf{S}_L$ attached to a partition $\lambda$
and $\bS_{\lambda}$ is the Schur functor corresponding to $\lambda$.  Here
$A^{\otimes L}$ denotes the tensor product of copies of $A$ indexed by $L$.
The above decomposition respects the action of $\mf{S}_L$ and is functorial
with respect to $A$.

Let $P$ be the two-dimensional vector space over $\C$ with basis $\{x, y \}$.
Let $\bw{P} \to \C$ be the isomorphism taking $x \wedge y$ to 1.  We define
$\SL(P)=\Sp(P)$ to be the group of linear transformations of $P$ preserving
this alternating form. 

Since $V_L$ is the space of degree one invariants, it is equal, by definition,
to $(P^{\otimes L})^{\SL(P)}$.  The action of $\mf{S}_L$ on $V_L$ in this
description is the obvious one.  We thus have an $\mf{S}_L$-equivariant
isomorphism
\begin{displaymath}
V_L \otimes V_L = (P^{\otimes L} \otimes P^{\otimes L})^{\SL(P)
\times \SL(P)} = ((P \otimes P)^{\otimes L})^{\SL(P) \times \SL(P)}.
\end{displaymath}
The $\mf{S}_L$-action is the diagonal action on the first two spaces and the
usual one on the last space.  The above isomorphism is also equivariant with
respect to the transposition of factors $\tau$.  We now apply the Schur-Weyl
decomposition to obtain
\begin{displaymath}
V_L \otimes V_L \cong \bigoplus_{\lambda} M_{\lambda} \otimes \bS_{\lambda}(P
\otimes P)^{\SL(P) \times \SL(P)}.
\end{displaymath}
The space $P \otimes P$ has a natural symmetric inner product coming from the
alternating inner product on $P$.  This inner product is preserved by the group
$\SL(P) \times \SL(P)$, and it is not hard to see that the resulting map
$\SL(P) \times \SL(P) \to \SO(P \otimes P)$ is surjective.  The transposition
$\tau$ on $P \otimes P$ preserves the inner product but has determinant $-1$.
Thus $\tau$ and $\SO(P \otimes P)$ generate $\Orth(P \otimes P)$.  Now, the
space $\Sym^2(V_L)$ is just the $\tau$ invariant part of $V_L \otimes V_L$,
so
\begin{displaymath}
\Sym^2(V_L) \cong \bigoplus_{\lambda} M_{\lambda} \otimes \bS_{\lambda}(P
\otimes P)^{\Orth(P \otimes P)}.
\end{displaymath}
Similarly, $\bw{V_L}$ is just the subspace of $V_L \otimes V_L$ on which $\tau$
acts by $-1$ and so
\begin{displaymath}
\bw{V_L} \cong \bigoplus_{\lambda} M_{\lambda} \otimes \bS_{\lambda}(P \otimes
P)^{\Orth(P \otimes P), -}
\end{displaymath}
where the minus sign means to take the subspace on which $\Orth(P \otimes P)$
acts by its sign representation.  The first three lines of the table in
Proposition~\ref{prop:pfil2rep} now follow from the $n=4$ case of following
lemma.  This lemma appears as statements (1) and (2) in the proof of
\cite[Lemma~2.2]{Kumar}.

\begin{lemma}[S.~Kumar]
\label{KumarLemma}
Let $\lambda$ be a partition and let $V$ be a vector space of dimension $n$
with a non-degenerate symmetric inner product.
\begin{itemize}
\item $\bS_{\lambda}(V)^{\Orth(V)}$ is one-dimensional if $\lambda$ has at most
$n$ parts, all even, and is zero otherwise.
\item $\bS_{\lambda}(V)^{\Orth(V), -}$ is one-dimensional if $\lambda$ has
exactly $n$ parts, all odd, and is zero otherwise.
\end{itemize}
\end{lemma}

We now turn our attention to the space $R^{(2)}_L$, the degree two invariants.
We have an $\mf{S}_L$-equivariant isomorphism
$R^{(2)}_L=(\Sym^2(P)^{\otimes L})^{\SL(P)}$, so
\begin{displaymath}
R^{(2)}_L \cong \bigoplus_{\lambda} M_{\lambda} \otimes \bS_{\lambda}(\Sym^2(P))^{\SL(P)}.
\end{displaymath}
Use the alternating inner product on $P$ to define a symmetric inner product on
$\Sym^2(P)$ via
\begin{displaymath}
\langle vv', ww' \rangle=\langle v, w \rangle \langle v', w' \rangle +
\langle v, w' \rangle \langle v', w \rangle.
\end{displaymath}
The group $\SL(P)$ preserves this inner product and it is not hard to show that
the map $\SL(P) \to \SO(\Sym^2(P))$ is surjective.  We thus have
\begin{displaymath}
R^{(2)}_L=\bigoplus_{\lambda} M_{\lambda} \otimes \bS_{\lambda}(\Sym^2(P))^{
\SO(\Sym^2(P))}.
\end{displaymath}
The fourth line of the table in Proposition~\ref{prop:pfil2rep} now follows
from the following lemma:

\begin{lemma}
Let $\lambda$ be a partition of an even number and let $V$ be a three
dimensional vector space with a non-degenerate symmetric inner product.  Then
$\bS_{\lambda}(V)^{\SO(V)}$ is one-dimensional if $\lambda$ has at most three
parts, all of which are even, and is zero otherwise.
\end{lemma}

\begin{proof}
By the second part of Lemma~\ref{KumarLemma} we have
$\bS_{\lambda}(V)^{\Orth(V), -}=0$, as three odd numbers cannot have an even
sum.  We thus have $\bS_{\lambda}(V)^{\SO(V)}=\bS_{\lambda}(V)^{\Orth(V)}$ and
the result follows from the first part of Lemma~\ref{KumarLemma}.
\end{proof}

The final line of the table in Proposition~\ref{prop:pfil2rep} follows from
$R^{(2)}_L \cong \Sym^2(V_L)/I_L^{(2)}$.

\subsection{Degree three spaces}

We now turn our attention to the cubic spaces $V_L^{\otimes 3}$ and
$\Sym^3(V_L)$.  We say that a $3$-regular graph on $L$ is a \emph{benzene
cycle} if it is a cycle in which the edges alternate between being single and
doubled (Figure~\ref{fig:benzene}).  We use this term because molecules of
benzene are depicted with such graphs (on six points), as in the figure.
We use the term \emph{benzene chain} for a chain of edges which alternate
between being single and doubled.  A benzene 2-cycle is interpreted to mean a
triple edge.  The main result of this section is the following: 
\Ncom{$\text{benzene cycle/chain}$}

\begin{proposition}
\label{prop:pfil3}
The space $\Sym^3(V_L)$ is spanned over $\Z[\half]$ by graphs which are unions
of benzene 2-, 4- and 6-cycles.  In particular, $\gr_p(\Sym^3(V_L))=0$ unless
the parts of $p$ are at most 6.
\end{proposition}

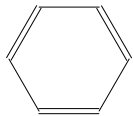
\begin{figure}[!ht]
\begin{center}
\leavevmode
\begin{xy}
(-4, -6.928)*{}="A"; (-8, 0)*{}="B"; (-4, 6.928)*{}="C";
(4, 6.928)*{}="D"; (8, 0)*{}="E"; (4, -6.928)*{}="F";
"A"; "B"; **\dir{-}; "B"; "C"; **\dir{=};
"C"; "D"; **\dir{-}; "D"; "E"; **\dir{=};
"E"; "F"; **\dir{-}; "F"; "A"; **\dir{=};
\end{xy}
\end{center}
\caption{A benzene 6-cycle. \label{fig:benzene}
}
\end{figure}

We deduce Proposition~\ref{prop:pfil3} from the following:

\begin{proposition}
\label{prop:discon-span}
Let $L$ be a set of cardinality at least 8 and let $\Delta$ be a fixed matching
on $L$.  Then $V_L$ is spanned by those $Y_{\Gamma}$ for which the graph
$\Delta  \Gamma$ is not connected.
\end{proposition}

\begin{proof}
Let $\Gamma$ be a given matching.  Think of $\Delta$ as having blue edges and
$\Gamma$ as having black edges, so $\Delta \Gamma$ is a regular 2-colored
graph.  We must show that by using only black Pl\"ucker relations we can write
it as a sum of disconnected graphs.  (We assume that $\Delta  \Gamma$ is
connected to begin with.)

Let $a$, $b$, $c$ and $d$ be four consecutive vertices in $\Delta
 \Gamma$,
where $\uedge{ab}$ is a black edge.  We have the black Pl\"ucker relation
\begin{displaymath}
0 \qquad = \qquad
\begin{xy}
(-10, -4.33)*{}="A"; (-5, 4.33)*{}="B"; (5, 4.33)*{}="C";
(10, -4.33)*{}="D";
{\ar@[black]@{-} "A"; "B"};
{\ar@[blue]@{..}@[|(2)] "B"; "C"};
{\ar@[black]@{-} "C"; "D"};
(-10, -6.33)*{\ss a}; (-5, 6.33)*{\ss b}; (5, 6.33)*{\ss c};
(10, -6.33)*{\ss d};
\end{xy}
\qquad + \qquad
\begin{xy}
(-10, -4.33)*{}="A"; (-5, 4.33)*{}="B"; (5, 4.33)*{}="C";
(10, -4.33)*{}="D";
{\ar@[black]@{-} "A"; "B"};
{\ar@[blue]@{..}@[|(2)] "B"; "C"};
{\ar@[black]@{-} "C"; "D"};
(-10, -6.33)*{\ss a}; (-5, 6.33)*{\ss c}; (5, 6.33)*{\ss b};
(10, -6.33)*{\ss d};
\end{xy}
\qquad + \qquad
\begin{xy}
(-10, -4.33)*{}="A"; (-5, 4.33)*{}="B"; (5, 4.33)*{}="C";
(10, -4.33)*{}="D";
{\ar@[black]@{-} "A"; "D"};{\ar@<-.2ex>@[blue]@{..}@[|(2)] "B"; "C"};
{\ar@<.2ex>@[black]@{-} "B"; "C"};
(-10, -6.33)*{\ss a}; (-5, 6.33)*{\ss b}; (5, 6.33)*{\ss c};
(10, -6.33)*{\ss d};
\end{xy}
\end{displaymath}
The rightmost graph is disconnected.  Thus, working modulo such graphs, we
may transpose blue edges at the cost of a sign.  Now let $a$, $b$, $c$, $d$,
$e$ and $f$ be six consecutive vertices where $\uedge{a b}$ is a black edge.  We
have the black Pl\"ucker relation
\begin{displaymath}
0 \qquad = \qquad
\begin{xy}
(-4, -6.928)*{}="A"; (-8, 0)*{}="B"; (-4, 6.928)*{}="C";
(4, 6.928)*{}="D"; (8, 0)*{}="E"; (4, -6.928)*{}="F";
{\ar@[black]@{-} "A"; "B"};
{\ar@[blue]@{..}@[|(2)] "B"; "C"};
{\ar@[black]@{-} "C"; "D"};
{\ar@[blue]@{..}@[|(2)] "D"; "E"};
{\ar@[black]@{-} "E"; "F"};
(-4, -8.928)*{\ss a}; (-10, 0)*{\ss b}; (-4, 8.928)*{\ss c};
(4, 8.928)*{\ss d}; (10, 0)*{\ss e}; (4, -8.928)*{\ss f};
\end{xy}
\qquad + \qquad
\begin{xy}
(-4, -6.928)*{}="A"; (-8, 0)*{}="B"; (-4, 6.928)*{}="C";
(4, 6.928)*{}="D"; (8, 0)*{}="E"; (4, -6.928)*{}="F";
{\ar@[black]@{-} "A"; "B"};
{\ar@[blue]@{..}@[|(2)] "B"; "C"};
{\ar@[black]@{-} "C"; "D"};{\ar@[blue]@{..}@[|(2)] "D"; "E"};
{\ar@[black]@{-} "E"; "F"};
(-4, -8.928)*{\ss a}; (-10, 0)*{\ss e}; (-4, 8.928)*{\ss d};
(4, 8.928)*{\ss c}; (10, 0)*{\ss b}; (4, -8.928)*{\ss f};
\end{xy}
\qquad + \qquad
\begin{xy}
(-4, -6.928)*{}="A"; (-8, 0)*{}="B"; (-4, 6.928)*{}="C";
(4, 6.928)*{}="D"; (8, 0)*{}="E"; (4, -6.928)*{}="F";
{\ar@[black]@{-} "A"; "F"};
{\ar@[blue]@{..}@[|(2)] "B"; "C"};
{\ar@[black]@{-} "C"; "D"};
{\ar@[blue]@{..}@[|(2)] "D"; "E"};
{\ar@[black]@{-} "E"; "B"};
(-4, -8.928)*{\ss a}; (-10, 0)*{\ss b}; (-4, 8.928)*{\ss c};
(4, 8.928)*{\ss d}; (10, 0)*{\ss e}; (4, -8.928)*{\ss f};
\end{xy}\end{displaymath}
As before, the rightmost graph is disconnected.  This shows that we may take
four consecutive vertices and reverse their direction, at the cost of a sign,
assuming the outer two edges are blue.  In the same way, by considering eight
consecutive vertices we see that a consecutive string of six vertices may be
flipped at the cost of a sign, assuming that the outer two edges are blue.

Now let $a$, $b$, $c$, $d$, $e$ and $f$ be six consecutive vertices, the outer
two edges of which are blue.  We write $[a, b, c, d, e, f]$ to denote this
situation.  By the above, we have
\begin{displaymath}
\begin{split}
[a, b, c, d, e, f]
&=- [f, e, d, c, b, a]
=  [f, e, a, b, c, d] \\
&= - [b, a, e, f, c, d]
=  [b, a, d, c, f, e] 
= - [a, b, c, d, e, f]
\end{split}
\end{displaymath}
so $Y_{\Gamma}=-Y_{\Gamma}$ modulo matchings for which $\Delta \Gamma$ is
disconnected, establishing the proposition.
\end{proof}

By keeping track of the discarded graphs in the above proof we obtain a
complicated identity, shown in Figure~\ref{f:id8}, that we will use on a few
later occasions.  As an immediate corollary of the above proposition we have
the following.

\begin{figure}[!ht]
\begin{displaymath}
\begin{split}
& (-2) \quad
\begin{xy}
(3.06, 7.39)*{}="A"; (7.39, 3.06)*{}="B"; (7.39, -3.06)*{}="C";
(3.06, -7.39)*{}="D"; (-3.06, -7.39)*{}="E"; (-7.39, -3.06)*{}="F";
(-7.39, 3.06)*{}="G"; (-3.06, 7.39)*{}="H";
{\ar@{-}@[black] "A"; "B"}; {\ar@{-}@[black] "C"; "D"};
{\ar@{-}@[black] "E"; "F"}; {\ar@{-}@[black] "G"; "H"};
{\ar@{..}@[blue]@[|(2)] "B"; "C"}; {\ar@{..}@[blue]@[|(2)] "D"; "E"};
{\ar@{..}@[blue]@[|(2)] "F"; "G"}; {\ar@{..}@[blue]@[|(2)] "A"; "H"};
\end{xy}
\\[5ex]
=& \qquad
\begin{xy}
(3.06, 7.39)*{}="A"; (7.39, 3.06)*{}="B"; (7.39, -3.06)*{}="C";
(3.06, -7.39)*{}="D"; (-3.06, -7.39)*{}="E"; (-7.39, -3.06)*{}="F";
(-7.39, 3.06)*{}="G"; (-3.06, 7.39)*{}="H";
{\ar@{-}@[black]@<.2ex> "A"; "H"}; {\ar@{-}@[black] "C"; "D"};
{\ar@{-}@[black] "E"; "F"}; {\ar@{-}@[black] "G"; "B"};
{\ar@{..}@[blue]@[|(2)] "B"; "C"}; {\ar@{..}@[blue]@[|(2)] "D"; "E"};
{\ar@{..}@[blue]@[|(2)] "F"; "G"}; {\ar@{..}@[blue]@[|(2)]@<-.2ex> "A"; "H"};
\end{xy}
\qquad + \qquad
\begin{xy}
(3.06, 7.39)*{}="A"; (7.39, 3.06)*{}="B"; (7.39, -3.06)*{}="C";
(3.06, -7.39)*{}="D"; (-3.06, -7.39)*{}="E"; (-7.39, -3.06)*{}="F";
(-7.39, 3.06)*{}="G"; (-3.06, 7.39)*{}="H";
{\ar@{-}@[black] "A"; "D"}; {\ar@{-}@<.2ex>@[black] "B"; "C"};
{\ar@{-}@[black] "E"; "F"}; {\ar@{-}@[black] "G"; "H"};
{\ar@{..}@<-.2ex>@[blue]@[|(2)] "B"; "C"}; {\ar@{..}@[blue]@[|(2)] "D"; "E"};
{\ar@{..}@[blue]@[|(2)] "F"; "G"}; {\ar@{..}@[blue]@[|(2)] "A"; "H"};
\end{xy}
\qquad + \qquad
\begin{xy}
(3.06, 7.39)*{}="A"; (7.39, 3.06)*{}="B"; (7.39, -3.06)*{}="C";
(3.06, -7.39)*{}="D"; (-3.06, -7.39)*{}="E"; (-7.39, -3.06)*{}="F";
(-7.39, 3.06)*{}="G"; (-3.06, 7.39)*{}="H";
{\ar@{-}@[black] "A"; "B"}; {\ar@{-}@[black] "C"; "F"};
{\ar@{-}@<.2ex>@[black] "D"; "E"}; {\ar@{-}@[black] "G"; "H"};
{\ar@{..}@[blue]@[|(2)] "B"; "C"}; {\ar@{..}@<-.2ex>@[blue]@[|(2)] "D"; "E"};
{\ar@{..}@[blue]@[|(2)] "F"; "G"}; {\ar@{..}@[blue]@[|(2)] "A"; "H"};
\end{xy}
\qquad + \qquad
\begin{xy}
(3.06, 7.39)*{}="A"; (7.39, 3.06)*{}="B"; (7.39, -3.06)*{}="C";
(3.06, -7.39)*{}="D"; (-3.06, -7.39)*{}="E"; (-7.39, -3.06)*{}="F";
(-7.39, 3.06)*{}="G"; (-3.06, 7.39)*{}="H";
{\ar@{-}@[black] "A"; "B"}; {\ar@{-}@[black] "C"; "D"};
{\ar@{-}@[black] "E"; "H"}; {\ar@{-}@<.2ex>@[black] "F"; "G"};
{\ar@{..}@[blue]@[|(2)] "B"; "C"}; {\ar@{..}@[blue]@[|(2)] "D"; "E"};
{\ar@{..}@<-.2ex>@[blue]@[|(2)] "F"; "G"}; {\ar@{..}@[blue]@[|(2)] "A"; "H"};
\end{xy}
\\[5ex]
+& \qquad
\begin{xy}
(3.06, 7.39)*{}="A"; (7.39, 3.06)*{}="B"; (7.39, -3.06)*{}="C";
(3.06, -7.39)*{}="D"; (-3.06, -7.39)*{}="E"; (-7.39, -3.06)*{}="F";
(-7.39, 3.06)*{}="G"; (-3.06, 7.39)*{}="H";
{\ar@{-}@[black] "A"; "B"}; {\ar@{-}@[black] "C"; "H"};
{\ar@{-}@[black]@<.2ex> "D"; "E"}; {\ar@{-}@<.2ex>@[black] "F"; "G"};
{\ar@{..}@[blue]@[|(2)] "B"; "C"}; {\ar@{..}@[blue]@<-.2ex>@[|(2)] "D"; "E"};
{\ar@{..}@<-.2ex>@[blue]@[|(2)] "F"; "G"}; {\ar@{..}@[blue]@[|(2)] "A"; "H"};
\end{xy}
\qquad + \qquad
\begin{xy}
(3.06, 7.39)*{}="A"; (7.39, 3.06)*{}="B"; (7.39, -3.06)*{}="C";
(3.06, -7.39)*{}="D"; (-3.06, -7.39)*{}="E"; (-7.39, -3.06)*{}="F";
(-7.39, 3.06)*{}="G"; (-3.06, 7.39)*{}="H";
{\ar@{-}@[black] "A"; "F"}; {\ar@{-}@[black]@<.2ex> "B"; "C"};
{\ar@{-}@<.2ex>@[black] "D"; "E"}; {\ar@{-}@[black] "G"; "H"};
{\ar@{..}@[blue]@<-.2ex>@[|(2)] "B"; "C"};
{\ar@{..}@<-.2ex>@[blue]@[|(2)] "D"; "E"};
{\ar@{..}@[blue]@[|(2)] "F"; "G"}; {\ar@{..}@[blue]@[|(2)] "A"; "H"};
\end{xy}
\qquad + \qquad
\begin{xy}
(3.06, 7.39)*{}="A"; (7.39, 3.06)*{}="B"; (7.39, -3.06)*{}="C";
(3.06, -7.39)*{}="D"; (-3.06, -7.39)*{}="E"; (-7.39, -3.06)*{}="F";
(-7.39, 3.06)*{}="G"; (-3.06, 7.39)*{}="H";
{\ar@{-}@[black] "A"; "H"}; {\ar@{-}@[black]@<.2ex> "B"; "C"};
{\ar@{-}@[black] "D"; "G"}; {\ar@{-}@[black] "E"; "F"};
{\ar@{..}@[blue]@<-.2ex>@[|(2)] "B"; "C"}; {\ar@{..}@[blue]@[|(2)] "D"; "E"};
{\ar@{..}@[blue]@[|(2)] "F"; "G"}; {\ar@{..}@[blue]@<-.2ex>@[|(2)] "A"; "H"};
\end{xy}
\qquad + \qquad
\begin{xy}
(3.06, 7.39)*{}="A"; (7.39, 3.06)*{}="B"; (7.39, -3.06)*{}="C";
(3.06, -7.39)*{}="D"; (-3.06, -7.39)*{}="E"; (-7.39, -3.06)*{}="F";
(-7.39, 3.06)*{}="G"; (-3.06, 7.39)*{}="H";
{\ar@{-}@[black] "A"; "H"}; {\ar@{-}@[black] "B"; "E"};
{\ar@{-}@[black] "C"; "D"}; {\ar@{-}@<.2ex>@[black] "F"; "G"};
{\ar@{..}@[blue]@[|(2)] "B"; "C"}; {\ar@{..}@[blue]@[|(2)] "D"; "E"};
{\ar@{..}@<-.2ex>@[blue]@[|(2)] "F"; "G"};
{\ar@{..}@[blue]@<-.2ex>@[|(2)] "A"; "H"};
\end{xy}
\\[5ex]
+& \qquad
\begin{xy}
(3.06, 7.39)*{}="A"; (7.39, 3.06)*{}="B"; (7.39, -3.06)*{}="C";
(3.06, -7.39)*{}="D"; (-3.06, -7.39)*{}="E"; (-7.39, -3.06)*{}="F";
(-7.39, 3.06)*{}="G"; (-3.06, 7.39)*{}="H";
{\ar@{-}@[black] "A"; "B"}; {\ar@{-}@[black] "C"; "H"};
{\ar@{-}@[black] "D"; "G"}; {\ar@{-}@[black] "E"; "F"};
{\ar@{..}@[blue]@[|(2)] "B"; "C"}; {\ar@{..}@[blue]@[|(2)] "D"; "E"};
{\ar@{..}@[blue]@[|(2)] "F"; "G"}; {\ar@{..}@[blue]@[|(2)] "A"; "H"};
\end{xy}
\qquad + \qquad
\begin{xy}
(3.06, 7.39)*{}="A"; (7.39, 3.06)*{}="B"; (7.39, -3.06)*{}="C";
(3.06, -7.39)*{}="D"; (-3.06, -7.39)*{}="E"; (-7.39, -3.06)*{}="F";
(-7.39, 3.06)*{}="G"; (-3.06, 7.39)*{}="H";
{\ar@{-}@[black] "A"; "F"}; {\ar@{-}@[black] "B"; "E"};
{\ar@{-}@[black] "C"; "D"}; {\ar@{-}@[black] "G"; "H"};
{\ar@{..}@[blue]@[|(2)] "B"; "C"}; {\ar@{..}@[blue]@[|(2)] "D"; "E"};
{\ar@{..}@[blue]@[|(2)] "F"; "G"}; {\ar@{..}@[blue]@[|(2)] "A"; "H"};
\end{xy}
\qquad + \qquad
\begin{xy}
(3.06, 7.39)*{}="A"; (7.39, 3.06)*{}="B"; (7.39, -3.06)*{}="C";
(3.06, -7.39)*{}="D"; (-3.06, -7.39)*{}="E"; (-7.39, -3.06)*{}="F";
(-7.39, 3.06)*{}="G"; (-3.06, 7.39)*{}="H";
{\ar@{-}@[black]@<.2ex> "A"; "H"}; {\ar@{-}@[black]@<.2ex> "B"; "C"};
{\ar@{-}@[black]@<.2ex> "D"; "E"}; {\ar@{-}@<.2ex>@[black] "F"; "G"};
{\ar@{..}@[blue]@<-.2ex>@[|(2)] "B"; "C"};
{\ar@{..}@[blue]@<-.2ex>@[|(2)] "D"; "E"};
{\ar@{..}@[blue]@<-.2ex>@[|(2)] "F"; "G"};
{\ar@{..}@[blue]@<-.2ex>@[|(2)] "A"; "H"};
\end{xy}
\qquad - \qquad
\begin{xy}
(3.06, 7.39)*{}="A"; (7.39, 3.06)*{}="B"; (7.39, -3.06)*{}="C";
(3.06, -7.39)*{}="D"; (-3.06, -7.39)*{}="E"; (-7.39, -3.06)*{}="F";
(-7.39, 3.06)*{}="G"; (-3.06, 7.39)*{}="H";
{\ar@{-}@[black] "A"; "D"}; {\ar@{-}@[black] "B"; "G"};
{\ar@{-}@[black] "C"; "F"}; {\ar@{-}@[black] "E"; "H"};
{\ar@{..}@[blue]@[|(2)] "B"; "C"}; {\ar@{..}@[blue]@[|(2)] "D"; "E"};
{\ar@{..}@[blue]@[|(2)] "F"; "G"}; {\ar@{..}@[blue]@[|(2)] "A"; "H"};
\end{xy}
\end{split}
\end{displaymath}
\caption{A graphical identity.  The blue edges here are
not relevant to the identity:  they are just drawn to emphasize the
disconnectedness of the graphs on the right side.  Thus this is a linear
relation between $Y_{\Gamma}$'s where the $\Gamma$ are matchings on 8 points.
This identity can be obtained by applying the procedure of
Proposition~\ref{prop:discon-span} to the term on the left side, or by
applying the straightening algorithm to the final term on the right side.
\label{f:id8}}
\Rnts{Used repeatedly} 
\end{figure}
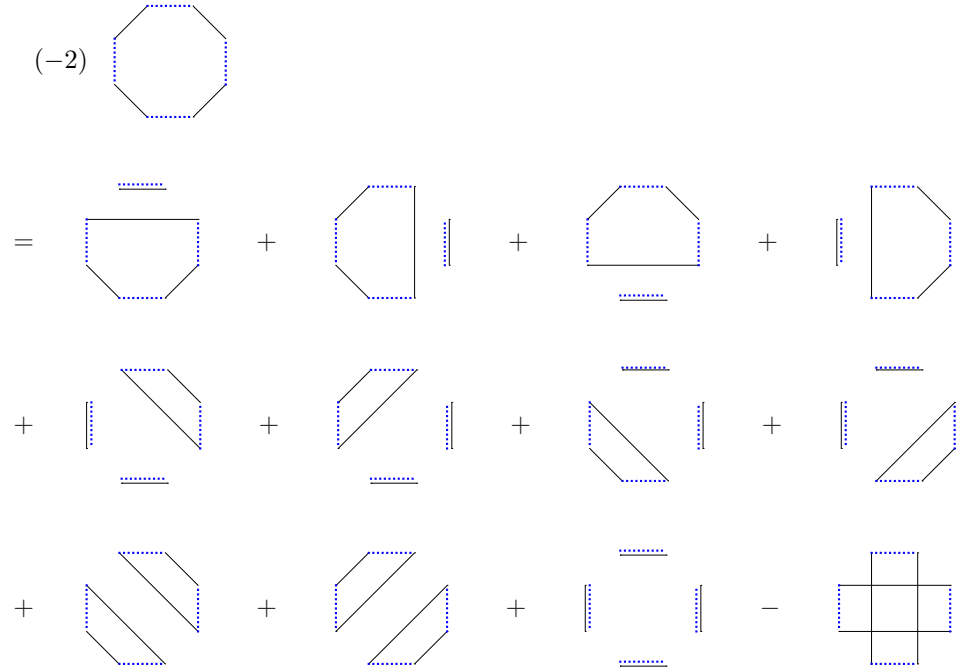

\begin{corollary}
\label{cor:benzene}
Let $\Gamma$ be a regular 3-colored graph and let $\Phi$ be a benzene cycle in
$\Gamma$.  Then in $V_L^{\otimes 3}$ one can write $Y_{\Gamma}$ as a linear
combination of $Y_{\Gamma_i}$'s where in each $\Gamma_i$ the subgraph $\Phi$ is
replaced with a union of benzene 2-, 4- and 6-cycles.  This holds over
$\Z[\half]$.
\Rnts{Used repeatedly:  P\ref{prop:pfil3},
P\ref{prop:sr},T\ref{thm:retro}(a)}
\end{corollary}

\begin{proof}
Assume $\Phi$ has at least eight vertices (otherwise there is nothing to
prove).  Let $\Delta$ be the doubled edges of $\Phi$ and let $\Phi'$ be the
single edges of $\Phi$, so that $\Phi=\Delta \Phi'$.  Use
Proposition~\ref{prop:discon-span} to rewrite $Y_{\Phi'}$ as a sum of
$Y_{\Phi_i'}$ with $\Delta \Phi_i'$ disconnected.  This will rewrite
$Y_{\Gamma}$ as a sum of $Y_{\Gamma_i}$ where in each $Y_{\Gamma_i}$ the
benzene cycle $\Phi$ has been replaced with a union of smaller benzene cycles.
Continuing in this way, one deduces the proposition.
\end{proof}

We can now prove Proposition~\ref{prop:pfil3}.

\begin{proof}[Proof of Proposition~\ref{prop:pfil3}]
Let $\Gamma$ be a regular 3-colored graph on $L$, thought of as having colors
red, green and blue.  Consider the red-green subgraph of $\Gamma$.  We may
apply Proposition~\ref{prop:pfil2} to rewrite this graph as a sum of graphs
which are unions of 2-cycles.  Thus we may as well assume that the red-green
subgraph of $\Gamma$ is made up of 2-cycles.  By now considering the blue edges
as well, we see that $\Gamma$ is a union of benzene cycles.  We may now appeal
to Corollary~\ref{cor:benzene} to break up large benzene cycles into smaller
ones.
\end{proof}

\subsection{The action of $\mf{S}_L$ on degree three spaces}

We now examine the spaces $\gr_p(\Sym^3(V_L))$ more closely when $L$ is small
and determine their structure as $\mf{S}_L$-modules.  We assume throughout that
$|L|!$ is invertible.  We denote by $M_L$ the $\Z[\mf{S}_L]$-module
with a basis given by the set of undirected matchings on $L$.

\begin{proposition}
\label{prop:matchrep}
As an $\mf{S}_L$-representation the space $M_L$ is multiplicity free and
contains those irreducibles corresponding to partitions of $|L|$ into even
parts.
\end{proposition}

This proposition is well-known and essentially equivalent to the decomposition
of the plethysm $\Sym^n \Sym^2$.  In any case, we only need to use this when
$L$ has cardinality four or six, where it can easily be established by hand.
We now begin our study of $\gr_p(\Sym^3(V_L))$.  For the sake of brevity, we
will denote this space simply by $\gr_p$ in this section.  We will also write
$\mf{S}_n$ in place of $\mf{S}_L$ where $n=|L|$.  Our first result is the
following:

\begin{proposition}
\label{prop:match-cube}
The space $\gr_{2+\cdots+2}$ is multiplicity free and contains only those
irreducible representations of $\mf{S}_n$ which have an even number of parts.
\end{proposition}

\begin{proof}
The space $\gr_{2+\cdots+2}$ is the image of the map $\epsilon \otimes M_L \to
\Sym^3(V_L)$ which takes a matching $\Gamma$ to $Y_{\Gamma}^3$.  The result now
follows from Proposition~\ref{prop:matchrep}.
\end{proof}

We will only need to use the above proposition for $n \le 6$.  We now determine
the spaces $\gr_p$ completely for $n=4$ and $6$.  As a warm-up, consider the
$n=2$ case: $\gr_2$ is one-dimensional, and it is the sign representation of
$\mf{S}_2$.  Reason: the space $\Sym^3(V_L)$ is spanned by $Y_{\Gamma}$ where
$\Gamma$ is the graph on $L$ with a tripled edge.  The group $\mf{S}_2$ fixes
$\Gamma$ and so acts on $Y_{\Gamma}$ through the sign character.  All other
$\gr_p$ in this case are zero.  We now turn to $n=4$.

\begin{proposition}
\label{prop:g4}
(a) The space $\gr_{2, 2}$ is free over $\Z[\tfrac{1}{4!}]$ and
three-dimensional.  As an $\mf{S}_4$-representation it decomposes into two
irreducibles corresponding to
the partitions $2+2$ and $1+1+1+1$. \newline
(b) The space $\gr_4$ is free over $\Z[\tfrac{1}{4!}]$ and
one-dimensional.  The representation of $\mf{S}_4$ on it is trivial.
\Rnts{used in \S \ref{placeone}}
\end{proposition}

\begin{proof}
The space $\gr_4$ is spanned by benzene 4-cycles.  We have the identity
\begin{displaymath}
0 \qquad = \qquad
\begin{xy}
(0, -4)*{}="A"; (0, 4)*{}="B"; (8, -4)*{}="C"; (8, 4)*{}="D";
{\ar@<.2ex>@[red]@{--} "A"; "B"};
{\ar@<-.2ex>@[green]@{-} "A"; "B"};
{\ar@<.2ex>@[red]@{--} "C"; "D"};
{\ar@<-.2ex>@[green]@{-} "C"; "D"};
{\ar@[blue]@{..}@[|(2)] "B"; "D"};
{\ar@[blue]@{..}@[|(2)] "A"; "C"};
\end{xy}
\qquad + \qquad
\begin{xy}
(0, -4)*{}="A"; (0, 4)*{}="B"; (8, -4)*{}="C"; (8, 4)*{}="D";
{\ar@<.2ex>@[red]@{--} "A"; "B"};
{\ar@<-.2ex>@[green]@{-} "A"; "B"};
{\ar@<.2ex>@[red]@{--} "C"; "D"};
{\ar@<-.2ex>@[green]@{-} "C"; "D"};
{\ar@[blue]@{..}@[|(2)] "B"; "C"};
{\ar@[blue]@{..}@[|(2)] "A"; "D"};
\end{xy}\qquad + \qquad
\begin{xy}
(0, -4)*{}="A"; (0, 4)*{}="B"; (8, -4)*{}="C"; (8, 4)*{}="D";
{\ar@<.4ex>@[red]@{--} "A"; "B"};
{\ar@[green]@{-} "A"; "B"};
{\ar@<-.4ex>@[blue]@{..}@[|(2)] "A"; "B"};
{\ar@<.4ex>@[red]@{--} "C"; "D"};
{\ar@[green]@{-} "C"; "D"};
{\ar@<-.4ex>@[blue]@{..}@[|(2)] "C"; "D"};
\end{xy}
\end{displaymath}
obtained by Pl\"uckering the blue edges.  We thus see that if $\Gamma$ is
benzene 4-cycle in which $\uedge{ab}$ is a doubled edge then $Y_{\Gamma}+
Y_{(a  b) \Gamma}=0$ in $\gr_4$, which can be rephrased as $(a  b)
Y_{\Gamma}=Y_{\Gamma}$.  In particular, two benzene 4-cycles are equal in
$\gr_4$ (up to a possible sign) if they have the same doubled edges.  By cubing
\begin{displaymath}
\begin{xy}
(0, -4)*{}="A"; (0, 4)*{}="B"; (8, -4)*{}="C"; (8, 4)*{}="D";
{\ar@{-} "A"; "D"}; {\ar@{-} "B"; "C"};
\end{xy}
\qquad = \qquad
\begin{xy}
(0, -4)*{}="A"; (0, 4)*{}="B"; (8, -4)*{}="C"; (8, 4)*{}="D";
{\ar@{-} "A"; "C"}; {\ar@{-} "B"; "D"};
\end{xy}
\qquad + \qquad
\begin{xy}
(0, -4)*{}="A"; (0, 4)*{}="B"; (8, -4)*{}="C"; (8, 4)*{}="D";
{\ar@{-} "A"; "B"}; {\ar@{-} "C"; "D"};
\end{xy}
\end{displaymath}
we obtain
\begin{displaymath}
0 \qquad = \qquad\begin{xy}
(0, -4)*{}="A"; (0, 4)*{}="B"; (8, -4)*{}="C"; (8, 4)*{}="D";
{\ar@<.2ex>@[red]@{--} "A"; "B"};
{\ar@<-.2ex>@[green]@{-} "A"; "B"};
{\ar@<.2ex>@[red]@{--} "C"; "D"};
{\ar@<-.2ex>@[green]@{-} "C"; "D"};
{\ar@[blue]@{..}@[|(2)] "B"; "D"};
{\ar@[blue]@{..}@[|(2)] "A"; "C"};
\end{xy}
\qquad + \qquad
\begin{xy}
(0, -4)*{}="A"; (0, 4)*{}="B"; (8, -4)*{}="C"; (8, 4)*{}="D";
{\ar@<.2ex>@[red]@{--} "A"; "C"};
{\ar@<-.2ex>@[green]@{-} "A"; "C"};
{\ar@<.2ex>@[red]@{--} "B"; "D"};
{\ar@<-.2ex>@[green]@{-} "B"; "D"};
{\ar@[blue]@{..}@[|(2)] "A"; "B"};
{\ar@[blue]@{..}@[|(2)] "C"; "D"};
\end{xy}
\end{displaymath}
in $\gr_4$ (all other terms belong to $F_{2, 2}$).  This shows that we can
switch which edges are the doubled edges.  We have thus shown that $\gr_4$ is
one dimensional.  Since we already know that some transpositions act by the
identity, it follows that $\mf{S}_4$ acts trivially on $\gr_4$.

Now, $\Sym^3(V_L)$ is four-dimensional.  We know that $\gr_4$ is
one-dimensional and so $\gr_{2, 2}$ must be three-dimensional.  By
Proposition~\ref{prop:match-cube} we know that $\gr_{2, 2}$ is a quotient of
the direct sum of the irreducible representations of $\mf{S}_4$ corresponding
to $2+2$ and $1+1+1+1$.  Since this direct sum has dimension $3$, it follows
that $\gr_{2, 2}$ must equal it.
\end{proof}

Finally we consider the case $n=6$.  As stated in \S \ref{outline}, the
computer calculations in the proof are very mild and could probably be done by
hand in a matter of hours.

\begin{proposition}
\label{prop:g6}
(a) The space $\gr_{2, 2, 2}$ is free over $\Z[\tfrac{1}{6!}]$ and
15-dimensional.  As an $\mf{S}_6$-representation it decomposes into three
irreducibles, corresponding to the partitions $3+3$, $2+2+1+1$ and
$1+1+1+1+1+1$.
\newline
(b)  The space $\gr_{4, 2}$ is free over $\Z[\tfrac{1}{6!}]$, 15-dimensional
and decomposes into two irreducibles: one corresponding to $5+1$ and the other
to $4+1+1$.
\newline
(c) The space $\gr_6$ is free over $\Z[\tfrac{1}{6!}]$ and five-dimensional.
It is irreducible and corresponds to $3+3$.
\newline
(d) Let $Q$ be a set of cardinality six and let $c$, $d$ and $e$ be three
distinct elements of $Q$.  Then $\gr_6(\Sym^3(V_Q))$ is spanned by benzene
6-cycles in which $cd$ or $ce$ appears as a doubled edge.  (We call the set
$Q$ here, rather than $L$, since that is what it will be called in the one
place where we apply this statement.)
\Rnts{used in \S \ref{placeone}}
\end{proposition}

\begin{proof}
The code for the computer calculations required here can be found on the
webpage \cite{webpage}.

(a) A computer calculation shows that $F_{2, 2, 2}=\gr_{2, 2, 2}$ is 15
dimensional.  By Proposition~\ref{prop:match-cube}, the space $\gr_{2, 2, 2}$
is a quotient of the direct sum of the irreducible representations of
$\mf{S}_6$ corresponding to the partitions given in the statement of this
proposition.  Since this direct sum is also 15 dimensional the quotient map
is an isomorphism.

(b) A computer calculation shows that $F_{4, 2}$ is 30 dimensional, and so
$\gr_{4, 2}=F_{4, 2}/F_{2,2,2}$ is 15 dimensional.  Now observe that there is
a map $F_4 \otimes F_2 \to F_{4, 2}$ which takes an element of $F_4$ on the
vertices $\{1, 2, 3, 4\}$ and appends a tripled edge on the vertices
$\{5, 6\}$.  One sees using this that that $\gr_{4, 2}$ is a quotient of
$\Ind_{\mf{S}_4 \times \mf{S}_2}^{\mf{S}_6}(\gr_4 \otimes \gr_2)$.  By the
Littlewood-Richardson rule, the induction is a direct sum of the two
irreducible representations of $\mf{S}_6$ corresponding to the partitions given
in the statement of the proposition.  Since this sum is also 15 dimensional the
quotient map is an isomorphism.

(c) We have a non-canonical $\mf{S}_6$-equivariant decomposition $\Sym^3(V_L)
\cong \gr_{2,2,2} \oplus \gr_{4, 2} \oplus \gr_6$.  On the other hand, a
character computation shows that $\Sym^3(V_L) \cong \gr_{2,2,2} \oplus
\gr_{4, 2} \oplus M_{3+3}$ (where $M_{3,3}$ is the irreducible representation
corresponding to $3+3$). Thus $\gr_6= M_{3+3}$.

(d)  This is a straightforward computer calculation.
\end{proof}


\section{Retrogeneration of the ideal}
\label{s:retro}

In this section, we prove that for $|L|$ sufficiently large, the ideal of
relations is retrogenerated (generated by quadratics and relations on fewer
points):

\begin{theorem}
\label{thm:retro}
Let $L$ be an even set of cardinality at least 10.  Then $I_L=I_L^{\retro}$.
For $|L| \ge 12$ this holds over $\Z[\half]$ while for $|L|=10$ it holds
over $\Z[\tfrac{1}{10!}]$.
\end{theorem}

The ideal $I_L^{\retro}$ of retrogenerated relations will be defined in
\S \ref{ss:outer}.  We use general structural arguments to prove
Theorem~\ref{thm:retro} when $|L| \ge 12$ but when $|L|=10$ the small size of
the graphs involved forces us to give an inelegant ad hoc argument.  We
suggest that a reader consider skipping the arguments in the $|L|=10$ case
on a first reading.

\subsection{Outer multiplication and the retrogenerated ideal $I^{\retro}_L$}
\label{ss:outer}

Let $L$ and $L'$ be even sets.  We have an \emph{outer multiplication} map 
\begin{displaymath}
\boxtimes:V_L^{\otimes n} \otimes V_{L'}^{\otimes n} \to
V_{L \, \amalg \, L'}^{\otimes n}, \qquad
Y_{\Gamma} \otimes Y_{\Gamma'} \mapsto Y_{\Gamma \, \amalg \, \Gamma'}.
\end{displaymath}
In words, one takes a colored graph $\Gamma$ on $L$ and a colored graph
$\Gamma'$ on $L'$, with the same set of colors, and obtains a colored graph on
$L \amalg L'$  by taking the disjoint union of $\Gamma$ and $\Gamma'$.   We
will often omit the symbol $\boxtimes$ and write the outer product using
juxtaposition.  
\Ncom{$\text{outer mult.} \boxtimes$}

Outer multiplication does not descend to symmetric powers, as shown by the
following example:
\begin{displaymath}
\begin{xy}
(0, 4)*{}="A"; (8, 4)*{}="B"; (0, -4)*{}="E"; (8, -4)*{}="F";
{\ar@[red]@{--} "A"; "B"};
{\ar@[red]@{--} "E"; "F"};
{\ar@[green]@{-} "A"; "E"};
{\ar@[green]@{-} "B"; "F"};
\end{xy}
\quad \boxtimes \quad
\begin{xy}
(16, 4)*{}="C"; (24, 4)*{}="D";
(16, -4)*{}="G"; (24, -4)*{}="H";
{\ar@[red]@{--} "C"; "D"};
{\ar@[red]@{--} "G"; "H"};
{\ar@[green]@{-} "C"; "G"};
{\ar@[green]@{-} "D"; "H"};
\end{xy}
\qquad = \qquad
\begin{xy}
(0, 4)*{}="A"; (8, 4)*{}="B"; (16, 4)*{}="C"; (24, 4)*{}="D";
(0, -4)*{}="E"; (8, -4)*{}="F"; (16, -4)*{}="G"; (24, -4)*{}="H";
{\ar@[red]@{--} "A"; "B"};
{\ar@[red]@{--} "E"; "F"};
{\ar@[green]@{-} "A"; "E"};
{\ar@[green]@{-} "B"; "F"};
{\ar@[red]@{--} "C"; "D"};
{\ar@[red]@{--} "G"; "H"};
{\ar@[green]@{-} "C"; "G"};
{\ar@[green]@{-} "D"; "H"};
\end{xy}
\end{displaymath}
\begin{displaymath}
\begin{xy}
(0, 4)*{}="A"; (8, 4)*{}="B"; (0, -4)*{}="E"; (8, -4)*{}="F";
{\ar@[red]@{--} "A"; "B"};
{\ar@[red]@{--} "E"; "F"};
{\ar@[green]@{-} "A"; "E"};
{\ar@[green]@{-} "B"; "F"};
\end{xy}
\quad \boxtimes \quad
\begin{xy}
(16, 4)*{}="C"; (24, 4)*{}="D"; (16, -4)*{}="G"; (24, -4)*{}="H";
{\ar@[green]@{-} "C"; "D"};
{\ar@[green]@{-} "G"; "H"};
{\ar@[red]@{--} "C"; "G"};
{\ar@[red]@{--} "D"; "H"};
\end{xy}
\qquad = \qquad
\begin{xy}
(0, 4)*{}="A"; (8, 4)*{}="B"; (16, 4)*{}="C"; (24, 4)*{}="D";
(0, -4)*{}="E"; (8, -4)*{}="F"; (16, -4)*{}="G"; (24, -4)*{}="H";
{\ar@[red]@{--} "A"; "B"};
{\ar@[red]@{--} "E"; "F"};
{\ar@[green]@{-} "A"; "E"};
{\ar@[green]@{-} "B"; "F"};
{\ar@[green]@{-} "C"; "D"};
{\ar@[green]@{-} "G"; "H"};
{\ar@[red]@{--} "C"; "G"};
{\ar@[red]@{--} "D"; "H"};
\end{xy}
\end{displaymath}
Let $L$ and $L'$ be the two sets of four vertices occurring on the left sides.
The two left sides above define equal elements of $\Sym^2(V_L) \otimes
\Sym^2(V_{L'})$.  However, the two right sides are different elements of
$\Sym^2(V_{L \, \amalg \, L'})$ ---  their difference is the simplest binomial
relation \eqref{e:simplest}. 

One sees from the above example that outer multiplication does not descend to
symmetric powers for the following reason:  if $\Gamma$ is an $n$-colored graph
on $L$ and one permutes the colors in each connected component of $\Gamma$ to
obtain a new $n$-colored graph $\Gamma'$ then $Y_{\Gamma}$ and $Y_{\Gamma'}$ do
not represent the same element of $\Sym^n(V_L)$ in general.  Now, if $\Gamma$
and $\Gamma'$ are as in the previous sentence then $Y_{\Gamma}-Y_{\Gamma'}$
lies in the ideal $Q_L$ of $\Sym(V_L)$ generated by quadratic relations.
(Reason: any permutation of colors can be obtained by successive transpositions
of colors, and the relations thus arising are clearly quadratic.)  We hence
find that outer multiplication descends to a map
\begin{displaymath}
\boxtimes : \Sym^n(V_L)/Q_L^{(n)} \otimes \Sym^n(V_{L'})/Q_{L'}^{(n)} \to
\Sym^n(V_{L \, \amalg \, L'})/Q_{L \, \amalg \, L'}^{(n)}
\end{displaymath}
and that $I_L^{(n)}/Q_L^{(n)} \otimes \Sym^n(V_{L'})/Q_{L'}^{(n)}$ is mapped
into $I_{L \, \amalg \, L'}^{(n)}/Q_{L \, \amalg \, L'}^{(n)}$ under
$\boxtimes$, that is, the outer product of anything with a relation is still a
relation.  A motivating example appeared in the introduction:  the Segre
relation on $6$ points \eqref{e:Segre} induces a relation on $8$ points
\eqref{Segre8}.

For a given $L$, we define the \emph{ideal of retrogenerated relations},
denoted $I_L^{\retro}$, to be the ideal of $\Sym(V_L)/Q_L$ generated by the
images of $I^{(n)}_{L'}/Q_{L'}^{(n)} \otimes \Sym^n(V_{L''})/Q_{L''}^{(n)}$
under $\boxtimes$ as $(L', L'')$ varies over all partitions of $L$ into two
disjoint proper even subsets and $n$ varies over all positive integers.  We
also write $I_L^{\retro}$  for the inverse image of $I_L^{\retro}$ under
$\Sym(V_L) \to \Sym(V_L)/Q_L$.  We have inclusions $Q_L \subset I_L^{\retro}
\subset I_L$.  (Theorem~\ref{thm:quadgen} will show that these three ideals
are all the same.)  We say a relation (in $I_L$) is \emph{retrogenerated} if it
lies in $I^{\retro}_L$. 
\Ncom{$I^{\retro}_L \text{ideal of retrogen.\ rels.}$}

A basic fact is that outer multiplication does not increase the ``essential
degree'' of a relation:

\begin{proposition}
\label{prop:stabdeg}
Let $L=L' \amalg L''$ be a partition of $L$ into two proper even subsets.
Let $x \in I_{L'}^{(n)}/Q_{L'}^{(n)}$ belong to the ideal generated by
relations of degree $\le k$ and let $y \in \Sym^n(V_{L''})/Q^{(n)}_{L''}$.
Then the outer product $x \boxtimes y$ belongs to the ideal of $\Sym(V_L)/Q_L$
generated by relations of degree $\le k$.
\end{proposition}

\begin{proof}
Write $x=\sum a_i r_i$ where $r_i$ is a relation of degree $\le k$ and $a_i$
belongs to $\Sym(V_{L'})/Q_{L'}$ and write $y=\sum y_i$ where each $y_i$
is a product of degree one elements.  The outer product of $x \boxtimes y$
is a sum of terms of the form $a_i r_i \boxtimes y_j$.  Say $r_i$ has degree
$k' \le k$ so that $a_i$ has degree $n-k'$, and write $y_j=b_1 \cdots b_n$.
Then $a_i r_i \boxtimes y_j=(a_i \boxtimes (b_1 \cdots b_{n-k'}))(r_i \boxtimes
(b_{n-k'+1} \cdots b_n))$ --- this is the basic compatibility between outer
multiplication and usual multiplication and is trivial to verify.  We have
thus shown that $a_i r_i \boxtimes y_j$ is a multiple of $r_i \boxtimes
(b_{n-k'+1} \cdots b_n)$, a relation of degree $k' \le k$.
\end{proof}

Outer multiplication, simple binomial relations and the retrogenerated ideal
are quite formal constructions and are present when studying $X^n \cq G$ for
any $X$ and $G$.  They are described nicely by the formalism of third author
mentioned in \S \ref{andrew}.  By contrast, the following two propositions are
specific to the present case.

\begin{proposition}
\label{prop:gsd-discon}
Let $\Sigma=(\Gamma, \ms{U})$ be a generalized Segre datum for which $\Gamma$
is disconnected.  Then $\Rel(\Sigma)$ belongs to $I_L^{\retro}$.
\end{proposition}

\begin{proof}
Write $\Gamma=\Gamma_1 \amalg \Gamma_2$.  If one of $\Gamma_1$ or $\Gamma_2$
is entirely contained within one of the parts $U_R$, $U_G$ or $U_B$ then
the relation $\Rel(\Sigma)$ is manifestly retrogenerated.  For instance,
if $\Gamma_1$ is contained within one of the parts then $\Gamma_2$ with
the partition induced from $\ms{U}$ forms a generalized Segre datum $\Sigma_2$
and $\Rel(\Sigma)$ is the outer product of $Y_{\Gamma_1}$ and $\Rel(\Sigma_2)$.
If neither $\Gamma_1$ nor $\Gamma_2$ is contained solely within one part
then each contains a pair of special edges and the datum $\Sigma$ is
forced to be degenerate (of the second case given in \S \ref{ss:gsd-degen}).
The relation $\Rel(\Sigma)$ thus belongs to $Q_L \subset I_L^{\retro}$ by
Proposition~\ref{prop:gsd-degen}.
\end{proof}

The next proposition is a key point in our inductive arguments.

\begin{proposition}
\label{prop:outer246}
The ideal $I_L^{\retro} \subset \Sym(V_L)/Q_L$ is generated over $\Z[\half]$ by
the images of $I^{(3)}_{L'}/Q_{L'}^{(3)} \otimes \Sym^3(V_{L''})/Q_{L''}^{(3)}$
under $\boxtimes$ as $(L', L'')$ varies over all partitions of $L$ into two
disjoint subsets where $L''$ has cardinality 2, 4 or 6.
\end{proposition}

\begin{proof}
Proposition~\ref{prop:stabdeg} shows that $I_L^{\retro}$ is generated by
elements of the form $r \boxtimes Y_{\Gamma}$ where $r$ belongs to
$I_{L'}^{(3)}/Q_{L'}^{(3)}$ and $Y_{\Gamma}$ belongs to
$\Sym^3(V_{L''})/Q_{L''}^{(3)}$, as $(L', L'')$ varies over all partitions of
$L$ into two disjoint even subsets.  Proposition~\ref{prop:pfil3} shows that
we can write $Y_{\Gamma}=\sum a_i Y_{\Gamma_i}$ where $a_i$ belongs to
$\Z[\half]$ and each $\Gamma_i$ is a union of benzene 2-, 4- and 6-cycles.
It follows that $Y_{\Gamma_i}$ is itself an outer product $Y_{\Gamma_{i1}}
\boxtimes \cdots \boxtimes Y_{\Gamma_{in}}$ where each $Y_{\Gamma_{ij}}$
is a benzene 2-, 4- or 6-cycle.  Since $\boxtimes$ is associative, we
have $r \boxtimes Y_{\Gamma_i}=(r \boxtimes Y_{\Gamma_{i1}} \boxtimes
\cdots \boxtimes Y_{\Gamma_{i(n-1)}}) \boxtimes Y_{\Gamma_{in}}$, which
expresses $r \boxtimes Y_{\Gamma_i}$ as the outer product of a graph on
2, 4 or 6 fewer points with a graph on 2, 4 or 6 points.  Thus $r \boxtimes
Y_{\Gamma}$ is a sum of such graphs, which establishes the proposition.
\end{proof}

\subsection{Square rotation relations}

To prove Theorem~\ref{thm:retro} we introduce the \emph{square rotation
relations}, needed only in this proof.  A \emph{square rotation datum} is a
pair $\Pi=(\Gamma, U)$ where $U$ is a subset of $L$ of cardinality 4 and
$\Gamma$ is an undirected graph on $L$ with edges colored purple and black
such that:
\begin{itemize}
\item The vertices of $L \setminus U$ have black valence one and purple valence
two.
\item The vertices of $U$ have black valence zero and purple valence one.
\end{itemize}
Let $\Pi$ be a given square rotation datum.  Suppose that a pair of edges
$e$ and  $e'$ in $\Gamma$ have the same color.  Let $\Gamma+\Gamma'+\Gamma''=0$
be the Pl\"ucker relation on $e$ and $e'$.  Then $\Pi'=(\Gamma', U)$ and
$\Pi''=(\Gamma'', U)$ are both square rotation data.  We define the \emph{space
of square rotation data} to be the $\Z$-span of the square rotation data modulo
the relations $\Pi+\Pi'+\Pi''=0$.
\Ncom{$\text{square rot.\ rels., space of square rotation data}$}

Let $\Pi$ be a square rotation datum.  We have the following quadratic
relation on the four points in $U$:
\begin{equation}
\label{eq:sr}
\begin{xy}
(0, -5)*{}="A"; (10, -5)*{}="B"; (10, 5)*{}="C"; (0, 5)*{}="D";
{\ar@[Thistle]@{-} "A"; "B"};
{\ar@[Thistle]@{-} "C"; "D"};
{\ar@[black]@{-}@[|(3)] "A"; "D"};
{\ar@[black]@{-}@[|(3)] "B"; "C"};
\end{xy}
\qquad = \qquad
\begin{xy}
(0, -5)*{}="A"; (10, -5)*{}="B"; (10, 5)*{}="C"; (0, 5)*{}="D";
{\ar@[black]@{-}@[|(3)] "A"; "B"};
{\ar@[black]@{-}@[|(3)] "C"; "D"};
{\ar@[Thistle]@{-} "A"; "D"};
{\ar@[Thistle]@{-} "B"; "C"};
\end{xy}
\end{equation}
Multiplying both sides by $\Gamma$ we obtain an element of $R_L^{(1)}
\otimes R_L^{(2)}$ which maps to zero in $R_L$.  We may thus regard it as an
element of $I_L^{(3)}/Q_L^{(3)}$, similar to what we did for generalized Segre
relations.  We call such relations \emph{square rotation relations}.  We have
a linear map
\begin{displaymath}
\Rel:\{ \textrm{the space of square rotation data} \} \to I_L^{(3)}/Q_L^{(3)}
\end{displaymath}
mapping a square rotation datum to its associated relation.

\subsection{Retrogeneration of square rotation relations}

The result is:

\begin{proposition}
\label{prop:sr}
If $L$ has cardinality at least 12 then any square rotation relation is
retrogenerated.
\end{proposition}

We use a lemma to prove the proposition.  Let $\Pi$ be a
given square rotation datum.  The purple subgraph of $\Gamma$ has valence two
everywhere except at the four vertices in $U$.  It thus breaks up into a union
of cycles and two paths terminating in $U$.  We call these two
paths the \emph{special paths}.  The \emph{length} of a special path
is the number of vertices it contains. 
\Ncom{$\text{special path, length}$}

\begin{lemma}
(a) A square rotation relation with both special paths of even length lies
in $Q_L$. \newline
(b) A square rotation relation with one special path of even length is
quadratic.
\end{lemma}

\begin{proof}
(a) The way we obtain a square rotation relation from the square rotation datum
is to simply append the graph $\Gamma$ to the relation \eqref{eq:sr}.  Now, on
each side of \eqref{eq:sr} there are two purple edges.  Change the color of
these edges to grey.  The hypothesis on the lengths of the special paths says
that we can color every other edge in the special paths grey (keeping the first
and last edges purple) so that each vertex belongs to one grey and one purple
edge.  We may assume that the purple cycles occurring in $\Gamma$ have even
size (as we can force this using Pl\"ucker relations), so we can pick a
purple-grey alternating coloring of them.  Thus we have factored the purple
subgraph of $\Gamma$ in such a way that all the purple edges appearing in
\eqref{eq:sr} have been colored grey.  But now the relation is evidently
essentially quadratic (i.e., belongs to $Q_L$) since it is taking place solely
on the black-grey graph.

(b) If one special path has even length and one has odd length then $\Gamma$
has an odd purple cycle.  We can thus Pl\"ucker the odd special path and the
odd cycle so that both special paths have even length.  We now use part (a).
\end{proof}

We now prove Proposition~\ref{prop:sr}.

\begin{proof}[Proof of Proposition~\ref{prop:sr}]
Let a square rotation datum be given.  By the above lemma, we can
assume that each of the special paths has odd length.  Using
arguments similar to those occurring in the proof of
Proposition~\ref{prop:pfil2}, we can force the special paths to have
lengths three.  Thus each special path contains a single vertex not
belonging to $U$.  Call these two vertices $x$ and $y$.  We now use
Proposition~\ref{prop:pfil2} to convert the purple cycles in $\Gamma$ into
2-cycles.  We thus have a benzene chain joining $x$ with $y$.  The graph
$\Gamma$ now looks like the following (with the possibility that there are
some additional black-purple benzene cycles not pictured):
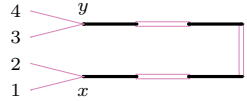
\begin{displaymath}
\begin{xy}
(0, 0)*{}="X"; (0, 7)*{}="Y";
(7, 0)*{}="P1"; (14, 0)*{}="P2"; (14, 7)*{}="P3"; (7, 7)*{}="P4";
(21, 0)*{}="P5"; (21, 7)*{}="P6";
(-7, -1.75)*{}="1"; (-7, 1.75)*{}="2"; (-7, 5.25)*{}="3"; (-7, 8.75)*{}="4";
(0, -2)*{\ss x}; (0, 9)*{\ss y};
(-9, -1.75)*{\ss 1}; (-9, 1.75)*{\ss 2}; (-9, 5.25)*{\ss 3};
(-9, 8.75)*{\ss 4};
{\ar@[black]@{-}@[|(3)] "X"; "P1"};
{\ar@[black]@{-}@[|(3)] "Y"; "P4"};
{\ar@<.2ex>@[Thistle]@{-} "P1"; "P2"};
{\ar@<-.2ex>@[Thistle]@{-} "P1"; "P2"};
{\ar@<.2ex>@[Thistle]@{-} "P3"; "P4"};
{\ar@<-.2ex>@[Thistle]@{-} "P3"; "P4"};
{\ar@<.2ex>@[Thistle]@{-} "P5"; "P6"};
{\ar@<-.2ex>@[Thistle]@{-} "P5"; "P6"};
{\ar@[black]@{-}@[|(3)] "P2"; "P5"};
{\ar@[black]@{-}@[|(3)] "P3"; "P6"};
{\ar@[Thistle]@{-} "X"; "1"};
{\ar@[Thistle]@{-} "X"; "2"};
{\ar@[Thistle]@{-} "Y"; "3"};
{\ar@[Thistle]@{-} "Y"; "4"};
\end{xy}
\end{displaymath}
Here the numbered vertices constitute the set $U$.  The special paths are
1-$x$-2 and 3-$y$-4.  If there are in fact benzene cycles in $\Gamma$ then the
relation is retrogenerated (the proof of this is similar to that of
Proposition~\ref{prop:gsd-discon}).  We can thus assume that there are no
benzene cycles and so the graph really does look like the above one.  Since we
have at least 12 vertices, the benzene chain will have at least four single
black edges, so we can apply Corollary~\ref{cor:benzene} or identity of
Figure~\ref{f:id8} to break up the benzene chain and get a disconnected graph.
The associated relation will therefore be retrogenerated.
\end{proof}

\subsection{Retrogeneration of the ideal on 12 points (Theorem~\ref{thm:retro}
with $|L|=12$)}

We begin with two lemmas.

\begin{lemma}
\label{lem:retro-a}
Let $\Sigma=(\Gamma, \ms{U})$ be a generalized Segre datum, and let
$U \subset U_G$ be a
set of four vertices such that the blue-green graph of $\Gamma$ has a 4-cycle
contained in $U$.  Let $\ol{\Gamma}$ be the graph obtained by rotating the
colors in this 4-cycle and let $\ol{\Sigma}=(\ol{\Gamma}, \ms{U})$.  Then
$\Rel(\Sigma) \equiv  \Rel(\ol{\Sigma})$  modulo quadratic and square rotation
relations.
\end{lemma}

\begin{proof}
We must show that
\begin{displaymath}
\Rel(\Sigma)-\Rel(\ol{\Sigma})=Y_{\Gamma}-Y'_{\Gamma}
-Y_{\ol{\Gamma}}+Y'_{\ol{\Gamma}}
\end{displaymath}
belongs to the ideal generated by quadratic and square rotation relations.  
Clearly, $Y_{\Gamma}-Y_{\ol{\Gamma}} \in Q_L$ since the red
subgraph in each is the same, while $Y'_{\Gamma}-Y'_{\ol{\Gamma}}$ is
a square rotation relation by definition.
\end{proof}

\begin{lemma}
\label{lem:retro-b}
Let $\Sigma=(\Gamma, \ms{U})$ be a generalized Segre datum and let $U
\subset U_G$ be a set of four vertices such that the blue-green graph of
$\Gamma$ has a 4-cycle contained in $U$.  Let $\{ \Gamma_i \}$ be the three
graphs obtained by replacing the 4-cycle on $U$ by two 2-cycles (there are
three ways to do this).  Let $\{\Sigma_i\}$ be the corresponding generalized
Segre data.  Then $2 \Rel(\Sigma) \equiv \sum \Rel(\Sigma_i)$ modulo quadratic
and square rotation relations.
\end{lemma}

\begin{proof}
Recall the identity
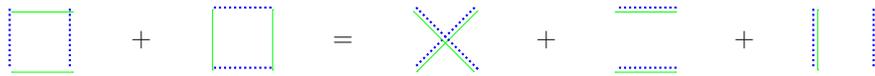
\begin{displaymath}
\begin{xy}
(0, -4)*{}="A"; (0, 4)*{}="B"; (8, -4)*{}="C"; (8, 4)*{}="D";
{\ar@<.2ex>@[blue]@{..}@[|(2)] "A"; "B"};
{\ar@<-.2ex>@[green]@{-} "A"; "C"};
{\ar@<.2ex>@[blue]@{..}@[|(2)] "C"; "D"};
{\ar@<-.2ex>@[green]@{-} "B"; "D"};
\end{xy}
\qquad + \qquad
\begin{xy}
(0, -4)*{}="A"; (0, 4)*{}="B"; (8, -4)*{}="C"; (8, 4)*{}="D";
{\ar@<.2ex>@[blue]@{..}@[|(2)] "A"; "C"};
{\ar@<-.2ex>@[green]@{-} "A"; "B"};
{\ar@<.2ex>@[blue]@{..}@[|(2)] "B"; "D"};
{\ar@<-.2ex>@[green]@{-} "C"; "D"};
\end{xy}
\qquad = \qquad
\begin{xy}
(0, -4)*{}="A"; (0, 4)*{}="B"; (8, -4)*{}="C"; (8, 4)*{}="D";
{\ar@<.2ex>@[blue]@{..}@[|(2)] "A"; "D"};
{\ar@<-.2ex>@[green]@{-} "A"; "D"};
{\ar@<.2ex>@[blue]@{..}@[|(2)] "B"; "C"};
{\ar@<-.2ex>@[green]@{-} "B"; "C"};
\end{xy}
\qquad + \qquad
\begin{xy}
(0, -4)*{}="A"; (0, 4)*{}="B"; (8, -4)*{}="C"; (8, 4)*{}="D";
{\ar@<.2ex>@[blue]@{..}@[|(2)] "A"; "C"};
{\ar@<-.2ex>@[green]@{-} "A"; "C"};
{\ar@<.2ex>@[blue]@{..}@[|(2)] "B"; "D"};
{\ar@<-.2ex>@[green]@{-} "B"; "D"};
\end{xy}
\qquad + \qquad
\begin{xy}
(0, -4)*{}="A"; (0, 4)*{}="B"; (8, -4)*{}="C"; (8, 4)*{}="D";
{\ar@<.2ex>@[blue]@{..}@[|(2)] "A"; "B"};
{\ar@<-.2ex>@[green]@{-} "A"; "B"};
{\ar@<.2ex>@[blue]@{..}@[|(2)] "C"; "D"};
{\ar@<-.2ex>@[green]@{-} "C"; "D"};
\end{xy}
\end{displaymath}
of the proof of Lemma~\ref{prop:pfil2-b}.  This identity holds in
$V_L^{\otimes 2}$, that is, it follows from the colored Pl\"ucker relations.
This shows that in the space of generalized Segre data, we have $\Sigma+\
\ol{\Sigma}= \sum_{i=1}^3 \Sigma_i$ where $\ol{\Sigma}$ was defined in the
previous lemma.  By the previous lemma, we have $\Rel(\Sigma)=
\Rel(\ol{\Sigma})$ modulo quadratic and square rotation relations.  This
proves the current lemma.
\end{proof}

We now complete the proof of Theorem~\ref{thm:retro} when $|L| \ge 12$.

\begin{proof}[Proof of Theorem~\ref{thm:retro} when $|L| \ge 12$]
By Theorem~\ref{thm:gsc}, it suffices to show that small generalized Segre
cubic relations are retrogenerated.  Thus let $\Sigma$ be a given small
generalized Segre datum.  We assume without loss of generality that $U_R$ has
cardinality two, and that $\Sigma$ is non-degenerate.

We begin by considering the blue-green subgraph of $U_G$.  This is a union of
cycles and a single chain going between the two special blue edges.  By using
Proposition~\ref{prop:pfil2}, or more accurately the identity of
Figure~\ref{f:id6}, we can convert this graph into a union of 2-cycles and
4-cycles and make the chain have length  three (so that the two special
blue edges are joined by a single green edge).  We can now use the above
two lemmas to convert the 4-cycles into 2-cycles, modulo retrogenerated
relations.

We have thus reduced to the case where the blue-green subgraph of $\Gamma$
in $U_G$ is made up of 2-cycles (except for the path of length three involving
the two special blue edges).  We now consider the red edges.  Except for the
special blue edges, the graph $\Gamma \vert_{U_G}$ is a union of benzene
chains with the blue and green edges being paired.  There are two incoming
special red edges and two red edges connected to the two special blue
edges.  From each of these edges a benzene chain emanates which must terminate at
one of the other edges.  These four red edges are thus contained in two benzene
chains.  If the two special red edges are contained in the same benzene
chain then the graph is degenerate and the relation lies in $Q_L$.  It therefore
suffices to consider the case where each special red edge connects to a
special blue edge via a benzene chain.

Now, we may run the entire argument given above inside  $U_B$ as well.
We thus conclude that $\Gamma \vert_{U_B}$ is a union of benzene chains
with the blue and green edges being paired.  The two special red edges
connect with the two special green edges by benzene chains.  The picture
is thus something like:
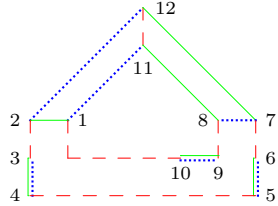
\begin{displaymath}
\begin{xy}
(-10, 0)*{}="A1"; (-15, 0)*{}="A2";
(10, 0)*{}="B1"; (15, 0)*{}="B2";
(0, 10)*{}="C1"; (0, 15)*{}="C2";
(-10, -5)*{}="D1"; (-15, -5)*{}="D2";
(-15, -10)*{}="E2";
(10, -5)*{}="F1"; (15, -5)*{}="F2";
(5, -5)*{}="G1"; (15, -10)*{}="G2";
(-8, 0)*{\ss 1}; (-17, 0)*{\ss 2};
(-17, -5)*{\ss 3}; (-17, -10)*{\ss 4};
(5, -7)*{\ss 10}; (10, -7)*{\ss 9};
(17, -10)*{\ss 5}; (17, -5)*{\ss 6};
(17, 0)*{\ss 7}; (8, 0)*{\ss 8};
(0, 7)*{\ss 11}; (3, 15)*{\ss 12};
{\ar@[red]@{--} "C1"; "C2"};
{\ar@[blue]@{..}@[|(2)] "C1"; "A1"};
{\ar@[blue]@{..}@[|(2)] "C2"; "A2"};
{\ar@[green]@{-} "C1"; "B1"};
{\ar@[green]@{-} "C2"; "B2"};
{\ar@[green]@{-} "A1"; "A2"};
{\ar@[blue]@{..}@[|(2)] "B1"; "B2"};
{\ar@[red]@{--} "A1"; "D1"};
{\ar@[red]@{--} "A2"; "D2"};
{\ar@[blue]@{..}@[|(2)]@<.2ex> "D2"; "E2"};
{\ar@[green]@{-}@<-.2ex> "D2"; "E2"};
{\ar@[red]@{--} "B1"; "F1"};
{\ar@[red]@{--} "B2"; "F2"};
{\ar@[blue]@{..}@[|(2)]@<.2ex> "F1"; "G1"};
{\ar@[green]@{-}@<-.2ex> "F1"; "G1"};
{\ar@[blue]@{..}@[|(2)]@<.2ex> "F2"; "G2"};
{\ar@[green]@{-}@<-.2ex> "F2"; "G2"};
{\ar@[red]@{--} "D1"; "G1"};
{\ar@[red]@{--} "E2"; "G2"};
\end{xy}
\end{displaymath}
Here $U_G=\{1, 2, 3, 4\}$, $U_B=\{5, 6, 7, 8, 9, 10\}$ and $U_R=\{11, 12\}$.
If there are other benzene cycles present, 
then the relation is immediately retrogenerated.  Also, the two benzene chains
connecting the left and right side could be crossed; that is, the chains
could go from 1 to 7 and 2 to 8 instead of as they do.  However, one can
always rectify this by Pl\"uckering two red edges such as $\uedge{89}$ and
$\uedge{6 7}$ --- one of the resulting graphs has the chains uncrossed while the
other is degenerate.  So the above graph is the only sort we need
consider.

Now, we can repartition the vertices so that all the blue-green doubled
edges are contained in $U_G$.  For example, with the above graph we would
repartition so that $U_G=\{1, 2, 3, 4, 5, 6, 9, 10\}$, $U_B=\{7, 8\}$ and
$U_R=\{11, 12\}$.  The resulting Segre datum yields the same
relation as the original.  We next point out that we can put all the
blue-green doubled edges into a single benzene chain by Pl\"uckering two
red edges.  For example, in the above graph we would Pl\"ucker the edges
$\uedge{1 \;  10}$ and $\uedge{6 7}$.  In one graph, $\uedge{17}$ is a red edge and all the
blue-green doubled edges are in a single benzene chain running between 2 and
8.  The other graph is degenerate.  Now, since we have at least 12 vertices,
the benzene chain will have at least four single red edges.  It can
therefore be broken apart using Corollary~\ref{cor:benzene} or
identity of Figure~\ref{f:id8}, yielding a disconnected graph and therefore a
retrogenerated relation.  This completes the proof of Theorem~\ref{thm:retro}
when $|L| \ge 12$.
\end{proof}

\subsection{Retrogeneration of the ideal on 10 points (Theorem~\ref{thm:retro}
with $|L|=10$)}

For the remainder of \S \ref{s:retro} we let $L$ be a finite set of cardinality
10.  To prove Theorem~\ref{thm:retro} in this case, we show the following:

\begin{proposition}
\label{prop:retro10b}
The space $I_L^{(3)}/I_L^{\retro, (3)}$ has dimension at most two
(over $\Z[\half]$).
\end{proposition}

We now explain why this implies Theorem~\ref{thm:retro} for $|L|=10$.  For the next few
sentences we work over a field $k$ of characteristic not 2, 3, 5 or 7.  A
character computation shows that $\Sym^3(V_L)$ does not contain the trivial or
alternating representation.  Now, the representation $I_L^{(3)}/I_L^{\retro,
(3)}$ is at most two-dimensional, and a summand of $\Sym^3(V_L)$.
As any representation of $\mf{S}_L$ of dimension at most two is made up of some
combination of the trivial and alternating representations, it follows that
$I_L^{(3)}/I_L^{\retro, (3)}$ vanishes.  Thus $I_L^{(3)}=I_L^{\retro, (3)}$
over $k$.  Since this holds for all $k$ of characteristic not 2, 3, 5 or 7, we
conclude $I_L^{(3)}=I_L^{\retro, (3)}$.  Since $I_L$ is generated by quadratics
and $I_L^{(3)}$ (Theorem~\ref{thm:gsc}), this implies $I_L=I_L^{\retro}$.

We now turn to proving Proposition~\ref{prop:retro10b}.  We prove the
following result:

\begin{proposition}
\label{prop:retro10c}
Let $\ms{U}=\{U_G, U_R, U_B\}$ be a fixed partition of $L$ with $|U_R|=2$.
Consider the subspace $V$ of $I_L^{(3)}/I_L^{\retro, (3)}$ spanned by all
generalized Segre cubic relations coming from data with partition equal to
$\ms{U}$.  Then, over $\Z[\tfrac 1 2]$:
\begin{enumerate}
\item If $U_G$ (and thus $U_B$) has cardinality four then $V=0$.
\item If either $U_G$ or $U_B$ has cardinality six then $\dim{V} \le 1$.
\end{enumerate}
\end{proposition}

We explain why Proposition~\ref{prop:retro10c} implies
Proposition~\ref{prop:retro10b}.  Fix an  order on $L$.  Call a partition
$\ms{U}=\{U_G, U_R, U_B\}$ of $L$ \emph{admissible} if each part is non-empty
of even cardinality, $U_G<U_R<U_B$ in the order and $|U_R|=2$.  We call
a generalized Segre cubic datum \emph{admissible} if its partition is; we
extend the notion to relations in the obvious manner.  By inspection, there
are  three admissible partitions of $L$.  One of these has $|U_G|=
|U_B|=4$ while in the other two, one of $U_G$ or $U_B$ has cardinality six.
It follows from Proposition~\ref{prop:retro10c} that the admissible Segre cubic
relations span a subspace of $I_L^{(3)}/I_L^{\retro, (3)}$ of dimension at most
two.  On the other hand, we know that the admissible generalized Segre cubic
relations span $I_L^{(3)}/Q_L^{(3)}$ (see Remark~\ref{rem:gsc}).  As
$I_L^{\retro, (3)}$ contains $Q_L^{(3)}$ we conclude that $I_L^{(3)}/
I_L^{\retro, (3)}$ has dimension at most two, establishing
Proposition~\ref{prop:retro10b}. 
\Ncom{$\text{adm.\ partition or gSc datum}$}

We now begin proving Proposition~\ref{prop:retro10c}.  We first consider the
case where $|U_G|=4$ (so $|U_B|=4$ too).  Consider the blue-green subgraph of
$U_G$.  In it there are two green edges, two blue edges contained entirely in
$U_G$ and two blue edges (the special ones) going between $U_G$ and $U_R$.  By
using identity of Figure~\ref{f:id6} we can force there to be a blue-green
doubled edge.  This implies that the special blue edges are joined by one
green edge.  By now considering the red edges, we see that there are two
possibilities for the picture in $U_G$:
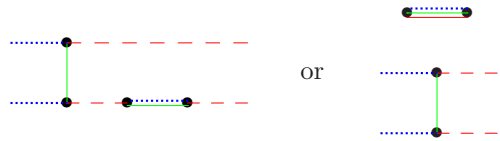
\begin{displaymath}
\begin{xy}
(0, -4)*{}="A1"; (0, 4)*{}="A2";
(8, -4)*{}="B1"; (8, 4)*{}="B2";
(16, -4)*{}="C1"; (16, 4)*{}="C2";
(24, -4)*{}="D1"; (24, 4)*{}="D2";
(32, -4)*{}="E1"; (32, 4)*{}="E2";
"B1"*{\bullet}; "B2"*{\bullet};
"C1"*{\bullet}; "D1"*{\bullet};
{\ar@[green]@{-} "B1"; "B2"};
{\ar@[blue]@{..}@[|(2)] "A1"; "B1"};
{\ar@[blue]@{..}@[|(2)] "A2"; "B2"};
{\ar@[red]@{--} "B1"; "C1"};
{\ar@[red]@{--} "B2"; "E2"};
{\ar@[blue]@{..}@[|(2)]@<.2ex> "C1"; "D1"};
{\ar@[green]@{-}@<-.2ex> "C1"; "D1"};
{\ar@[red]@{--} "D1"; "E1"};
\end{xy}
\qquad \textrm{or} \qquad
\begin{xy}
(0, 0)*{}="A1"; (0, -8)*{}="A2";
(8, 0)*{}="B1"; (8, -8)*{}="B2";
(16, 0)*{}="C1"; (16, -8)*{}="C2";
(4, 8)*{}="D1"; (12, 8)*{}="E1";
"B1"*{\bullet}; "B2"*{\bullet};
"D1"*{\bullet}; "E1"*{\bullet};
{\ar@[green]@{-} "B1"; "B2"};
{\ar@[blue]@{..}@[|(2)] "A1"; "B1"};
{\ar@[blue]@{..}@[|(2)] "A2"; "B2"};
{\ar@[red]@{--} "B1"; "C1"};
{\ar@[red]@{--} "B2"; "C2"};
{\ar@[blue]@{..}@[|(2)]@<.4ex> "D1"; "E1"};
{\ar@[green]@{-} "D1"; "E1"};
{\ar@[red]@{-}@<-.4ex> "D1"; "E1"};
\end{xy}
\end{displaymath}
The right case is disconnected, hence retrogenerated.  We thus need
only consider the left case.  We now go through the same considerations in
$U_B$ as we just did in $U_G$ and conclude that it too must look like the graph
on the left, except with the colors blue and green reversed.  We find that the
graph as a whole must be one of the two in the statement of the following
lemma.  That lemma then shows that that the generalized Segre relation we are
considering is retrogenerated, which completes the $|U_G|=4$ case.

\begin{lemma}
\label{lem:10a}
Let $\Gamma$ be one of the following two graphs.
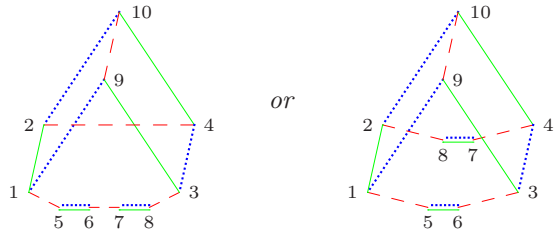
\begin{displaymath}
\begin{xy}
(-10, -12)*{}="A1"; (-8, -3)*{}="A2";
(0, 3)*{}="B1"; (2, 12)*{}="B2";
(10, -12)*{}="C1"; (12, -3)*{}="C2";
(-6, -14)*{}="D1"; (-2, -14)*{}="D2";
(2, -14)*{}="E1"; (6, -14)*{}="E2";
(-12, -12)*{\ss 1}; (-10, -3)*{\ss 2};
(12, -12)*{\ss 3}; (14, -3)*{\ss 4};
(2, 3)*{\ss 9}; (5, 12)*{\ss 10};
(-6, -16)*{\ss 5}; (-2, -16)*{\ss 6};
(2, -16)*{\ss 7}; (6, -16)*{\ss 8};
{\ar@[blue]@{..}@[|(2)] "A1"; "B1"};
{\ar@[red]@{--} "A1"; "D1"};
{\ar@[red]@{--} "D2"; "E1"};
{\ar@[red]@{--} "E2"; "C1"};
{\ar@<.2ex>@[blue]@{..}@[|(2)] "D1"; "D2"};
{\ar@<-.2ex>@[green]@{-} "D1"; "D2"};
{\ar@<.2ex>@[blue]@{..}@[|(2)] "E1"; "E2"};
{\ar@<-.2ex>@[green]@{-} "E1"; "E2"};
{\ar@[green]@{-} "B1"; "C1"};
{\ar@[blue]@{..}@[|(2)] "A2"; "B2"};
{\ar@[red]@{--} "A2"; "C2"};
{\ar@[green]@{-} "B2"; "C2"};
{\ar@[blue]@{..}@[|(2)] "C1"; "C2"};
{\ar@[red]@{--} "B1"; "B2"};
{\ar@[green]@{-} "A1"; "A2"};
\end{xy}
\qquad \textrm{or} \qquad
\begin{xy}
(-10, -12)*{}="A1"; (-8, -3)*{}="A2";
(0, 3)*{}="B1"; (2, 12)*{}="B2";
(10, -12)*{}="C1"; (12, -3)*{}="C2";
(-2, -14)*{}="D1"; (2, -14)*{}="D2";
(0, -5)*{}="E1"; (4, -5)*{}="E2";
(-12, -12)*{\ss 1}; (-10, -3)*{\ss 2};
(12, -12)*{\ss 3}; (14, -3)*{\ss 4};
(2, 3)*{\ss 9}; (5, 12)*{\ss 10};
(-2, -16)*{\ss 5}; (2, -16)*{\ss 6};
(4, -7)*{\ss 7}; (0, -7)*{\ss 8};
{\ar@[blue]@{..}@[|(2)] "A1"; "B1"};
{\ar@[red]@{--} "A1"; "D1"};
{\ar@[red]@{--} "D2"; "C1"};
{\ar@[red]@{--} "A2"; "E1"};
{\ar@[red]@{--} "E2"; "C2"};
{\ar@<.2ex>@[blue]@{..}@[|(2)] "D1"; "D2"};
{\ar@<-.2ex>@[green]@{-} "D1"; "D2"};
{\ar@<.2ex>@[blue]@{..}@[|(2)] "E1"; "E2"};
{\ar@<-.2ex>@[green]@{-} "E1"; "E2"};
{\ar@[green]@{-} "B1"; "C1"};
{\ar@[blue]@{..}@[|(2)] "A2"; "B2"};
{\ar@[green]@{-} "B2"; "C2"};
{\ar@[blue]@{..}@[|(2)] "C1"; "C2"};
{\ar@[red]@{--} "B1"; "B2"};
{\ar@[green]@{-} "A1"; "A2"};
\end{xy}
\end{displaymath}
Let $\ms{U}=\{U_R, U_G, U_B\}$ be a partition such that $U_R=\{9, 10\}$,
$U_G \supset \{1, 2\}$, $U_B \supset \{3, 4\}$ and $\Sigma=(\Gamma,
\ms{U})$ is a generalized Segre cubic datum.  Then $\Rel(\Sigma)$ is
retrogenerated over $\Z[\half]$.
\end{lemma}

\begin{proof}
We first note that the relation $\Rel(\Sigma)$ is independent of the partition
$\ms{U}$.  We therefore assume that
the vertices 5, 6, 7 and 8 belong to $U_G$.  Now, let $\Gamma$ be the left
graph.  Redraw $\Gamma$ as follows:
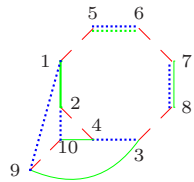
\begin{displaymath}
\begin{xy}
(3.06, 7.39)*{}="A"; (7.39, 3.06)*{}="B"; (7.39, -3.06)*{}="C";
(3.06, -7.39)*{}="D"; (-3.06, -7.39)*{}="E"; (-7.39, -3.06)*{}="F";
(-7.39, 3.06)*{}="G"; (-3.06, 7.39)*{}="H";
(-7.39, -7.39)*{}="I"; (-11.39, -11.39)*{}="J";
(3.06, 9.39)*{\ss 6}; (9.39, 3.06)*{\ss 7};
(9.39, -3.06)*{\ss 8}; (-3.06, 9.39)*{\ss 5};
(-9.39, 3.06)*{\ss 1}; (-5.39, -2.56)*{\ss 2};
(3.06, -9.39)*{\ss 3}; (-2.56, -5.39)*{\ss 4};
(-6.39, -8.39)*{\ss 10}; (-13.39, -11.39)*{\ss 9};
{\ar@{--}@[red] "A"; "B"};
{\ar@{--}@[red] "C"; "D"};
{\ar@{--}@[red] "E"; "F"};
{\ar@{--}@[red] "G"; "H"};
{\ar@{--}@[red] "I"; "J"};
{\ar@{.}@[|(2)]@<-.2ex>@[blue] "B"; "C"};
{\ar@{-}@[|(2)]@[green] "F"; "G"};
{\ar@{.}@[|(2)]@<-.2ex>@[blue] "A"; "H"};
{\ar@{.}@[|(2)]@[blue] "D"; "E"};
{\ar@{-}@[green] "E"; "I"};
{\ar@{-}@<.2ex>@[green] "B"; "C"};
{\ar@{-}@/^2ex/@[green] "D"; "J"};
{\ar@{.}@[|(2)]@<.2ex>@[green] "A"; "H"};
{\ar@{..}@[blue]@[|(2)] "F"; "I"};
{\ar@{..}@[blue]@[|(2)] "G"; "J"};
\end{xy}
\end{displaymath}
All the red edges other than $\uedge{9 \;  10}$ contain at least one vertex in
$U_G$ so we are allowed to Pl\"ucker any of them together to obtain a relation
of generalized Segre cubic data.  We thus may apply identity of
Figure~\ref{f:id8} to the red edges.  All the resulting terms either come from
fewer points by outer multiplication or else are degenerate.  (Precisely:  the
first two graphs from the first row of Figure~\ref{f:id8} come from fewer
points, and the last two are degenerate.  In the second row, all are degenerate
except for the third.  In the third row, all are degenerate except the first.)
We conclude that $\Rel(\Sigma)$ is retrogenerated.

Now let $\Gamma$ be the right graph.  By applying the Pl\"ucker relation to the
edges $\uedge{28}$ and $\uedge{36}$ we get two graphs, one of which is
degenerate, the other of which looks like the graph handled in the previous
paragraph.  Thus $\Rel(\Sigma)$ for the right graph is retrogenerated as well.
\end{proof}

We now establish the second case of Proposition~\ref{prop:retro10c}.  Thus
let $\ms{U}$ be a partition in which $|U_G|=6$.  In this case, both $U_R$ and
$U_B$ have cardinality two and so the situation is symmetric with respect
to these two colors.  We begin with the following observation:

\begin{lemma}
\label{lem:10b}
If there is a green-blue or green-red doubled edge in $U_G$ then $\Rel(\Sigma)$
is retrogenerated over $\Z[\half]$.
\end{lemma}

\begin{proof}
Say there is a green-blue doubled edge in $U_G$.  Then, as in the case
$|U_G|=4$, we can apply identity of Figure~\ref{f:id6} to force there to be a second
green-blue doubled edge in $U_G$.  The graph $\Gamma$ must now look like
one of the two in Lemma~\ref{lem:10a} and so, by that lemma, we obtain the
present one.
\end{proof}

Now consider a general generalized Segre cubic datum $\Sigma$ with
$|U_G|=6$.  By applying identity of Figure~\ref{f:id6} repeatedly, or appealing to
Proposition~\ref{prop:pfil2}, we can arrange it so that the blue-green
subgraph on $U_G$ is of the form
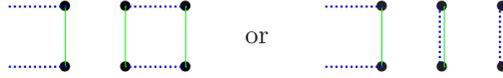
\begin{displaymath}
\begin{xy}
(0, -4)*{}="A1"; (0, 4)*{}="A2";
(8, -4)*{}="B1"; (8, 4)*{}="B2";
(16, -4)*{}="C1"; (16, 4)*{}="C2";
(24, -4)*{}="D1"; (24, 4)*{}="D2";
"B1"*{\bullet}; "B2"*{\bullet};
"C1"*{\bullet}; "C2"*{\bullet};
"D1"*{\bullet}; "D2"*{\bullet};
{\ar@[green]@{-} "B1"; "B2"};
{\ar@[blue]@{..}@[|(2)] "A1"; "B1"};
{\ar@[blue]@{..}@[|(2)] "A2"; "B2"};
{\ar@[blue]@{..}@[|(2)] "C1"; "D1"};
{\ar@[blue]@{..}@[|(2)] "C2"; "D2"};
{\ar@[green]@{-} "C1"; "C2"};
{\ar@[green]@{-} "D1"; "D2"};
\end{xy}
\qquad \textrm{or} \qquad
\begin{xy}
(0, -4)*{}="A1"; (0, 4)*{}="A2";
(8, -4)*{}="B1"; (8, 4)*{}="B2";
(16, -4)*{}="C1"; (16, 4)*{}="C2";
(24, -4)*{}="D1"; (24, 4)*{}="D2";
"B1"*{\bullet}; "B2"*{\bullet};
"C1"*{\bullet}; "C2"*{\bullet};
"D1"*{\bullet}; "D2"*{\bullet};
{\ar@[green]@{-} "B1"; "B2"};
{\ar@[blue]@{..}@[|(2)] "A1"; "B1"};
{\ar@[blue]@{..}@[|(2)] "A2"; "B2"};
{\ar@[blue]@{..}@[|(2)]@<.2ex> "C1"; "C2"};
{\ar@[green]@{-}@<-.2ex> "C1"; "C2"};
{\ar@[blue]@{..}@[|(2)]@<.2ex> "D1"; "D2"};
{\ar@[green]@{-}@<-.2ex> "D1"; "D2"};
\end{xy}
\end{displaymath}
The right graph has a blue-green doubled edge, and so the associated
relation belongs to $I_L'$ by Lemma~\ref{lem:10b}.  We therefore need
only consider the left graph.  We now consider the red edges.  There are
two that are completely contained in $U_G$ and the two special edges, each
of which has one vertex in $U_G$.  Here are three possible graphs:
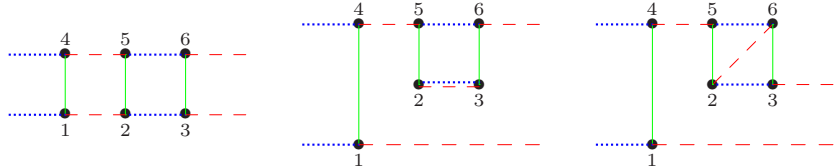
\begin{displaymath}
\begin{xy}
(0, -4)*{}="A1"; (0, 4)*{}="A2";
(8, -4)*{}="B1"; (8, 4)*{}="B2";
(16, -4)*{}="C1"; (16, 4)*{}="C2";
(24, -4)*{}="D1"; (24, 4)*{}="D2";
(32, -4)*{}="E1"; (32, 4)*{}="E2";
(8, -6)*{\ss 1}; (8, 6)*{\ss 4};
(16, -6)*{\ss 2}; (24, -6)*{\ss 3};
(16, 6)*{\ss 5}; (24, 6)*{\ss 6};
"B1"*{\bullet}; "B2"*{\bullet};
"C1"*{\bullet}; "C2"*{\bullet};
"D1"*{\bullet}; "D2"*{\bullet};
{\ar@[green]@{-} "B1"; "B2"};
{\ar@[blue]@{..}@[|(2)] "A1"; "B1"};
{\ar@[blue]@{..}@[|(2)] "A2"; "B2"};
{\ar@[blue]@{..}@[|(2)] "C1"; "D1"};
{\ar@[blue]@{..}@[|(2)] "C2"; "D2"};
{\ar@[green]@{-} "C1"; "C2"};
{\ar@[green]@{-} "D1"; "D2"};
{\ar@[red]@{--} "D1"; "E1"};
{\ar@[red]@{--} "D2"; "E2"};
{\ar@[red]@{--} "B1"; "C1"};
{\ar@[red]@{--} "B2"; "C2"};
\end{xy}
\qquad
\begin{xy}
(0, -8)*{}="A1"; (0, 8)*{}="A2";
(8, -8)*{}="B1"; (8, 8)*{}="B2";
(16, -8)*{}="C1"; (16, 8)*{}="C2";
(24, -8)*{}="D1"; (24, 8)*{}="D2";
(32, -8)*{}="E1"; (32, 8)*{}="E2";
(16, 0)*{}="F1"; (24, 0)*{}="F2";
(8, -10)*{\ss 1}; (8, 10)*{\ss 4};
(16, -2)*{\ss 2}; (24, -2)*{\ss 3};
(16, 10)*{\ss 5}; (24, 10)*{\ss 6};
"B1"*{\bullet}; "B2"*{\bullet};
"C2"*{\bullet}; "D2"*{\bullet};
"F1"*{\bullet}; "F2"*{\bullet};
{\ar@[green]@{-} "B1"; "B2"};
{\ar@[blue]@{..}@[|(2)] "A1"; "B1"};
{\ar@[blue]@{..}@[|(2)] "A2"; "B2"};
{\ar@[blue]@{..}@[|(2)] "C2"; "D2"};
{\ar@[blue]@{..}@[|(2)]@<.2ex> "F1"; "F2"};
{\ar@[red]@{--}@<-.2ex> "F1"; "F2"};
{\ar@[green]@{-} "C2"; "F1"};
{\ar@[green]@{-} "D2"; "F2"};
{\ar@[red]@{--} "D2"; "E2"};
{\ar@[red]@{--} "B1"; "E1"};
{\ar@[red]@{--} "B2"; "C2"};
\end{xy}
\qquad
\begin{xy}
(0, -8)*{}="A1"; (0, 8)*{}="A2";
(8, -8)*{}="B1"; (8, 8)*{}="B2";
(16, -8)*{}="C1"; (16, 8)*{}="C2";
(24, -8)*{}="D1"; (24, 8)*{}="D2";
(32, -8)*{}="E1"; (32, 8)*{}="E2";
(16, 0)*{}="F1"; (24, 0)*{}="F2"; (32, 0)*{}="F3";
(8, -10)*{\ss 1}; (8, 10)*{\ss 4};
(16, -2)*{\ss 2}; (24, -2)*{\ss 3};
(16, 10)*{\ss 5}; (24, 10)*{\ss 6};
"B1"*{\bullet}; "B2"*{\bullet};
"C2"*{\bullet}; "D2"*{\bullet};
"F1"*{\bullet}; "F2"*{\bullet};
{\ar@[green]@{-} "B1"; "B2"};
{\ar@[blue]@{..}@[|(2)] "A1"; "B1"};
{\ar@[blue]@{..}@[|(2)] "A2"; "B2"};
{\ar@[blue]@{..}@[|(2)] "C2"; "D2"};
{\ar@[blue]@{..}@[|(2)] "F1"; "F2"};
{\ar@[red]@{--} "F1"; "D2"};
{\ar@[green]@{-} "C2"; "F1"};
{\ar@[green]@{-} "D2"; "F2"};
{\ar@[red]@{--} "F2"; "F3"};
{\ar@[red]@{--} "B1"; "E1"};
{\ar@[red]@{--} "B2"; "C2"};
\end{xy}
\end{displaymath}
We will refer to these as type A, B and C, respectively.  Of course, there are
many other possibilities for what the graph could look like.  However, all
reduce to one of the above three types after some simple Pl\"ucker relations.
For instance, one could consider the graph which is like the type A one, but
where the blue and green colors are switched in the square 2-3-5-6.  By
applying the Pl\"ucker relation to the edges $\uedge{4 5}$ and the red special
edge containing 3, this relation is rewritten as a sum of two of those
appearing above.

We now have the following result:

\begin{lemma}
Type B and C relations are retrogenerated over $\Z[\half]$.
\end{lemma}

\begin{proof}
Consider the type B graph drawn above.  We apply identity of Figure~\ref{f:id6}
to the red-green chain 4-5-2-3-6-$\ast$, where $\ast$ is the relevant vertex of
$U_B$.  The first graph on the right side of the identity is disconnected and
therefore retrogenerated.  All the remaining graphs have green-red doubled
edges and so are also retrogenerated by Lemma~\ref{lem:10b}.  This completes
the type B case.  The type C case is handled similarly:  in the above labeling
one applies the identity of Figure~\ref{f:id6} to the chain 4-5-2-6-3-$\ast$
and then proceeds exactly as in the type B case.
\end{proof}

The following completes the proof of Proposition~\ref{prop:retro10c}.

\begin{lemma}
The type A relations span a subspace of $I_L^{(3)}/I_L^{\retro, (3)}$ of
dimension at most one over $\Z[\half]$.
\end{lemma}

\begin{proof}
Put $U_G=\{1,2,3,4,5,6\}$ and let $\Gamma$ be a 3-colored graph on $L$ whose
restriction to $U_G$ is the type A graph
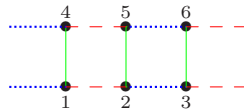
\begin{displaymath}
\begin{xy}
(0, -4)*{}="A1"; (0, 4)*{}="A2";
(8, -4)*{}="B1"; (8, 4)*{}="B2";
(16, -4)*{}="C1"; (16, 4)*{}="C2";
(24, -4)*{}="D1"; (24, 4)*{}="D2";
(32, -4)*{}="E1"; (32, 4)*{}="E2";
(8, -6)*{\ss 1}; (8, 6)*{\ss 4};
(16, -6)*{\ss 2}; (24, -6)*{\ss 3};
(16, 6)*{\ss 5}; (24, 6)*{\ss 6};
"B1"*{\bullet}; "B2"*{\bullet};
"C1"*{\bullet}; "C2"*{\bullet};
"D1"*{\bullet}; "D2"*{\bullet};
{\ar@[green]@{-} "B1"; "B2"};
{\ar@[blue]@{..}@[|(2)] "A1"; "B1"};
{\ar@[blue]@{..}@[|(2)] "A2"; "B2"};
{\ar@[blue]@{..}@[|(2)] "C1"; "D1"};
{\ar@[blue]@{..}@[|(2)] "C2"; "D2"};
{\ar@[green]@{-} "C1"; "C2"};
{\ar@[green]@{-} "D1"; "D2"};
{\ar@[red]@{--} "D1"; "E1"};
{\ar@[red]@{--} "D2"; "E2"};
{\ar@[red]@{--} "B1"; "C1"};
{\ar@[red]@{--} "B2"; "C2"};
\end{xy}
\end{displaymath}
Applying the Pl\"ucker relation to the edges $\uedge{4 5}$ and $\uedge{1 2}$
gives an expression for $\Gamma$ in terms of two other graphs $\Gamma_1$ and
$\Gamma_2$, one of which (say $\Gamma_1$) is degenerate.  Now apply the
Pl\"ucker relation to the edges $\uedge{5 6}$ and $\uedge{2 3}$ of $\Gamma_2$.
Again, this expresses $\Gamma_2$ in terms of two other graphs $\Gamma_3$ and
$\Gamma_4$, one of which (say $\Gamma_4$) is degenerate.  Thus, we see that
$\Rel(\Sigma)$ and $\Rel(\Sigma_4)$ are scalar multiples of each other.  But
$\Rel(\Sigma_4)$ is nothing other than $(2 5) \Rel(\Sigma)$.  We thus see that
$\Gamma$ and $(2 5) \Gamma$ yield the same relation (up to a scalar).  By
similar reasoning, we see that $(1 4)\Gamma$ and $(3  6)\Gamma$ give the same
relation as $\Gamma$.

Now consider the Pl\"ucker relation on $\Gamma$ on the edges $\uedge{12}$ and
$\uedge{3 \ast}$, where $\ast$ is the relevant vertex of $U_B$.  This results
in two graphs $\Gamma_1$ and $\Gamma_2$.  Say $\Gamma_1$ is the graph which
has the edge $\uedge{1 \ast}$.  Then $\Gamma_1$ is a type B graph and thus
retrogenerated.  Now apply the Pl\"ucker relation on the edges $\uedge{2 5}$
and $\uedge{3 6}$ to $\Gamma_2$ to get two new graphs $\Gamma_3$ and
$\Gamma_4$.  One of these, say $\Gamma_3$, has a green-blue doubled edge and
is thus retrogenerated.  The other, $\Gamma_4$, is just $(2 3)\Gamma$.  We
thus see that $\Gamma$ and $(2  3)\Gamma$ give the same relation, up to
a scalar.  Similar reasoning shows that $(4  5)\Gamma$, $(5  6)\Gamma$
and $(1  2)\Gamma$ give the same relation as $\Gamma$.

>From the previous two paragraphs, we see that if $\sigma$ is any permutation
of $U_G$ then $\Gamma$ and $\sigma \Gamma$ give the same relation, up to
a scalar.  Since every type A graph is of the form $\sigma \Gamma$ for some
$\sigma$, this establishes the lemma.
\end{proof}


\section{The ideal is generated by quadratics}
\label{s:base}

The goal of \S \ref{s:base} is to prove that $I_L$ is generated by quadratics
if $|L| \ge 8$:

\begin{theorem}
\label{thm:quadgen}
If $L$ is an even set of cardinality at least eight then $I_L= Q_L$ over
$\Z[\tfrac{1}{12!}]$.
\end{theorem}

As described in \S \ref{outline}, this concludes the proof of
Theorem~\ref{informal}.  Thanks to our work in \S \ref{s:retro}, the main
remaining work is to prove three base cases:

\begin{theorem}
\label{thm:base}
The ideals $I_8$, $I_{10}$ and $I_{12}$ are generated by quadratics over
$\Z[\tfrac{1}{12!}]$.
\end{theorem}

This implies Theorem~\ref{thm:quadgen} (and hence Theorem~\ref{informal}) by
the following inductive argument.  Let $n \ge 14$ and assume the theorem has
been established for even sets of cardinality less than $n$.  By Theorem~\ref{thm:retro}
we have $I_n=I_n^{\retro}$.  By Proposition~\ref{prop:outer246} we know that
$I_n^{\retro}$ is generated by the outer products of cubic relations on
$n-2$, $n-4$ and $n-6$ points with arbitrary graphs on 2, 4 and 6 points.
The inductive hypothesis, and the fact that $n \ge 14$, ensures that any
cubic relation on $n-2$, $n-4$ or $n-6$ points lies in the ideal generated by
quadratics.  Since outer multiplication does not increase essential degree
(Proposition~\ref{prop:stabdeg}), we find $I_n^{\retro}=Q_n$.  Thus
Theorem~\ref{thm:quadgen} is established by induction.

As for Theorem~\ref{thm:base}, quadratic generation of $I_8$ is proved over
$\Z[\tfrac{1}{3}]$ in \cite{8pts}.  (Quadratic generation of $I_8$ over $\Q$
had previously been established by computer; see e.g., \cite{Koike} or
\cite[Lemma~1.1]{fs}.)  We handle the 10 and 12 point cases by showing that
$I_{10}$ (resp.\ $I_{12}$) is generated by relations coming from 8 (resp.\ 8
and 10) points: one does not need relations coming from 6 points to generate
$I_L$.  The quadratic generation of $I_8$ then implies that of $I_{10}$ and
hence of $I_{12}$ as well.

Here is an overview of the argument.  We let $I'$ be the quotient of the cubic
part of the ideal by quadratic relations and relations retrogenerated from 8
(resp.\ 8 and 10) points.  We wish to show that this space is zero.  By
Theorem~\ref{thm:retro},  $I'$ is generated as an $\mf{S}_{10}$-module
(resp.\ $\mf{S}_{12}$-module) by a single relation:  the outer product of the
Segre cubic on 6 points by a benzene 4-cycle (resp.\ 6-cycle).  (Hence we only
need to prove that these two relations are generated by quadratics.  This
could in principle be checked by computer, but with current technology would
probably require an algorithm of equal difficulty to our proof!)  We can thus
write $I'$ as a quotient of a representation $V$ induced from $\mf{S}_6 \times
\mf{S}_4$ (resp.\ $\mf{S}_6 \times \mf{S}_6$), where the first factor acts on
the Segre cubic on 6 points and the second on the benzene cycle.  We then write
down a family of elements in the kernel of $V \to I'$, which allows us to
obtain an upper bound on the dimension of $I'$.  Finally, we examine the
irreducible $\mf{S}_{10}$-modules (resp.\ $\mf{S}_{12}$-modules) occurring in
$V$ and find that they all have larger dimension than the bound found for $I'$.
This shows that $I'$ is zero.  The fact that we use the representation theory
of $\mf{S}_{10}$ and $\mf{S}_{12}$ is the source of the denominator $12!$ in
Theorem~\ref{thm:quadgen}.
\Ncom{$I'$}

Throughout \S \ref{s:base}, $L$ will always denote a set of cardinality 10 or
12.  We handle the two cases simultaneously as much as possible.  When we
make different statements for the two cases, we give the 10-point statement
first and the 12-point statement second.

\subsection{The spaces $I_{\ms{P}}'$ and the A-, B- and C-relations}

By a \emph{bipartition} of $L$ we mean an ordered pair $\ms{P}=(P, Q)$ with
$P \coprod Q = L$,  where $P$ has cardinality six, and $Q$ has cardinality four
if $|L|=10$, and cardinality six if $|L|=12$.  By a graph of \emph{type A}
(resp.  \emph{type B}, resp.\ \emph{type C}) with respect to a bipartition
$\ms{P}$ we mean one of the form 
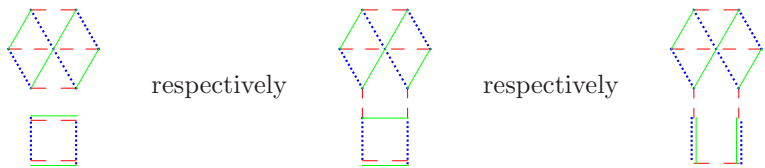
\Ncom{$\text{bipartition, type $A/B/C$}$}
\begin{displaymath}
\begin{xy}
(-3, 10.2)*{}="A"; (3, 10.2)*{}="B"; (6, 5)*{}="C";
(3, -.2)*{}="D"; (-3, -.2)*{}="E"; (-6, 5)*{}="F";
(-3, -4.2)*{}="G"; (3, -4.2)*{}="H"; (3, -10.2)*{}="I"; (-3, -10.2)*{}="J";
{\ar@{--}@[red] "A"; "B"};
{\ar@{--}@[red] "C"; "F"};
{\ar@{--}@[red] "D"; "E"};
{\ar@{--}@<-.2ex>@[red] "G"; "H"};
{\ar@{--}@<-.2ex>@[red] "I"; "J"};
{\ar@{.}@[|(2)]@[blue] "A"; "D"};
{\ar@{.}@[|(2)]@[blue] "B"; "C"};
{\ar@{.}@[|(2)]@[blue] "E"; "F"};
{\ar@{.}@[|(2)]@[blue] "H"; "I"};
{\ar@{.}@[|(2)]@[blue] "J"; "G"};
{\ar@{-}@[green] "A"; "F"};
{\ar@{-}@[green] "B"; "E"};
{\ar@{-}@[green] "C"; "D"};
{\ar@{-}@<.2ex>@[green] "G"; "H"};
{\ar@{-}@<.2ex>@[green] "I"; "J"};
\end{xy}
\qquad \textrm{respectively} \qquad
\begin{xy}
(-3, 10.2)*{}="A"; (3, 10.2)*{}="B"; (6, 5)*{}="C";
(3, -.2)*{}="D"; (-3, -.2)*{}="E"; (-6, 5)*{}="F";
(-3, -4.2)*{}="G"; (3, -4.2)*{}="H"; (3, -10.2)*{}="I"; (-3, -10.2)*{}="J";
{\ar@{--}@[red] "A"; "B"};
{\ar@{--}@[red] "C"; "F"};
{\ar@{--}@[red] "D"; "H"};
{\ar@{--}@[red] "E"; "G"};
{\ar@{--}@<-.2ex>@[red] "I"; "J"};
{\ar@{.}@[|(2)]@[blue] "A"; "D"};
{\ar@{.}@[|(2)]@[blue] "B"; "C"};
{\ar@{.}@[|(2)]@[blue] "E"; "F"};
{\ar@{.}@[|(2)]@[blue] "H"; "I"};
{\ar@{.}@[|(2)]@[blue] "J"; "G"};
{\ar@{-}@[green] "A"; "F"};
{\ar@{-}@[green] "B"; "E"};
{\ar@{-}@[green] "C"; "D"};
{\ar@{-}@[green] "G"; "H"};
{\ar@{-}@<.2ex>@[green] "I"; "J"};
\end{xy}
\qquad \textrm{respectively} \qquad
\begin{xy}
(-3, 10.2)*{}="A"; (3, 10.2)*{}="B"; (6, 5)*{}="C";
(3, -.2)*{}="D"; (-3, -.2)*{}="E"; (-6, 5)*{}="F";
(-3, -4.2)*{}="G"; (3, -4.2)*{}="H"; (3, -10.2)*{}="I"; (-3, -10.2)*{}="J";
{\ar@{--}@[red] "A"; "B"};
{\ar@{--}@[red] "C"; "F"};
{\ar@{--}@[red] "D"; "H"};
{\ar@{--}@[red] "E"; "G"};
{\ar@{--}@[red] "I"; "J"};
{\ar@{.}@[|(2)]@[blue] "A"; "D"};
{\ar@{.}@[|(2)]@[blue] "B"; "C"};
{\ar@{.}@[|(2)]@[blue] "E"; "F"};
{\ar@{.}@[|(2)]@<.2ex>@[blue] "H"; "I"};
{\ar@{.}@[|(2)]@<.2ex>@[blue] "J"; "G"};
{\ar@{-}@[green] "A"; "F"};
{\ar@{-}@[green] "B"; "E"};
{\ar@{-}@[green] "C"; "D"};
{\ar@{-}@<-.2ex>@[green] "H"; "I"};
{\ar@{-}@<-.2ex>@[green] "J"; "G"};
\end{xy}
\end{displaymath}
when $|L|=10$, or one of the form
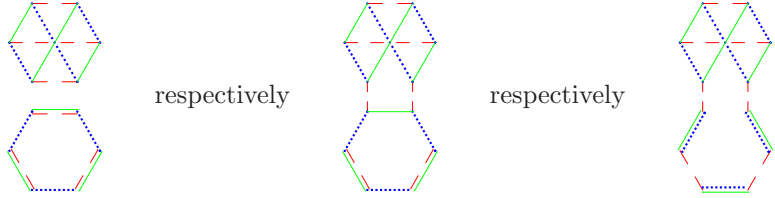
\begin{displaymath}
\begin{xy}
(-3, 12.39)*{}="A"; (3, 12.39)*{}="B"; (6, 7.195)*{}="C";
(3, 2)*{}="D"; (-3, 2)*{}="E"; (-6, 7.195)*{}="F";
(-3, -2)*{}="G"; (3, -2)*{}="H"; (6, -7.195)*{}="I";
(3, -12.39)*{}="J"; (-3, -12.39)*{}="K"; (-6, -7.195)*{}="L";
{\ar@{--}@[red] "A"; "B"};
{\ar@{--}@[red] "C"; "F"};
{\ar@{--}@[red] "D"; "E"};
{\ar@{--}@<-.2ex>@[red] "G"; "H"};
{\ar@{--}@<-.2ex>@[red] "I"; "J"};
{\ar@{--}@<-.2ex>@[red] "K"; "L"};
{\ar@{.}@[|(2)]@[blue] "A"; "D"};
{\ar@{.}@[|(2)]@[blue] "B"; "C"};
{\ar@{.}@[|(2)]@[blue] "E"; "F"};
{\ar@{.}@[|(2)]@[blue] "H"; "I"};
{\ar@{.}@[|(2)]@[blue] "J"; "K"};
{\ar@{.}@[|(2)]@[blue] "L"; "G"};
{\ar@{-}@[green] "A"; "F"};
{\ar@{-}@[green] "B"; "E"};
{\ar@{-}@[green] "C"; "D"};
{\ar@{-}@<.2ex>@[green] "G"; "H"};
{\ar@{-}@<.2ex>@[green] "I"; "J"};
{\ar@{-}@<.2ex>@[green] "K"; "L"};
\end{xy}
\qquad \textrm{respectively} \qquad
\begin{xy}
(-3, 12.39)*{}="A"; (3, 12.39)*{}="B"; (6, 7.195)*{}="C";
(3, 2)*{}="D"; (-3, 2)*{}="E"; (-6, 7.195)*{}="F";
(-3, -2)*{}="G"; (3, -2)*{}="H"; (6, -7.195)*{}="I";
(3, -12.39)*{}="J"; (-3, -12.39)*{}="K"; (-6, -7.195)*{}="L";
{\ar@{--}@[red] "A"; "B"};
{\ar@{--}@[red] "C"; "F"};
{\ar@{--}@[red] "D"; "H"};
{\ar@{--}@[red] "E"; "G"};
{\ar@{--}@<-.2ex>@[red] "I"; "J"};
{\ar@{--}@<-.2ex>@[red] "K"; "L"};
{\ar@{.}@[|(2)]@[blue] "A"; "D"};
{\ar@{.}@[|(2)]@[blue] "B"; "C"};
{\ar@{.}@[|(2)]@[blue] "E"; "F"};
{\ar@{.}@[|(2)]@[blue] "H"; "I"};
{\ar@{.}@[|(2)]@[blue] "J"; "K"};
{\ar@{.}@[|(2)]@[blue] "L"; "G"};
{\ar@{-}@[green] "A"; "F"};
{\ar@{-}@[green] "B"; "E"};
{\ar@{-}@[green] "C"; "D"};
{\ar@{-}@[green] "G"; "H"};
{\ar@{-}@<.2ex>@[green] "I"; "J"};
{\ar@{-}@<.2ex>@[green] "K"; "L"};
\end{xy}
\qquad \textrm{respectively} \qquad
\begin{xy}
(-3, 12.39)*{}="A"; (3, 12.39)*{}="B"; (6, 7.195)*{}="C";
(3, 2)*{}="D"; (-3, 2)*{}="E"; (-6, 7.195)*{}="F";
(-3, -2)*{}="G"; (3, -2)*{}="H"; (6, -7.195)*{}="I";
(3, -12.39)*{}="J"; (-3, -12.39)*{}="K"; (-6, -7.195)*{}="L";
{\ar@{--}@[red] "A"; "B"};
{\ar@{--}@[red] "C"; "F"};
{\ar@{--}@[red] "D"; "H"};
{\ar@{--}@[red] "E"; "G"};
{\ar@{--}@[red] "I"; "J"};
{\ar@{--}@[red] "K"; "L"};
{\ar@{.}@[|(2)]@[blue] "A"; "D"};
{\ar@{.}@[|(2)]@[blue] "B"; "C"};
{\ar@{.}@[|(2)]@[blue] "E"; "F"};
{\ar@{.}@[|(2)]@<-.2ex>@[blue] "H"; "I"};
{\ar@{.}@[|(2)]@<-.2ex>@[blue] "J"; "K"};
{\ar@{.}@[|(2)]@<-.2ex>@[blue] "L"; "G"};
{\ar@{-}@[green] "A"; "F"};
{\ar@{-}@[green] "B"; "E"};
{\ar@{-}@[green] "C"; "D"};
{\ar@{-}@<.2ex>@[green] "H"; "I"};
{\ar@{-}@<.2ex>@[green] "J"; "K"};
{\ar@{-}@<.2ex>@[green] "L"; "G"};
\end{xy}
\end{displaymath}
when $|L|=12$.  In all cases the vertices of the top hexagon belong to $P$
while the vertices in the bottom square or hexagon belong to $Q$.

We name some edges that we will refer to often.  The most obvious distinguished
edges are the two red edges joining the $P$-vertices and $Q$-vertices.  We call
them the \emph{special red edges}.  There is a unique blue (resp.\ green) edge
in the top hexagon meeting one of the special red edges; we call it the
\emph{(top) special blue (resp.\ green) edge}.  Finally, for type B graphs, we
call the green edge in the bottom hexagon which meets the special red edges the
\emph{bottom special green edge}.  We call the chain $a$-$b$-$c$-$d$, where
$\uedge{ab}$ is the special blue edge, $\uedge{bc}$ is the special red edge to
which it connects and $\uedge{cd}$ is the bottom special green edge, the
\emph{special chain} (only defined for type B graphs).
\Ncom{$\text{special edges and chains}$}

Let $\Gamma$ be a graph of type A, B or C with respect to $\ms{P}$.  Define a
partition $\ms{U}=\{U_R, U_G, U_B\}$ of $L$ as follows.  In the type A case,
$U_R$, $U_G$ and $U_B$ each contain two vertices in the top hexagon, and these
vertices must be connected by a red, green and blue edge respectively.  The
bottom benzene cycle is in $U_B$.  In the type B and C cases we take $U_G$
to be the two vertices belonging to the top special green edge, we take $U_B$
to be the two vertices belonging to the top special blue edge together with the
bottom set (that is, the set $Q$), and we take $U_R$ to be the leftover two
vertices in the top hexagon (they form a red edge).  In each case, $(\Gamma,
\ms{U})$ is a generalized Segre cubic datum, and so has an associated
generalized Segre cubic relation.  In the type A case the partition $\ms{U}$
was not uniquely determined, but its image in $I'_L$ is independent of
the choice of partition.  We may thus in all cases speak unambiguously of the
relation defined by $\Gamma$, denoted  by $\Rel(\Gamma)$.  Note that if
$\sigma$ is a permutation of $L$ then $\Rel(\sigma \Gamma)=\epsilon(\sigma)
\sigma  ( \Rel(\Gamma))$.

We let $I'_{\ms{P}}$ be the subspace of $I'_L$ generated by the relations
attached to all type A graphs with respect to $\ms{P}$.  Now, $I'_{\ms{P}}$
carries a natural representation of $\mf{S}_P \times \mf{S}_Q$.  As such it is
a quotient of the representation $I_P^{(3)} \otimes \gr_{4, Q}^{(3)}$
(resp.\ $I_P^{(3)} \otimes \gr_{6, Q}^{(3)}$).  In particular, it is either
zero or irreducible of dimension one (resp.\ five), as $I_P^{(3)}$ is
irreducible of dimension one and $\gr_{4, Q}^{(3)}$ (resp.\ $\gr_{6, Q}^{(3)}$)
is irreducible of dimension one (resp.\ five; see Propositions~\ref{prop:g4}
and \ref{prop:g6}).  The natural map (of $\mf{S}_L$modules)  \label{placeone}
\begin{displaymath}
\bigoplus_{\ms{P}} I'_{\ms{P}} \to I'_L
\end{displaymath}
is surjective since $I_L$ is retrogenerated (Theorem~\ref{thm:retro}) and
there are no relations on 2 or 4 points.

\begin{proposition}
\label{prop-abc-ip}
Let $\ms{P}$ be a bipartition of $L$.  Then the type B relations with respect
to $\ms{P}$ belong to $I'_{\ms{P}}$ and span it.  The same is true for the
type C relations.
\end{proposition}

We prove this proposition by showing that type A relations can be written in
terms of type B relations, type B in terms of type C and finally type C in
terms of type A, all with respect to the same $\ms{P}$.  We accomplish
this in a series of lemmas.  After this proof we will have no need for the
type C relations.

\begin{lemma}
Any type A relation with respect to $\ms{P}$ belongs to the space spanned by
the type B relations with respect to $\ms{P}$.
\end{lemma}

\begin{proof}
Given a type A graph, Pl\"ucker a red edge from the top hexagon and a red
edge from the bottom benzene cycle.  This expresses the type A relation as
a sum of two type B relations.
\end{proof}

\begin{lemma}
Any type B relation with respect to $\ms{P}$ belongs to the space spanned by
the type C relations with respect to $\ms{P}$.
\end{lemma}

\begin{proof}
To establish this lemma we just have to rewrite the blue/green graph on $Q$
in terms of squares of matchings.  This is exactly what we did in
when proving Theorem~\ref{thm:retro}(a) (see in particular
Lemmas~\ref{lem:retro-a} and~\ref{lem:retro-b}).  The key point is that
all square rotation relations which come up in our particular case have special paths of even length and are therefore quadratic.

Alternatively, one can argue as follows. 
The generalized Segre relation
associated to graphs of type B is a genuine cubic relation, that is, the
purple/black graph on the right side of the relation factors into a product
of degree one graphs.  For instance, in the 12 point case the factored 
relation may be written as
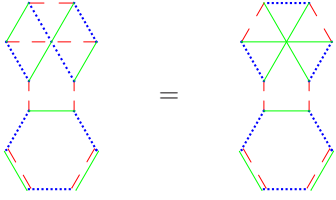
\begin{displaymath}
\begin{xy}
(-3, 12.39)*{}="A"; (3, 12.39)*{}="B"; (6, 7.195)*{}="C";
(3, 2)*{}="D"; (-3, 2)*{}="E"; (-6, 7.195)*{}="F";
(-3, -2)*{}="G"; (3, -2)*{}="H"; (6, -7.195)*{}="I";
(3, -12.39)*{}="J"; (-3, -12.39)*{}="K"; (-6, -7.195)*{}="L";
{\ar@{--}@[red] "A"; "B"};
{\ar@{--}@[red] "C"; "F"};
{\ar@{--}@[red] "D"; "H"};
{\ar@{--}@[red] "E"; "G"};
{\ar@{--}@<-.2ex>@[red] "I"; "J"};
{\ar@{--}@<-.2ex>@[red] "K"; "L"};
{\ar@{.}@[|(2)]@[blue] "A"; "D"};
{\ar@{.}@[|(2)]@[blue] "B"; "C"};
{\ar@{.}@[|(2)]@[blue] "E"; "F"};
{\ar@{.}@[|(2)]@[blue] "H"; "I"};
{\ar@{.}@[|(2)]@[blue] "J"; "K"};
{\ar@{.}@[|(2)]@[blue] "L"; "G"};
{\ar@{-}@[green] "A"; "F"};
{\ar@{-}@[green] "B"; "E"};
{\ar@{-}@[green] "C"; "D"};
{\ar@{-}@[green] "G"; "H"};
{\ar@{-}@<.2ex>@[green] "I"; "J"};
{\ar@{-}@<.2ex>@[green] "K"; "L"};
\end{xy}
\qquad = \qquad
\begin{xy}
(-3, 12.39)*{}="A"; (3, 12.39)*{}="B"; (6, 7.195)*{}="C";
(3, 2)*{}="D"; (-3, 2)*{}="E"; (-6, 7.195)*{}="F";
(-3, -2)*{}="G"; (3, -2)*{}="H"; (6, -7.195)*{}="I";
(3, -12.39)*{}="J"; (-3, -12.39)*{}="K"; (-6, -7.195)*{}="L";
{\ar@{--}@[red] "D"; "H"};
{\ar@{--}@[red] "B"; "C"};
{\ar@{--}@[red] "A"; "F"};
{\ar@{--}@[red] "E"; "G"};
{\ar@{--}@<-.2ex>@[red] "I"; "J"};
{\ar@{--}@<-.2ex>@[red] "K"; "L"};
{\ar@{.}@[|(2)]@[blue] "A"; "B"};
{\ar@{.}@[|(2)]@[blue] "E"; "F"};
{\ar@{.}@[|(2)]@[blue] "H"; "I"};
{\ar@{.}@[|(2)]@[blue] "J"; "K"};
{\ar@{.}@[|(2)]@[blue] "L"; "G"};
{\ar@{.}@[|(2)]@[blue] "C"; "D"};
{\ar@{-}@[green] "A"; "D"};
{\ar@{-}@[green] "C"; "F"};
{\ar@{-}@[green] "B"; "E"};
{\ar@{-}@[green] "G"; "H"};
{\ar@{-}@<.2ex>@[green] "I"; "J"};
{\ar@{-}@<.2ex>@[green] "K"; "L"};
\end{xy}
\end{displaymath}
(the black graph on the right is colored blue while the purple graph has
been factored into the green and red graphs).  Now, the bottom half of each
side is the same.  One can therefore apply the same colored Pl\"ucker and
quadratic relations on each side and rewrite the blue/green graph on $Q$ as 
a sum of squares of matchings.  This expresses the type B relation in terms
of type C relations.
\end{proof}

\begin{lemma}
Any type C relation with respect to $\ms{P}$ belongs to the space spanned by
the type A relations with respect to $\ms{P}$.
\end{lemma}

\begin{proof}
The proof of the 10 point case proceeds exactly as the proof of
Lemma~\ref{lem:10a}.  Note that after applying the identity of
Figure~\ref{f:id8} to the redrawn graph, the first graph on the second line of
the right side is a type A graph with respect to $\ms{P}$.  All other graphs
either come from 8 points or else are degenerate.  To prove the 12 point case,
we again use the identity of Figure~\ref{f:id8}, applying it to the bottom four
red edges, that is, the two special red edges and the two red edges contained
in the bottom hexagon.  The first term in Figure~\ref{f:id8} is a type A
relation with respect to $\ms{P}$.  All the other terms come from 8 or 10
points.
\end{proof}

\subsection{Relations among different  $I'_{\ms{P}}$}

In this section we demonstrate some linear dependencies between the various
spaces $I'_{\ms{P}}$:

\begin{proposition}
\label{prop-A-rel}
Let $\Gamma$ be a type A graph with respect to a bipartition $\ms{P}=(P,Q)$,
let $a$ and $b$ be two distinct elements of $P$ and let $c$ and $d$ be two
distinct elements of $Q$ which are joined by a doubled edge.
Then $\Rel(\Gamma)$ is contained in $\sum I'_{\sigma \ms{P}}$, where
the sum is taken over those permutations $\sigma$ of $\{a,b,c,d\}$ 
for which $\sigma \ms{P} \ne \ms{P}$.
\end{proposition}

We deduce this proposition from the following one:

\begin{proposition}
\label{prop-B-rel}
Let $\Gamma$ be a type B graph with respect to $\ms{P}$ and let $a$-$b$-$c$-$d$
be its special chain.  Then
\begin{displaymath}
\Rel(\Gamma)=(a  d)\Rel(\Gamma) - (b  c) \Rel(\Gamma)
+(a  d) (b  c) \Rel(\Gamma)+\Rel(\Delta)
\end{displaymath}  
where $\Delta$ is a type A graph with respect to
$(b  d) \ms{P}$.  In particular, $\Rel(\Gamma)$ belongs to $\sum I'_{\sigma
\ms{P}}$, the sum taken over those permutations $\sigma$ of $\{ a,b,c,d\}$
for which $\sigma \ms{P} \ne \ms{P}$.
\end{proposition}

\begin{proof}[Proof of Proposition~\ref{prop-A-rel}
given Proposition~\ref{prop-B-rel}] 
Let $\Gamma$, $a$, $b$, $c$ and $d$ be given as in the statement of 
Proposition~\ref{prop-A-rel}.  The vertices $c$ and $d$ are connected by
red and green edges.  Now, using a quadratic relation we may recolor the
$P$-part of $\Gamma$ so that $a$ and $b$ are connected by a blue edge.  Since
$\mf{S}_P$ acts on the Segre cubic relation on $P$ via the sign character,
all blue edges in $P$ are more or less the same so we may draw $\Gamma$ as:
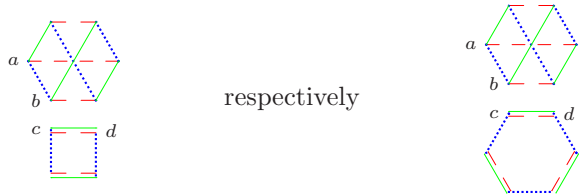
\begin{displaymath}
\begin{xy}
(-3, 10.2)*{}="A"; (3, 10.2)*{}="B"; (6, 5)*{}="C";
(3, -.2)*{}="D"; (-3, -.2)*{}="E"; (-6, 5)*{}="F";
(-3, -4.2)*{}="G"; (3, -4.2)*{}="H"; (3, -10.2)*{}="I"; (-3, -10.2)*{}="J";
(-8, 5)*{\ss a}; (-5, -.2)*{\ss b}; (-5, -4.2)*{\ss c}; (5, -4.2)*{\ss d};
{\ar@{--}@[red] "A"; "B"};
{\ar@{--}@[red] "C"; "F"};
{\ar@{--}@[red] "D"; "E"};
{\ar@{--}@<-.2ex>@[red] "G"; "H"};
{\ar@{--}@<-.2ex>@[red] "I"; "J"};
{\ar@{.}@[|(2)]@[blue] "A"; "D"};
{\ar@{.}@[|(2)]@[blue] "B"; "C"};
{\ar@{.}@[|(2)]@[blue] "E"; "F"};
{\ar@{.}@[|(2)]@[blue] "H"; "I"};
{\ar@{.}@[|(2)]@[blue] "J"; "G"};
{\ar@{-}@[green] "A"; "F"};
{\ar@{-}@[green] "B"; "E"};
{\ar@{-}@[green] "C"; "D"};
{\ar@{-}@<.2ex>@[green] "G"; "H"};
{\ar@{-}@<.2ex>@[green] "I"; "J"};
\end{xy}
\qquad\qquad \textrm{respectively} \qquad\qquad
\begin{xy}
(-3, 12.39)*{}="A"; (3, 12.39)*{}="B"; (6, 7.195)*{}="C";
(3, 2)*{}="D"; (-3, 2)*{}="E"; (-6, 7.195)*{}="F";
(-3, -2)*{}="G"; (3, -2)*{}="H"; (6, -7.195)*{}="I";
(3, -12.39)*{}="J"; (-3, -12.39)*{}="K"; (-6, -7.195)*{}="L";
(-8, 7.195)*{\ss a}; (-5, 2)*{\ss b}; (-5, -2)*{\ss c}; (5, -2)*{\ss d};
{\ar@{--}@[red] "A"; "B"};
{\ar@{--}@[red] "C"; "F"};
{\ar@{--}@[red] "D"; "E"};
{\ar@{--}@<-.2ex>@[red] "G"; "H"};
{\ar@{--}@<-.2ex>@[red] "I"; "J"};
{\ar@{--}@<-.2ex>@[red] "K"; "L"};
{\ar@{.}@[|(2)]@[blue] "A"; "D"};
{\ar@{.}@[|(2)]@[blue] "B"; "C"};
{\ar@{.}@[|(2)]@[blue] "E"; "F"};
{\ar@{.}@[|(2)]@[blue] "H"; "I"};
{\ar@{.}@[|(2)]@[blue] "J"; "K"};
{\ar@{.}@[|(2)]@[blue] "L"; "G"};
{\ar@{-}@[green] "A"; "F"};
{\ar@{-}@[green] "B"; "E"};
{\ar@{-}@[green] "C"; "D"};
{\ar@{-}@<.2ex>@[green] "G"; "H"};
{\ar@{-}@<.2ex>@[green] "I"; "J"};
{\ar@{-}@<.2ex>@[green] "K"; "L"};
\end{xy}
\end{displaymath}
We now Pl\"ucker the red edge joining $c$ and $d$ and the unique red edge
containing $b$.  We obtain an expression $\Gamma=\Gamma_1+\Gamma_2$ where
each $\Gamma_I$ is a type B graph with respect to $\ms{P}$.  The special chain
in $\Gamma_1$ is $a$-$b$-$c$-$d$ while in $\Gamma_2$ it is
$a$-$b$-$d$-$c$.  By Proposition~\ref{prop-B-rel}, both $\Rel(\Gamma_1)$ and
$\Rel(\Gamma_2)$ belong to   $\sum I'_{\sigma \ms{P}}$, where the sum is over
those permutations $\sigma$ of $\{ a,b,c,d \}$ for which  $\sigma \ms{P} \neq
\ms{P}$.  Thus $\Rel(\Gamma)$ belongs to this space as well, which completes
the proof.
\end{proof}

We now begin proving Proposition~\ref{prop-B-rel}.  Consider a type B graph
$\Gamma$.  For convenience, we label it:
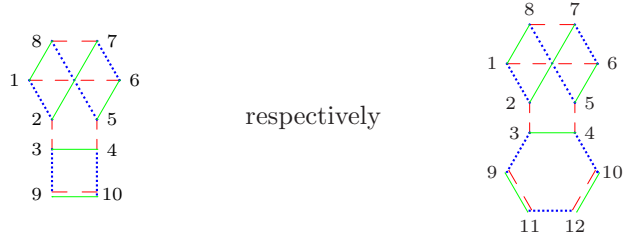
\begin{displaymath}
\begin{xy}
(-3, 10.2)*{}="A"; (3, 10.2)*{}="B"; (6, 5)*{}="C";
(3, -.2)*{}="D"; (-3, -.2)*{}="E"; (-6, 5)*{}="F";
(-3, -4.2)*{}="G"; (3, -4.2)*{}="H"; (3, -10.2)*{}="I"; (-3, -10.2)*{}="J";
(-8, 5)*{\ss 1}; (-5, -.2)*{\ss 2}; (-5, -4.2)*{\ss 3}; (5, -4.3)*{\ss 4};
(5, -.2)*{\ss 5}; (8, 5)*{\ss 6}; (5, 10.2)*{\ss 7}; (-5, 10.2)*{\ss 8};
(-5, -10.2)*{\ss 9}; (5, -10.2)*{\ss 10};
{\ar@{--}@[red] "A"; "B"};
{\ar@{--}@[red] "C"; "F"};
{\ar@{--}@[red] "D"; "H"};
{\ar@{--}@[red] "E"; "G"};
{\ar@{--}@<-.2ex>@[red] "I"; "J"};
{\ar@{.}@[|(2)]@[blue] "A"; "D"};
{\ar@{.}@[|(2)]@[blue] "B"; "C"};
{\ar@{.}@[|(2)]@[blue] "E"; "F"};
{\ar@{.}@[|(2)]@[blue] "H"; "I"};
{\ar@{.}@[|(2)]@[blue] "J"; "G"};
{\ar@{-}@[green] "A"; "F"};
{\ar@{-}@[green] "B"; "E"};
{\ar@{-}@[green] "C"; "D"};
{\ar@{-}@[green] "G"; "H"};
{\ar@{-}@<.2ex>@[green] "I"; "J"};
\end{xy}
\qquad\qquad \textrm{respectively} \qquad\qquad
\begin{xy}
(-3, 12.39)*{}="A"; (3, 12.39)*{}="B"; (6, 7.195)*{}="C";
(3, 2)*{}="D"; (-3, 2)*{}="E"; (-6, 7.195)*{}="F";
(-3, -2)*{}="G"; (3, -2)*{}="H"; (6, -7.195)*{}="I";
(3, -12.39)*{}="J"; (-3, -12.39)*{}="K"; (-6, -7.195)*{}="L";
(-5.5, 2)*{\ss 2}; (5, 2)*{\ss 5}; (8, 7.195)*{\ss 6}; (3, 14.39)*{\ss 7};
(-3, 14.39)*{\ss 8}; (-8, 7.195)*{\ss 1};
(-5, -2)*{\ss 3}; (5, -2)*{\ss 4}; (8, -7.195)*{\ss 10}; (3, -14.39)*{\ss 12};
(-3, -14.39)*{\ss 11}; (-8, -7.195)*{\ss 9};
{\ar@{--}@[red] "A"; "B"};
{\ar@{--}@[red] "C"; "F"};
{\ar@{--}@[red] "D"; "H"};
{\ar@{--}@[red] "E"; "G"};
{\ar@{--}@<-.2ex>@[red] "I"; "J"};
{\ar@{--}@<-.2ex>@[red] "K"; "L"};
{\ar@{.}@[|(2)]@[blue] "A"; "D"};
{\ar@{.}@[|(2)]@[blue] "B"; "C"};
{\ar@{.}@[|(2)]@[blue] "E"; "F"};
{\ar@{.}@[|(2)]@[blue] "H"; "I"};
{\ar@{.}@[|(2)]@[blue] "J"; "K"};
{\ar@{.}@[|(2)]@[blue] "L"; "G"};
{\ar@{-}@[green] "A"; "F"};
{\ar@{-}@[green] "B"; "E"};
{\ar@{-}@[green] "C"; "D"};
{\ar@{-}@[green] "G"; "H"};
{\ar@{-}@<.2ex>@[green] "I"; "J"};
{\ar@{-}@<.2ex>@[green] "K"; "L"};
\end{xy}
\end{displaymath}
With this labeling, the special chain $a$-$b$-$c$-$d$ is $1$-$2$-$3$-$4$.  We
let $\ms{P}$ be the relevant bipartition.

By Pl\"uckering the edges $\uedge{2 7}$ and $\uedge{3 4}$ we obtain $\Gamma
=-\Gamma_2 -\Gamma_3$ where $\Gamma_2$ and $\Gamma_3$ are:
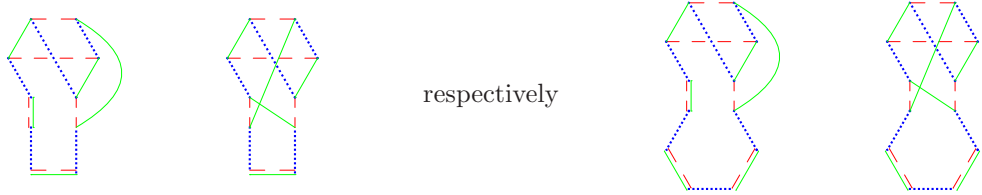
\begin{displaymath}
\begin{xy}
(-3, 10.2)*{}="A"; (3, 10.2)*{}="B"; (6, 5)*{}="C";
(3, -.2)*{}="D"; (-3, -.2)*{}="E"; (-6, 5)*{}="F";
(-3, -4.2)*{}="G"; (3, -4.2)*{}="H"; (3, -10.2)*{}="I"; (-3, -10.2)*{}="J";
{\ar@{--}@[red] "A"; "B"};
{\ar@{--}@[red] "C"; "F"};
{\ar@{--}@[red] "D"; "H"};
{\ar@{--}@<-.2ex>@[red] "E"; "G"};
{\ar@{--}@<-.2ex>@[red] "I"; "J"};
{\ar@{.}@[|(2)]@[blue] "A"; "D"};
{\ar@{.}@[|(2)]@[blue] "B"; "C"};
{\ar@{.}@[|(2)]@[blue] "E"; "F"};
{\ar@{.}@[|(2)]@[blue] "H"; "I"};
{\ar@{.}@[|(2)]@[blue] "J"; "G"};
{\ar@{-}@[green] "A"; "F"};
{\ar@{-}@/_4ex/@[green] "H"; "B"};
{\ar@{-}@[green] "C"; "D"};
{\ar@{-}@<.2ex>@[green] "E"; "G"};
{\ar@{-}@<.2ex>@[green] "I"; "J"};
\end{xy}
\qquad \qquad
\begin{xy}
(-3, 10.2)*{}="A"; (3, 10.2)*{}="B"; (6, 5)*{}="C";
(3, -.2)*{}="D"; (-3, -.2)*{}="E"; (-6, 5)*{}="F";
(-3, -4.2)*{}="G"; (3, -4.2)*{}="H"; (3, -10.2)*{}="I"; (-3, -10.2)*{}="J";
{\ar@{--}@[red] "A"; "B"};
{\ar@{--}@[red] "C"; "F"};
{\ar@{--}@[red] "D"; "H"};
{\ar@{--}@[red] "E"; "G"};
{\ar@{--}@<-.2ex>@[red] "I"; "J"};
{\ar@{.}@[|(2)]@[blue] "A"; "D"};
{\ar@{.}@[|(2)]@[blue] "B"; "C"};
{\ar@{.}@[|(2)]@[blue] "E"; "F"};
{\ar@{.}@[|(2)]@[blue] "H"; "I"};
{\ar@{.}@[|(2)]@[blue] "J"; "G"};
{\ar@{-}@[green] "A"; "F"};
{\ar@{-}@[green] "G"; "B"};
{\ar@{-}@[green] "C"; "D"};
{\ar@{-}@[green] "E"; "H"};
{\ar@{-}@<.2ex>@[green] "I"; "J"};
\end{xy}
\qquad \qquad \textrm{respectively} \qquad \qquad
\begin{xy}
(-3, 12.39)*{}="A"; (3, 12.39)*{}="B"; (6, 7.195)*{}="C";
(3, 2)*{}="D"; (-3, 2)*{}="E"; (-6, 7.195)*{}="F";
(-3, -2)*{}="G"; (3, -2)*{}="H"; (6, -7.195)*{}="I";
(3, -12.39)*{}="J"; (-3, -12.39)*{}="K"; (-6, -7.195)*{}="L";
{\ar@{--}@[red] "A"; "B"};
{\ar@{--}@[red] "C"; "F"};
{\ar@{--}@[red] "D"; "H"};
{\ar@{--}@<-.2ex>@[red] "E"; "G"};
{\ar@{--}@<-.2ex>@[red] "I"; "J"};
{\ar@{--}@<-.2ex>@[red] "K"; "L"};
{\ar@{.}@[|(2)]@[blue] "A"; "D"};
{\ar@{.}@[|(2)]@[blue] "B"; "C"};
{\ar@{.}@[|(2)]@[blue] "E"; "F"};
{\ar@{.}@[|(2)]@[blue] "H"; "I"};
{\ar@{.}@[|(2)]@[blue] "J"; "K"};
{\ar@{.}@[|(2)]@[blue] "L"; "G"};
{\ar@{-}@[green] "A"; "F"};
{\ar@{-}@<.2ex>@[green] "E"; "G"};
{\ar@{-}@[green] "C"; "D"};
{\ar@{-}@/_4ex/@[green] "H"; "B"};
{\ar@{-}@<.2ex>@[green] "I"; "J"};
{\ar@{-}@<.2ex>@[green] "K"; "L"};
\end{xy}
\qquad \qquad
\begin{xy}
(-3, 12.39)*{}="A"; (3, 12.39)*{}="B"; (6, 7.195)*{}="C";
(3, 2)*{}="D"; (-3, 2)*{}="E"; (-6, 7.195)*{}="F";
(-3, -2)*{}="G"; (3, -2)*{}="H"; (6, -7.195)*{}="I";
(3, -12.39)*{}="J"; (-3, -12.39)*{}="K"; (-6, -7.195)*{}="L";
{\ar@{--}@[red] "A"; "B"};
{\ar@{--}@[red] "C"; "F"};
{\ar@{--}@[red] "D"; "H"};
{\ar@{--}@[red] "E"; "G"};
{\ar@{--}@<-.2ex>@[red] "I"; "J"};
{\ar@{--}@<-.2ex>@[red] "K"; "L"};
{\ar@{.}@[|(2)]@[blue] "A"; "D"};
{\ar@{.}@[|(2)]@[blue] "B"; "C"};
{\ar@{.}@[|(2)]@[blue] "E"; "F"};
{\ar@{.}@[|(2)]@[blue] "H"; "I"};
{\ar@{.}@[|(2)]@[blue] "J"; "K"};
{\ar@{.}@[|(2)]@[blue] "L"; "G"};
{\ar@{-}@[green] "A"; "F"};
{\ar@{-}@[green] "H"; "E"};
{\ar@{-}@[green] "C"; "D"};
{\ar@{-}@[green] "G"; "B"};
{\ar@{-}@<.2ex>@[green] "I"; "J"};
{\ar@{-}@<.2ex>@[green] "K"; "L"};
\end{xy}
\end{displaymath}
We now Pl\"ucker the edges $\uedge{12}$ and $\uedge{39}$ in $\Gamma_3$ to
obtain a relation $\Gamma_3 = - \Gamma_4 - \Gamma_5$ where $\Gamma_4$ and
$\Gamma_5$ are given by:
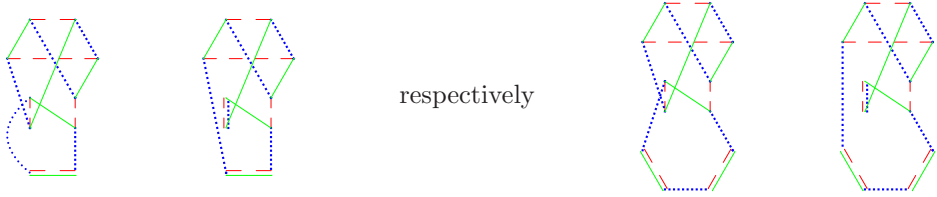
\begin{displaymath}
\begin{xy}
(-3, 10.2)*{}="A"; (3, 10.2)*{}="B"; (6, 5)*{}="C";
(3, -.2)*{}="D"; (-3, -.2)*{}="E"; (-6, 5)*{}="F";
(-3, -4.2)*{}="G"; (3, -4.2)*{}="H"; (3, -10.2)*{}="I"; (-3, -10.2)*{}="J";
{\ar@{--}@[red] "A"; "B"};
{\ar@{--}@[red] "C"; "F"};
{\ar@{--}@[red] "D"; "H"};
{\ar@{--}@[red] "E"; "G"};
{\ar@{--}@<-.2ex>@[red] "I"; "J"};
{\ar@{.}@[|(2)]@[blue] "A"; "D"};
{\ar@{.}@[|(2)]@[blue] "B"; "C"};
{\ar@{.}@[|(2)]@/_2ex/@[blue] "E"; "J"};
{\ar@{.}@[|(2)]@[blue] "F"; "G"};
{\ar@{.}@[|(2)]@[blue] "I"; "H"};
{\ar@{-}@[green] "A"; "F"};
{\ar@{-}@[green] "G"; "B"};
{\ar@{-}@[green] "C"; "D"};
{\ar@{-}@[green] "E"; "H"};
{\ar@{-}@<.2ex>@[green] "I"; "J"};
\end{xy}
\qquad \qquad
\begin{xy}
(-3, 10.2)*{}="A"; (3, 10.2)*{}="B"; (6, 5)*{}="C";
(3, -.2)*{}="D"; (-3, -.2)*{}="E"; (-6, 5)*{}="F";
(-3, -4.2)*{}="G"; (3, -4.2)*{}="H"; (3, -10.2)*{}="I"; (-3, -10.2)*{}="J";
{\ar@{--}@[red] "A"; "B"};
{\ar@{--}@[red] "C"; "F"};
{\ar@{--}@[red] "D"; "H"};
{\ar@{--}@<-.2ex>@[red] "E"; "G"};
{\ar@{--}@<-.2ex>@[red] "I"; "J"};
{\ar@{.}@[|(2)]@[blue] "A"; "D"};
{\ar@{.}@[|(2)]@[blue] "B"; "C"};
{\ar@{.}@[|(2)]@<.2ex>@[blue] "E"; "G"};
{\ar@{.}@[|(2)]@[blue] "F"; "J"};
{\ar@{.}@[|(2)]@[blue] "I"; "H"};
{\ar@{-}@[green] "A"; "F"};
{\ar@{-}@[green] "G"; "B"};
{\ar@{-}@[green] "C"; "D"};
{\ar@{-}@[green] "E"; "H"};
{\ar@{-}@<.2ex>@[green] "I"; "J"};
\end{xy}
\qquad \qquad \textrm{respectively} \qquad \qquad
\begin{xy}
(-3, 12.39)*{}="A"; (3, 12.39)*{}="B"; (6, 7.195)*{}="C";
(3, 2)*{}="D"; (-3, 2)*{}="E"; (-6, 7.195)*{}="F";
(-3, -2)*{}="G"; (3, -2)*{}="H"; (6, -7.195)*{}="I";
(3, -12.39)*{}="J"; (-3, -12.39)*{}="K"; (-6, -7.195)*{}="L";
{\ar@{--}@[red] "A"; "B"};
{\ar@{--}@[red] "C"; "F"};
{\ar@{--}@[red] "D"; "H"};
{\ar@{--}@[red] "E"; "G"};
{\ar@{--}@<-.2ex>@[red] "I"; "J"};
{\ar@{--}@<-.2ex>@[red] "K"; "L"};
{\ar@{.}@[|(2)]@[blue] "A"; "D"};
{\ar@{.}@[|(2)]@[blue] "B"; "C"};
{\ar@{.}@[|(2)]@[blue] "G"; "F"};
{\ar@{.}@[|(2)]@[blue] "H"; "I"};
{\ar@{.}@[|(2)]@[blue] "J"; "K"};
{\ar@{.}@[|(2)]@[blue] "L"; "E"};
{\ar@{-}@[green] "A"; "F"};
{\ar@{-}@[green] "H"; "E"};
{\ar@{-}@[green] "C"; "D"};
{\ar@{-}@[green] "G"; "B"};
{\ar@{-}@<.2ex>@[green] "I"; "J"};
{\ar@{-}@<.2ex>@[green] "K"; "L"};
\end{xy}
\qquad \qquad
\begin{xy}
(-3, 12.39)*{}="A"; (3, 12.39)*{}="B"; (6, 7.195)*{}="C";
(3, 2)*{}="D"; (-3, 2)*{}="E"; (-6, 7.195)*{}="F";
(-3, -2)*{}="G"; (3, -2)*{}="H"; (6, -7.195)*{}="I";
(3, -12.39)*{}="J"; (-3, -12.39)*{}="K"; (-6, -7.195)*{}="L";
{\ar@{--}@[red] "A"; "B"};
{\ar@{--}@[red] "C"; "F"};
{\ar@{--}@[red] "D"; "H"};
{\ar@{--}@<-.2ex>@[red] "E"; "G"};
{\ar@{--}@<-.2ex>@[red] "I"; "J"};
{\ar@{--}@<-.2ex>@[red] "K"; "L"};
{\ar@{.}@[|(2)]@[blue] "A"; "D"};
{\ar@{.}@[|(2)]@[blue] "B"; "C"};
{\ar@{.}@[|(2)]@<.2ex>@[blue] "E"; "G"};
{\ar@{.}@[|(2)]@[blue] "H"; "I"};
{\ar@{.}@[|(2)]@[blue] "J"; "K"};
{\ar@{.}@[|(2)]@[blue] "L"; "F"};
{\ar@{-}@[green] "A"; "F"};
{\ar@{-}@[green] "H"; "E"};
{\ar@{-}@[green] "C"; "D"};
{\ar@{-}@[green] "G"; "B"};
{\ar@{-}@<.2ex>@[green] "I"; "J"};
{\ar@{-}@<.2ex>@[green] "K"; "L"};
\end{xy}
\end{displaymath}

Now, we have $\Gamma_4=(2  3) \Gamma$.  Thus (as generalized Segre data)
\begin{equation}
\label{eq-gam}
\Gamma - (2  3) \Gamma +\Gamma_2 = \Gamma_5.
\end{equation}

\begin{lemma}
\label{lem-gam2}
We have $(1 4) \Rel(\Gamma_2) = \Rel(\Gamma_2) + \Rel(\Delta)$ where $\Delta$
is a type A graph with respect to $(2  4) \ms{P}$.
\end{lemma}

\begin{proof}
The graph $\Gamma_2$ has a benzene chain with three (resp.\ four) blue edges,
beginning at vertex 1 and ending at vertex 4.  The picture is thus:
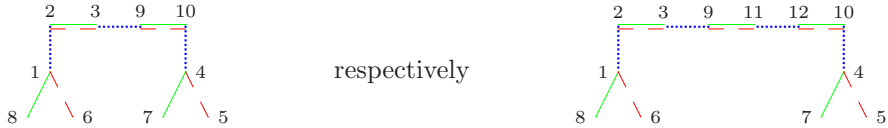
\begin{displaymath}
\begin{xy}
(-3, -6)*{}="A"; (3, -6)*{}="B"; (0, 0)*{}="C"; (0, 6)*{}="D"; (6, 6)*{}="E";
(12, 6)*{}="F"; (18, 6)*{}="G"; (18, 0)*{}="H"; (15, -6)*{}="I";
(21, -6)*{}="J";
(-5, -6)*{\ss 8}; (5, -6)*{\ss 6}; (-2, 0)*{\ss 1}; (0, 8)*{\ss 2};
(6, 8)*{\ss 3}; (12, 8)*{\ss 9}; (18, 8)*{\ss 10}; (20, 0)*{\ss 4};
(13, -6)*{\ss 7}; (23, -6)*{\ss 5};
{\ar@{-}@[green] "A"; "C"};
{\ar@{-}@<.2ex>@[green] "D"; "E"};
{\ar@{-}@<.2ex>@[green] "F"; "G"};
{\ar@{} "D"; "E"};
{\ar@{} "F"; "G"};
{\ar@{-}@[green] "I"; "H"};
{\ar@{--}@[red] "B"; "C"};
{\ar@{--}@<-.2ex>@[red] "D"; "E"};
{\ar@{--}@<-.2ex>@[red] "F"; "G"};
{\ar@{--}@[red] "J"; "H"};
{\ar@{.}@[|(2)]@[blue] "C"; "D"};
{\ar@{.}@[|(2)]@[blue] "E"; "F"};
{\ar@{.}@[|(2)]@[blue] "H"; "G"};
\end{xy}
\qquad\qquad \textrm{respectively} \qquad\qquad
\begin{xy}
(-3, -6)*{}="A"; (3, -6)*{}="B"; (0, 0)*{}="C"; (0, 6)*{}="D"; (6, 6)*{}="E";
(12, 6)*{}="F"; (18, 6)*{}="G"; (24, 6)*{}="H"; (30, 6)*{}="I";
(30, 0)*{}="J"; (27, -6)*{}="K"; (33, -6)*{}="L";
(-5, -6)*{\ss 8}; (5, -6)*{\ss 6}; (-2, 0)*{\ss 1}; (0, 8)*{\ss 2};
(6, 8)*{\ss 3}; (12, 8)*{\ss 9}; (18, 8)*{\ss 11}; (24, 8)*{\ss 12};
(30, 8)*{\ss 10}; (32, 0)*{\ss 4}; (25, -6)*{\ss 7}; (35, -6)*{\ss 5};
{\ar@{-}@[green] "A"; "C"};
{\ar@{-}@<.2ex>@[green] "D"; "E"};
{\ar@{-}@<.2ex>@[green] "F"; "G"};
{\ar@{-}@<.2ex>@[green] "H"; "I"};
{\ar@{} "D"; "E"};
{\ar@{} "F"; "G"};
{\ar@{} "H"; "I"};
{\ar@{-}@[green] "K"; "J"};
{\ar@{--}@[red] "B"; "C"};
{\ar@{--}@<-.2ex>@[red] "D"; "E"};
{\ar@{--}@<-.2ex>@[red] "F"; "G"};
{\ar@{--}@<-.2ex>@[red] "H"; "I"};
{\ar@{--}@[red] "L"; "J"};
{\ar@{.}@[|(2)]@[blue] "C"; "D"};
{\ar@{.}@[|(2)]@[blue] "E"; "F"};
{\ar@{.}@[|(2)]@[blue] "G"; "H"};
{\ar@{.}@[|(2)]@[blue] "J"; "I"};
\end{xy}
\end{displaymath}
Note that $U_R = \{7, 8 \}$ and $U_G = \{ 5, 6 \}$ while $U_B$ consists of the
remaining vertices.

We now Pl\"ucker the edges $\uedge{81}$ and $\uedge{74}$.  In the first term,
$8$ connects to $4$ and $7$ to $1$.  The second term is a degenerate Segre
datum (since $\uedge{81}$ and $\uedge{74}$ are the special green edges in
$U_B$).  It is thus quadratic and can be ignored.  We next Pl\"ucker the edges
$\uedge{61}$ and $\uedge{54}$, the results being similar.  We have thus shown
that $\Rel(\Gamma_2)$ defines the same element of $I'_L$ as the relation
associated to following graph (since we used two Pl\"ucker relations no
sign is introduced):
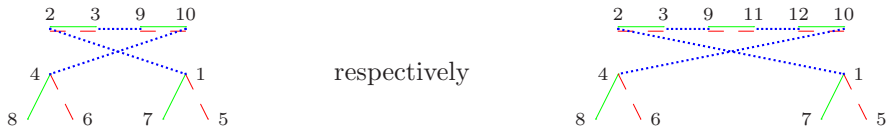
\begin{displaymath}
\begin{xy}
(-3, -6)*{}="A"; (3, -6)*{}="B"; (0, 0)*{}="C"; (0, 6)*{}="D"; (6, 6)*{}="E";
(12, 6)*{}="F"; (18, 6)*{}="G"; (18, 0)*{}="H"; (15, -6)*{}="I";
(21, -6)*{}="J";
(-5, -6)*{\ss 8}; (5, -6)*{\ss 6}; (-2, 0)*{\ss 4}; (0, 8)*{\ss 2};
(6, 8)*{\ss 3}; (12, 8)*{\ss 9}; (18, 8)*{\ss 10}; (20, 0)*{\ss 1};
(13, -6)*{\ss 7}; (23, -6)*{\ss 5};
{\ar@{-}@[green] "A"; "C"};
{\ar@{-}@<.2ex>@[green] "D"; "E"};
{\ar@{-}@<.2ex>@[green] "F"; "G"};
{\ar@{} "D"; "E"};
{\ar@{} "F"; "G"};
{\ar@{-}@[green] "I"; "H"};
{\ar@{--}@[red] "B"; "C"};
{\ar@{--}@<-.2ex>@[red] "D"; "E"};
{\ar@{--}@<-.2ex>@[red] "F"; "G"};
{\ar@{--}@[red] "J"; "H"};
{\ar@{.}@[|(2)]@[blue] "C"; "G"};
{\ar@{.}@[|(2)]@[blue] "E"; "F"};
{\ar@{.}@[|(2)]@[blue] "H"; "D"};
\end{xy}
\qquad\qquad \textrm{respectively} \qquad\qquad
\begin{xy}
(-3, -6)*{}="A"; (3, -6)*{}="B"; (0, 0)*{}="C"; (0, 6)*{}="D"; (6, 6)*{}="E";
(12, 6)*{}="F"; (18, 6)*{}="G"; (24, 6)*{}="H"; (30, 6)*{}="I";
(30, 0)*{}="J"; (27, -6)*{}="K"; (33, -6)*{}="L";
(-5, -6)*{\ss 8}; (5, -6)*{\ss 6}; (-2, 0)*{\ss 4}; (0, 8)*{\ss 2};
(6, 8)*{\ss 3}; (12, 8)*{\ss 9}; (18, 8)*{\ss 11}; (24, 8)*{\ss 12};
(30, 8)*{\ss 10}; (32, 0)*{\ss 1}; (25, -6)*{\ss 7}; (35, -6)*{\ss 5};
{\ar@{-}@[green] "A"; "C"};
{\ar@{-}@<.2ex>@[green] "D"; "E"};
{\ar@{-}@<.2ex>@[green] "F"; "G"};
{\ar@{-}@<.2ex>@[green] "H"; "I"};
{\ar@{} "D"; "E"};
{\ar@{} "F"; "G"};
{\ar@{} "H"; "I"};
{\ar@{-}@[green] "K"; "J"};
{\ar@{--}@[red] "B"; "C"};
{\ar@{--}@<-.2ex>@[red] "D"; "E"};
{\ar@{--}@<-.2ex>@[red] "F"; "G"};
{\ar@{--}@<-.2ex>@[red] "H"; "I"};
{\ar@{--}@[red] "L"; "J"};
{\ar@{.}@[|(2)]@[blue] "C"; "I"};
{\ar@{.}@[|(2)]@[blue] "E"; "F"};
{\ar@{.}@[|(2)]@[blue] "G"; "H"};
{\ar@{.}@[|(2)]@[blue] "J"; "D"};
\end{xy}
\end{displaymath}
We now Pl\"ucker $\uedge{1 2}$ and $\uedge{4 \; 10}$.  The first term is $(1 4)
\Gamma_2$.  The second term is a type A graph $\Delta$ with respect to $(2 4)
\ms{P}$.  We thus find $\Rel(\Gamma_2)=-\Rel((1  4) \Gamma_2) -\Rel(\Delta)$,
which establishes the proposition.
\end{proof}

\begin{lemma}
\label{lem-gam5}
We have $(1  4) \Rel(\Gamma_5)=\Rel(\Gamma_5)$.
\end{lemma}

\begin{proof}
We redraw $\Gamma_5$ in a more convenient way.
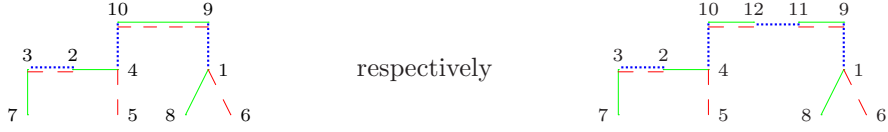
\begin{displaymath}
\begin{xy}
(0, -6)*{}="A"; (0, 0)*{}="B"; (6, 0)*{}="C"; (12, 0)*{}="D"; (12, -6)*{}="E";
(12, 6)*{}="F"; (24, 6)*{}="I"; (24, 0)*{}="J"; (21, -6)*{}="K";
(27, -6)*{}="L";
(-2, -6)*{\ss 7}; (0, 2)*{\ss 3}; (6, 2)*{\ss 2}; (14, 0)*{\ss 4};
(14, -6)*{\ss 5}; (12, 8)*{\ss 10}; (24, 8)*{\ss 9};
(26, 0)*{\ss 1}; (19, -6)*{\ss 8}; (29, -6)*{\ss 6};
{\ar@{-}@[green] "A"; "B"};
{\ar@{-}@[green] "C"; "D"};
{\ar@{-}@<.2ex>@[green] "F"; "I"};
{\ar@{} "F"; "I"};
{\ar@{-}@[green] "K"; "J"};
{\ar@{--}@<-.2ex>@[red] "B"; "C"};
{\ar@{} "B"; "C"};
{\ar@{--}@[red] "E"; "D"};
{\ar@{--}@<-.2ex>@[red] "F"; "I"};
{\ar@{--}@[red] "L"; "J"};
{\ar@{.}@[|(2)]@<.2ex>@[blue] "B"; "C"};
{\ar@{.}@[|(2)]@[blue] "D"; "F"};
{\ar@{.}@[|(2)]@[blue] "J"; "I"};
\end{xy}
\qquad\qquad \textrm{respectively} \qquad\qquad
\begin{xy}
(0, -6)*{}="A"; (0, 0)*{}="B"; (6, 0)*{}="C"; (12, 0)*{}="D"; (12, -6)*{}="E";
(12, 6)*{}="F"; (18, 6)*{}="G"; (24, 6)*{}="H"; (30, 6)*{}="I";
(30, 0)*{}="J"; (27, -6)*{}="K"; (33, -6)*{}="L";
(-2, -6)*{\ss 7}; (0, 2)*{\ss 3}; (6, 2)*{\ss 2}; (14, 0)*{\ss 4};
(14, -6)*{\ss 5}; (12, 8)*{\ss 10}; (18, 8)*{\ss 12}; (24, 8)*{\ss 11};
(30, 8)*{\ss 9}; (32, 0)*{\ss 1}; (25, -6)*{\ss 8}; (35, -6)*{\ss 6};
{\ar@{-}@[green] "A"; "B"};
{\ar@{-}@[green] "C"; "D"};
{\ar@{-}@<.2ex>@[green] "F"; "G"};
{\ar@{-}@<.2ex>@[green] "H"; "I"};
{\ar@{} "F"; "G"};
{\ar@{} "H"; "I"};
{\ar@{-}@[green] "K"; "J"};
{\ar@{--}@<-.2ex>@[red] "B"; "C"};
{\ar@{} "B"; "C"};
{\ar@{--}@[red] "E"; "D"};
{\ar@{--}@<-.2ex>@[red] "F"; "G"};
{\ar@{--}@<-.2ex>@[red] "H"; "I"};
{\ar@{--}@[red] "L"; "J"};
{\ar@{.}@[|(2)]@<.2ex>@[blue] "B"; "C"};
{\ar@{.}@[|(2)]@[blue] "D"; "F"};
{\ar@{.}@[|(2)]@[blue] "G"; "H"};
{\ar@{.}@[|(2)]@[blue] "J"; "I"};
\end{xy}
\end{displaymath}
As before, $U_R=\{7, 8\}$ and $U_G=\{5, 6\}$ while $U_B$ consists of the
remaining vertices.

We now Pl\"ucker the edges $\uedge{6 1}$ and $\uedge{5 4}$.  In the first term,
5 connects to 1 and 6 to 4.  The second term is a degenerate Segre datum and
can be ignored.  We next Pl\"ucker $\uedge{2 4}$ with $\uedge{8 1}$.  In the
first term, 8 connects to 4 and 2 to 1.  The second term is degenerate and can
be ignored.  We have thus shown that $\Rel(\Gamma_5)$ defines the same
element of $I'_L$ as the relation associated to the following graph:
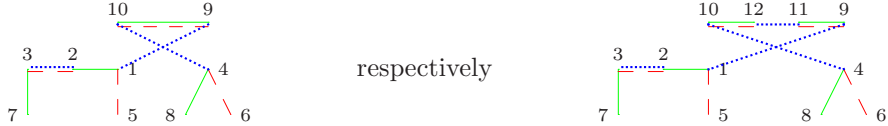
\begin{displaymath}
\begin{xy}
(0, -6)*{}="A"; (0, 0)*{}="B"; (6, 0)*{}="C"; (12, 0)*{}="D"; (12, -6)*{}="E";
(12, 6)*{}="F"; (24, 6)*{}="I"; (24, 0)*{}="J"; (21, -6)*{}="K";
(27, -6)*{}="L";
(-2, -6)*{\ss 7}; (0, 2)*{\ss 3}; (6, 2)*{\ss 2}; (14, 0)*{\ss 1};
(14, -6)*{\ss 5}; (12, 8)*{\ss 10}; (24, 8)*{\ss 9};
(26, 0)*{\ss 4}; (19, -6)*{\ss 8}; (29, -6)*{\ss 6};
{\ar@{-}@[green] "A"; "B"};
{\ar@{-}@[green] "C"; "D"};
{\ar@{-}@<.2ex>@[green] "F"; "I"};
{\ar@{} "F"; "I"};
{\ar@{-}@[green] "K"; "J"};
{\ar@{--}@<-.2ex>@[red] "B"; "C"};
{\ar@{} "B"; "C"};
{\ar@{--}@[red] "E"; "D"};
{\ar@{--}@<-.2ex>@[red] "F"; "I"};
{\ar@{--}@[red] "L"; "J"};
{\ar@{.}@[|(2)]@<.2ex>@[blue] "B"; "C"};
{\ar@{.}@[|(2)]@[blue] "D"; "I"};
{\ar@{.}@[|(2)]@[blue] "J"; "F"};
\end{xy}
\qquad\qquad \textrm{respectively} \qquad\qquad
\begin{xy}
(0, -6)*{}="A"; (0, 0)*{}="B"; (6, 0)*{}="C"; (12, 0)*{}="D"; (12, -6)*{}="E";
(12, 6)*{}="F"; (18, 6)*{}="G"; (24, 6)*{}="H"; (30, 6)*{}="I";
(30, 0)*{}="J"; (27, -6)*{}="K"; (33, -6)*{}="L";
(-2, -6)*{\ss 7}; (0, 2)*{\ss 3}; (6, 2)*{\ss 2}; (14, 0)*{\ss 1};
(14, -6)*{\ss 5}; (12, 8)*{\ss 10}; (18, 8)*{\ss 12}; (24, 8)*{\ss 11};
(30, 8)*{\ss 9}; (32, 0)*{\ss 4}; (25, -6)*{\ss 8}; (35, -6)*{\ss 6};
{\ar@{-}@[green] "A"; "B"};
{\ar@{-}@[green] "C"; "D"};
{\ar@{-}@<.2ex>@[green] "F"; "G"};
{\ar@{-}@<.2ex>@[green] "H"; "I"};
{\ar@{} "F"; "G"};
{\ar@{} "H"; "I"};
{\ar@{-}@[green] "K"; "J"};
{\ar@{--}@<-.2ex>@[red] "B"; "C"};
{\ar@{} "B"; "C"};
{\ar@{--}@[red] "E"; "D"};
{\ar@{--}@<-.2ex>@[red] "F"; "G"};
{\ar@{--}@<-.2ex>@[red] "H"; "I"};
{\ar@{--}@[red] "L"; "J"};
{\ar@{.}@[|(2)]@<.2ex>@[blue] "B"; "C"};
{\ar@{.}@[|(2)]@[blue] "D"; "I"};
{\ar@{.}@[|(2)]@[blue] "G"; "H"};
{\ar@{.}@[|(2)]@[blue] "J"; "F"};
\end{xy}
\end{displaymath}
We now Pl\"ucker the edges $\uedge{1 9}$ and $\uedge{4 \; 10}$.  The first term
is just $(1  4) \Gamma_5$.  The second term is retrogenerated from $8$
(resp.\ $10$) points and therefore does not contribute in $I'_L$.  We thus
have $\Rel(\Gamma_5)= -\Rel((1  4)\Gamma_5)$, which establishes the
proposition.
\end{proof}

We can now complete the proof of Proposition~\ref{prop-B-rel}.

\begin{proof}[Proof of Proposition~\ref{prop-B-rel}]
By Lemma~\ref{lem-gam5} and \eqref{eq-gam}, we see that $\Rel(\Gamma) -
\Rel((2  3) \Gamma) + \Rel(\Gamma_2)$ is invariant under $(1  4)$, that is,
\begin{displaymath}
\Rel(\Gamma)-\Rel((2  3) \Gamma)+\Rel(\Gamma_2)=
(1  4) \Rel(\Gamma)-(1  4)\Rel((2  3) \Gamma)+(1  4) \Rel(\Gamma_2).
\end{displaymath}
We now apply Lemma~\ref{lem-gam2} and write $(1 4)\Rel(\Gamma_2)=\Rel(\Gamma_2)
+\Rel(\Delta)$ where $\Delta$ is a type A graph with respect to $(2  4)
\ms{P}$.  The $\Gamma_2$ terms on each side of the equation cancel, and we are
left with
\begin{displaymath}
\Rel(\Gamma)= (1  4) \Rel(\Gamma) -(2  3)\Rel(\Gamma)
+(1  4)(2  3) \Rel(\Gamma)+\Rel(\Delta).
\end{displaymath}
This completes the proof.
\end{proof}

\subsection{Proof of Theorem~\ref{thm:base}}

We now complete the proof of Theorem~\ref{thm:base}.  As we have explained, to
do this it suffices to show that $I'_L=0$.  To do this we first use
Proposition~\ref{prop-A-rel} to obtain an upper bound for the dimension of
$I'_L$, and then we use representation theory to show that this upper bound
forces $I'_L$ to be zero.

Fix an order on $L$.  For $n=10$ (resp.\ $n=12$), we say that a bipartition
$\ms{P}=(P, Q)$ of $L$ is \emph{good} if there is at most one (resp.\ if there
are at most two) element(s) of $Q$ which are larger than the second smallest
element of $P$.  
\Ncom{$\text{good bipartition}$}

\begin{proposition}
\label{prop-good-surj}
The natural map $\bigoplus I'_{\ms{P}} \to I'$ is surjective, where the sum
is taken over all good bipartitions $\ms{P}$.
\end{proposition}

\begin{proof}
We handle the two cases $n=10$ and $12$ separately, for
the sake of clarity.  We begin with the 10 point case.  Let $\ms{P}$ be a
bipartition which is not good.  Let $a<b$ be the smallest two elements of $P$
and let $c<d$ be the largest two elements of $Q$.  Since $\ms{P}$ is not good,
we have $a<b<c<d$.  Now let $\Gamma$ be a type A graph with respect to
$\ms{P}$.  Since $I_P^{(3)} \otimes \gr_{4, Q}^{(3)}$ is one-dimensional, if
$\Gamma'$ is any other type A graph with respect to $\ms{P}$ then we have
$\Rel(\Gamma)=\Rel(\Gamma')$.  Thus we may assume that $c$ and $d$ are
connected by a doubled edge in $\Gamma$.  We may then apply
Proposition~\ref{prop-A-rel} to conclude that $\Gamma$ belongs to $\sum
I'_{\sigma \ms{P}}$, the sum taken over those permutations $\sigma$ of
$\{a,b,c,d\}$ for which $\sigma \ms{P} \ne \ms{P}$.  Each of these $\sigma
\ms{P}$ is closer to being good than $\ms{P}$ (measured by how many
elements of $Q$ are larger than the second smallest element of $P$, for
instance).  By induction, we deduce the proposition.

We now handle the case where $L$ has cardinality 12.  Let $\ms{P}$ be a
bipartition which is not good.  Let $a<b$ be the smallest two elements of $P$
and let $c<d<e$ be the largest three elements of $Q$.  Since $\ms{P}$ is not
good, we have $a<b<c<d<e$.  Now let $\Gamma$ be a type A graph with respect
to $\ms{P}$.  By Proposition~\ref{prop:g6}(d)  we can rewrite $\Rel(\Gamma)$ in
terms of $\Rel(\Gamma_i)$, where each $\Gamma_i$ is a type A graph with respect
to $\ms{P}$ in which either $\uedge{cd}$ or $\uedge{ce}$ appears as a doubled
edge.  We may as well then just assume that $\Gamma$ has $\uedge{c d}$ as a
doubled edge (the $\uedge{c e}$ argument is the same).  We now apply
Proposition~\ref{prop-A-rel} and find that $\Rel(\Gamma)$ belongs to $\sum
I'_{\sigma \ms{P}}$ as $\sigma$ varies over those permutations of $\{a, b, c,
d\}$ for which $\sigma \ms{P} \ne \ms{P}$.  As before, each $\sigma \ms{P}$ is
closer to being good than $\ms{P}$ and so we deduce the proposition by
induction.
\end{proof}

\begin{proposition}
For $n=10$ (resp.\ $n=12$) there are 25 (resp.\ 112) good bipartitions of $L$.
\end{proposition}

\begin{proof}
We again consider the two cases separately, and begin with the 10 point case.
Identify $L$ with $\{1, 2, \ldots, 10\}$.  Consider bipartitions of $L$ for
which the second smallest element of $P$ is $x$.  Of course, we must have
$x \le 6$.  On the other hand, if $x \le 4$ then at least two elements of $Q$
will be larger than $x$ and so the bipartition will not be good.  Thus we must
have $x=5$ or $x=6$.
\begin{itemize}
\item $x=6$.  In this case, $P$ must contain each of 7, 8, 9 and 10.  In
addition, $P$ contains one number less than 6.  There are five such choices,
each of which gives a good bipartition.
\item $x=5$.  In this case, $P$ must contain each of 6, 7, 8, 9 and 10 except
for one (it cannot contain them all for then $x$ would not be second smallest).
There are five ways to choose what to omit.  In addition, $P$ must contain one
number less than 5.  There are four such choices.  Thus, all told, there are
20 good bipartitions.
\end{itemize}
Thus in total we have $5+20=25$ good bipartitions.

We now handle the case where $L$ has cardinality 12.  Again, we identity
$L$ with $\{1, \ldots, 12\}$ and consider bipartitions for which the second
smallest element of $P$ is $x$.  We must have $x \le 8$.  Now, if $x \le 5$
then $Q$ necessarily has three elements larger than $x$ and so the bipartition
is not good.  Thus we only need to consider the cases $x=6,7,8$.
\begin{itemize}
\item $x=8$.  In this case, $P$ must contain each of 9, 10, 11 and 12.
In addition, $P$ contains one number less than 8.  There are seven choices
for such a number, and all give good bipartitions.
\item $x=7$.  In this case, $P$ must contain each of 8, 9, 10, 11 and 12
except for one (it cannot contain them all for then $x$ would not be
second smallest).  There are five ways to choose what to omit.  Furthermore,
$P$ must contain one element smaller than 7.  There are six such choices.
Thus there are 30 good bipartitions in this case.
\item $x=6$.  In this case, $P$ must contain each of 7, 8, 9, 10, 11 and 12
except for two.  There are 15 ways to omit two of these numbers.  Furthermore,
$P$ must contain one element smaller than 6.  There are five such choices.
Thus there are 75 good bipartitions in this case.
\end{itemize}
Thus in total we have $7+30+75=112$ good bipartitions.
\end{proof}

\begin{corollary}
For $n=10$ (resp.\ $n=12$) we have $\dim{I'_L} \le 25$ (resp.\ $\le 560$).
\end{corollary}

\begin{proof}
We know that $I'_L$ admits a surjection from $\bigoplus I'_{\ms{P}}$, the
sum taken over the good bipartitions $\ms{P}$.  When $L$ has cardinality 10,
there are 25 of these bipartitions and each $I'_{\ms{P}}$ has dimension at most
one.  When $L$ has cardinality 12, there are 112 of these bipartitions and each
$I'_{\ms{P}}$ has dimension at most five.  This gives the result.
\end{proof}

\begin{proposition}
For $n=10$ (resp.\ $n=12$) the space $\bigoplus I'_{\ms{P}}$ (summed over all
bipartitions $\ms{P}$) is either zero a direct sum of two irreducible
$\mf{S}_L$-representations of dimensions 84 and 126 (resp.\ three irreducible
representations of dimensions 616, 1925 and 2079).
\end{proposition}

\begin{proof}
The space $\bigoplus I'_{\ms{P}}$ is isomorphic as an $\mf{S}_L$-module to
$\Ind_{\mf{S}_P \times \mf{S}_Q}^{\mf{S}_L} I'_{\ms{P}}$ for any fixed
bipartition $\ms{P}=(P, Q)$.  We now handle the two cases separately.

First suppose $n=10$.  Then $I'_{\ms{P}}$ is either zero or isomorphic as an
$(\mf{S}_P \times \mf{S}_Q)$-module to $I_P^{(3)} \otimes \gr_{4, Q}^{(3)}$.
As an $\mf{S}_P$-module $I_P^{(3)}$ corresponds to the partition $1+1+1+1+1+1$ 
(the alternating representation) while as an $\mf{S}_Q$-module
$\gr_{4, Q}^{(3)}$ corresponds to the partition 4 (the trivial representation).
We now use the Littlewood-Richardson rule to compute $\Ind_{\mf{S}_P \times
\mf{S}_Q}^{\mf{S}_L} I'_{\ms{P}}$.  Excluding the case where $I'_{\ms{P}}=0$,
we find that the induction decomposes into a direct sum of two irreducibles,
corresponding to the partitions
\begin{displaymath}
5+1+1+1+1+1, \qquad 4+1+1+1+1+1+1.
\end{displaymath}
The hook length formula shows that these irreducibles have dimensions 126 and
84, respectively.

Now say $n=12$.  Again, $I'_{\ms{P}}$ is either zero or isomorphic as an
$(\mf{S}_P \times \mf{S}_Q)$-module to $I_P^{(3)} \otimes \gr_{6, Q}^{(3)}$.
The representation $I_P^{(3)}$ corresponds to the partition $1+1+1+1+1+1$ while
$\gr_{6, Q}^{(3)}$ corresponds to $3+3$.  The Littlewood-Richardson rule shows
that the induction decomposes into three irreducibles, corresponding to the
partitions
\begin{displaymath}
4+3+1+1+1+1+1, \qquad 4+4+1+1+1+1, \qquad 3+3+1+1+1+1+1+1.
\end{displaymath}
These irreducibles have dimensions 2079, 1925 and 616, respectively.
\end{proof}

We now complete the proof of Theorem~\ref{thm:base}.

\begin{proof}[Proof of Theorem~\ref{thm:base}]
We must show $I'_L=0$.  If $n=10$, we know that on one hand, $I'_L$ is at
most 25-dimensional, while on the other, it is a direct sum of at most two
irreducibles of dimensions 84 and 126.  It follows that $I'_L$ must be zero.
Similarly, when $n=12$, we know that, on one hand, $I'_L$ is
at most 560-dimensional, while on the other, it is a direct sum of at most
three irreducibles of dimensions 616, 1925 and 2079.  Again, we conclude
$I'_L=0$.  
\end{proof}


\section{The quadratics are generated by the simplest binomials}
\label{s:quad}

In \S \ref{s:quad}, we  prove the following, concluding the proof of
Theorem~\ref{mainthm}:
\Ncom{$\text{simplest/simple bin.\ rels.}$}

\begin{theorem}
\label{thm:quad}
If $L$ has even cardinality $n$, then the simplest binomial relations span
$I_L^{(2)}$ over $\Z[\tfrac{1}{n!}]$.
\end{theorem}

The proof uses the \emph{simple binomial relations}.  We show that the spans of
the simple binomial relations and simplest binomial relations coincide.  We
then use the $\mf{S}_L$-module structure of $I_L^{(2)}$ to show that simple
binomial relations generate it.  It suffices to prove the result over any
field of characteristic $0$ or greater than $n$, so we do this.

\subsection{Simple binomial relations}

Let $L$ be an even set.  A \emph{binomial quadratic datum} is a pair
$D=(\Gamma, U)$ where $\Gamma$ is an undirected regular 2-colored graph on $L$
and $U$ is a subset of $L$ such that all edges of $\Gamma$ are  contained in
$U$ or $L \setminus U$.  Define $\Gamma'$ to be the graph obtained by inverting
the colors of the edges of $\Gamma$ contained in $U$, and set 
\begin{displaymath}
\Rel(D)=Y_{\Gamma}-Y_{\Gamma'},
\end{displaymath}
which is clearly a relation.
\Ncom{$\text{binom.\ quad. datum}$}

Let $D$ be a binomial quadratic datum.  Suppose that a pair of edges $e$ and
$e'$ of $\Gamma$  have the same color and both lie in $U$ or $L \setminus U$.
Let $\Gamma'$ and $\Gamma''$ be the other graphs occurring in the colored
Pl\"ucker relation on $e$ and $e'$.  Then $D'=(\Gamma', U)$ and $D''=(\Gamma'',
U)$ are binomial quadratic data.  We define the \emph{space of binomial
quadratic data} to be the $\Z$-span of the binomial quadratic data modulo the
(Pl\"ucker) relations $D+D'+D''=0$.  The association $D \mapsto \Rel(D)$
descends to a linear map
\begin{displaymath}
\Rel:\{ \textrm{the space of binomial quadratic data} \} \to I_L^{(2)}.
\end{displaymath}

We call a binomial quadratic datum \emph{simple} if $U$ has cardinality four.
We call the resulting relations \emph{simple binomial relations}.  We say
that a binomial quadratic datum is \emph{simplest} if it is simple and
in addition $\Gamma$ is made up of 2-cycles and two 4-cycles.  The
associated relations are the \emph{simplest binomial relations} defined in
the introduction; \eqref{e:simplest} is an example.
Note that if $D=(U, \Gamma)$ is a binomial quadratic datum and $\Gamma$ is a
union of 2-cycles and zero or one 4-cycle then $\Rel(D)=0$.  Although there are
more simple binomial relations than simplest binomial relations, they span the
same space: 
\Ncom{$\text{simple bin q datum, simple bin rels, simplest bin q datum, simplest bin rels}$}

\begin{proposition}
\label{prop:simp-binom}
Every simple binomial relation is a linear combination of simplest binomial
relations.
\end{proposition}

Thus to prove Theorem~\ref{thm:quad} it suffices to show that the
simple binomial relations span $I_L^{(2)}$.
(We only need characteristic not $2$ to prove this.)

\begin{proof}
Let $D=(\Gamma, U)$ be a simple binomial quadratic datum.  By
Proposition~\ref{prop:pfil2} we can use colored Pl\"ucker relations
in $L \setminus U$ to write $\Gamma \vert_{L \setminus U}$ as a sum of graphs,
each of which is a union of 2-cycles and at most one 4-cycle.  This expresses
$\Rel(D)$ in terms of simplest binomial relations.
\end{proof}

Let $L$ be an even set and let $U$ be a subset of cardinality four.  Let $L'
= L \setminus U$.  Define a map
\begin{displaymath}
\iota:\bw{V_{L'}} \otimes \bw{V_U} \to I_L^{(2)}
\quad \quad \textrm{by} \quad \quad 
(Y_{\Gamma} \wedge Y_{\Gamma'}) \otimes (Y_{\Delta} \wedge Y_{\Delta'})
\mapsto
Y_{\Gamma \Delta} Y_{\Gamma' \Delta'}-Y_{\Gamma' \Delta} Y_{\Gamma \Delta'}.
\end{displaymath}
This is  the simple binomial relation associated to the
the simple binomial data $(\Phi, U)$ where $\Phi$ is the 2-colored graph
$\Gamma \Gamma' \Delta \Delta'$ in which $\Gamma \Delta$ has one color and
$\Gamma' \Delta'$ the other.
Note that $\iota$ is $(\mf{S}_{L'} \times \mf{S}_U)$-equivariant.

\begin{lemma}
\label{lem:iota}
The map $\iota$ is injective.
\end{lemma}

\begin{proof}
Let $Y_{\Delta}$, $Y_{\Delta'}$ be a basis of $V_U$.  It suffices to show that
the map
$\bw{V_{L'}} \to \Sym^2(R_L^{(1)})$ given by 
\begin{displaymath}
Y_{\Gamma} \wedge Y_{\Gamma'} \mapsto
Y_{\Gamma \Delta} Y_{\Gamma' \Delta'}-Y_{\Gamma' \Delta} Y_{\Gamma \Delta'}
\end{displaymath}
is injective.  Let $f$ (resp.\ $f'$) be the map $V_{L'} \to R_L^{(1)}$ given by
$Y_{\Gamma} \mapsto Y_{\Gamma \Delta}$ (resp.\ $Y_{\Gamma} \mapsto
Y_{\Gamma \Delta'}$).  Both
$f$ and $f'$ are injective.  Furthermore, the images of $f$ and $f'$ are
linearly disjoint.  To see this, use the fact that planar graphs
form a basis for $R_L^{(1)}$ (Theorem~\ref{thm:kempe2}):  Put $L$ on a circle
so that the vertices in $U$
are consecutive.  If $\Gamma$ and $\Gamma'$ are planar graphs on $L'$
and $\Delta$ and $\Delta'$ are distinct planar graphs on $U$ then $\Gamma
\Delta$ and $\Gamma' \Delta'$ are distinct planar graphs on $L$.  The
lemma now follows from the following lemma in linear algebra.
\end{proof}

\begin{lemma}
Let $f, g:V \to W$ be linear maps of vector spaces which are injective and
have linearly disjoint images.  Then the map
\begin{displaymath}
\bw{V} \to \Sym^2(W), \qquad v \wedge w \mapsto f(v) g(w) - f(w) g(v)
\end{displaymath}
is also injective.
\end{lemma}

\begin{proof}
Let $v_i$ be a basis for $V$ and assume that $v=\sum \alpha_{ij} v_i \wedge
v_j$ belongs to the kernel, so
\begin{displaymath}
0=\sum \alpha_{ij} (f(v_i) g(v_j)-f(v_j) g(v_i))=\sum (\alpha_{ij} -
\alpha_{ji}) f(v_i) g(v_j).
\end{displaymath}
The vectors $f(v_i) g(v_j)$ are linearly independent in $\Sym^2(W)$, so
 $\alpha_{ij}=\alpha_{ji}$ and hence $v=0$.
\end{proof}

\subsection{Completion of the proof of Theorem~\ref{thm:quad}}

Let $X$ denote the set of partitions of $n$ into exactly four even parts and
let $X^-$ be the set of partitions of $n-4$ into exactly four odd parts.
Thus $X$ has the partitions occurring in $I_L$ and $X^-$ has those occurring
in $\bw{V_{L'}}$ (Proposition  \ref{prop:pfil2rep}).  For a partition $\lambda
\in X$ we define a corresponding partition $\lambda^- \in X^-$ by removing one
box from each row in the Young diagram:
$(a,b,c,d) \mapsto (a-1, b-1, c-1, d-1)$.   This bijection
of sets $X \rightarrow X^-$
has a representation-theoretic characterization:
\begin{lemma}
Let $\lambda^- \in X^-$ and $\mu \in X$ be given. 
(Recall from \S \ref{action} that $M_{\lambda}$ is 
the irreducible representation of $\mf{S}_L$ corresponding to
partition $\lambda$.) Then
\begin{displaymath}
\dim \Hom_{\mf{S}_{L'} \times \mf{S}_U}(M_{\lambda^-} \otimes \epsilon, M_{\mu})
=\begin{cases}
1 & \textrm{if $\lambda=\mu$} \\
0 & \textrm{otherwise}
\end{cases}
\end{displaymath}
Here $\epsilon$ denotes the sign representation of $\mf{S}_U$.
\end{lemma}

\begin{proof}
By Frobenius reciprocity, the dimension of the $\Hom$ space is equal to the
multiplicity of $M_{\mu}$ occurring in the induction $\Ind_{\mf{S}_{L'} \times
\mf{S}_U}^{\mf{S}_L} \left( M_{\lambda^-} \otimes \epsilon \right)$.  This can
be computed using the Littlewood-Richardson rule,  which is simple
in this case because $\epsilon$ is just the alternating representation (the
Young diagram is a single column of four boxes).  The key point is that
the only way to add on four boxes to $\lambda$ and end up with something
in $X$ is to put one box at the end of each row.
\end{proof}

\begin{corollary}
The map $\iota$ carries $\left( \bw{V_{L'}} \right)[\lambda^-] \otimes
\bw{V_U}$ into $I_L^{(2)}[\lambda]$ for any $\lambda \in X$,
where as usual $V[\lambda]$ is the $\lambda$-isotypic part of
a representation $V$ of the symmetric group.  \label{lastcor}
\end{corollary}

We now can prove Theorem~\ref{thm:quad}.

\begin{proof}[Proof of Theorem~\ref{thm:quad}]
According to the above corollary, $\iota$ induces a map
\begin{displaymath}
\left( \bw{V_{L'}} \right)[\lambda^-] \otimes \bw{V_U} \to
I_L^{(2)}[\lambda]
\end{displaymath}
for any $\lambda \in X$.  This map is injective (Lemma~\ref{lem:iota}), and
$\left( \bw{V_{L'}} \right)[\lambda^-]$ is non-zero
(Proposition \ref{prop:pfil2rep}).  Thus
the image of $\iota$ has non-zero projection to each $I_L^{(2)}
[\lambda]$.  Since $I_L^{(2)}$ is multiplicity free
(Proposition \ref{prop:pfil2rep}), 
the image of $\iota$ generates $I_L^{(2)}$ as an $\mf{S}_L$-module.  Hence
the simple binomial relations span $I_L^{(2)}$.  Since every
simple binomial relation is a linear combination of simplest binomial
relations (Proposition \ref{prop:simp-binom}), we deduce
Theorem~\ref{thm:quad}.
\end{proof}


\end{document}